\documentclass[a4paper,12pt]{amsart}

\usepackage[utf8]{inputenc}
\usepackage[english]{babel}
\usepackage{graphicx}
\usepackage{ae}
\usepackage{aecompl}
\usepackage[a4paper]{geometry}
\usepackage{setspace}
\usepackage{cite}
\usepackage{amsmath}
\usepackage{amssymb}
\usepackage{amsthm}
\usepackage{csquotes}
\usepackage{enumitem}
\usepackage{tikz}
\usetikzlibrary{positioning,shapes,calc}
\usetikzlibrary{decorations.markings,arrows.meta}
\tikzset{negated/.style={
        decoration={markings,
            mark= at position 0.5 with {
                \node[transform shape] (tempnode) {\small\slash};
            }
        },
        postaction={decorate}
    }
}
\usepackage{aliascnt}
\usepackage{graphics}
\usepackage[unicode,bookmarksopen]{hyperref}
\newcommand{\arxivlink}[1]{\href{https://arxiv.org/abs/#1}{\texttt{arXiv:#1}}}

\numberwithin{equation}{subsection}

\newtheorem{theorem}{Theorem}[subsection]

\newcommand{\MakeTheoremAndCounter}[2]{\newaliascnt{#1}{theorem}
\newtheorem{#1}[#1]{#2}
\aliascntresetthe{#1}
\expandafter\providecommand\csname#1autorefname\endcsname{#2}}

\newcommand\nobracketcite[2]{\expandafter\cite[#1]{#2}}

\MakeTheoremAndCounter{lemma}{Lemma}
\MakeTheoremAndCounter{claim}{Claim}
\MakeTheoremAndCounter{example}{Example}
\MakeTheoremAndCounter{fact}{Fact}
\MakeTheoremAndCounter{question}{Question}
\MakeTheoremAndCounter{corollary}{Corollary}
\MakeTheoremAndCounter{notation}{Notation}
\MakeTheoremAndCounter{problem}{Problem}
\MakeTheoremAndCounter{proposition}{Proposition}
\MakeTheoremAndCounter{conjecture}{Conjecture}

\theoremstyle{definition}\MakeTheoremAndCounter{remark}{Remark}
\theoremstyle{definition}\MakeTheoremAndCounter{definition}{Definition}

\def\restatedname{no theorem name}
\theoremstyle{plain}
\newtheorem*{restatethm}{\restatedname}
\newenvironment{restate}[1]{\expandafter\def\expandafter\restatedname{\autoref{#1}}\begin{restatethm}}{\end{restatethm}}

\newcommand{\autorefre}[1]{\hyperref[re:#1]{\autoref*{#1}}}

\newenvironment{multistmt}{~\\\vspace{-4ex}
\begin{enumerate}[label={\upshape(\arabic*)}, noitemsep]}{\end{enumerate}}

\newenvironment{propertylist}{
\begin{enumerate}[label={\upshape(\Roman*)}, noitemsep, nolistsep]}{\end{enumerate}}

\DeclareMathOperator{\diam}{diam}
\DeclareMathOperator{\supp}{supp}
\DeclareMathOperator{\intby}{~d}
\DeclareMathOperator*{\Int}{int}

\DeclareMathOperator{\concat}{{}^{\smallfrown}}

\allowdisplaybreaks 
\newcommand{\titleofthepaper}{Haar null and Haar meager sets: a survey and new results}
\hypersetup{pdftitle=\titleofthepaper}
\begin{document}
\title{\titleofthepaper}

\author{M\'arton Elekes}
\address{Alfr\'ed R\'enyi Institute of Mathematics, Hungarian Academy of Sciences,
PO Box 127, 1364 Budapest, Hungary and E\"otv\"os Lor\'and
University, Institute of Mathematics, P\'azm\'any P\'eter s. 1/c,
1117 Budapest, Hungary}
\email{elekes.marton@renyi.mta.hu}
\urladdr{http://www.renyi.hu/$\sim$emarci}

\author{Don\'at Nagy}
\address{E\"otv\"os Lor\'and University, Institute of Mathematics, P\'azm\'any P\'eter
s. 1/c, 1117 Budapest, Hungary}
\email{nagdon@bolyai.elte.hu}

\thanks{The first author was supported by the National Research, Development and Innovation Office -- NKFIH, grants no.~113047, 104178 and 124749. The second author was supported by the National Research, Development and Innovation Office -- NKFIH, grants no.~104178 and 124749.\\
\includegraphics[height=1cm]{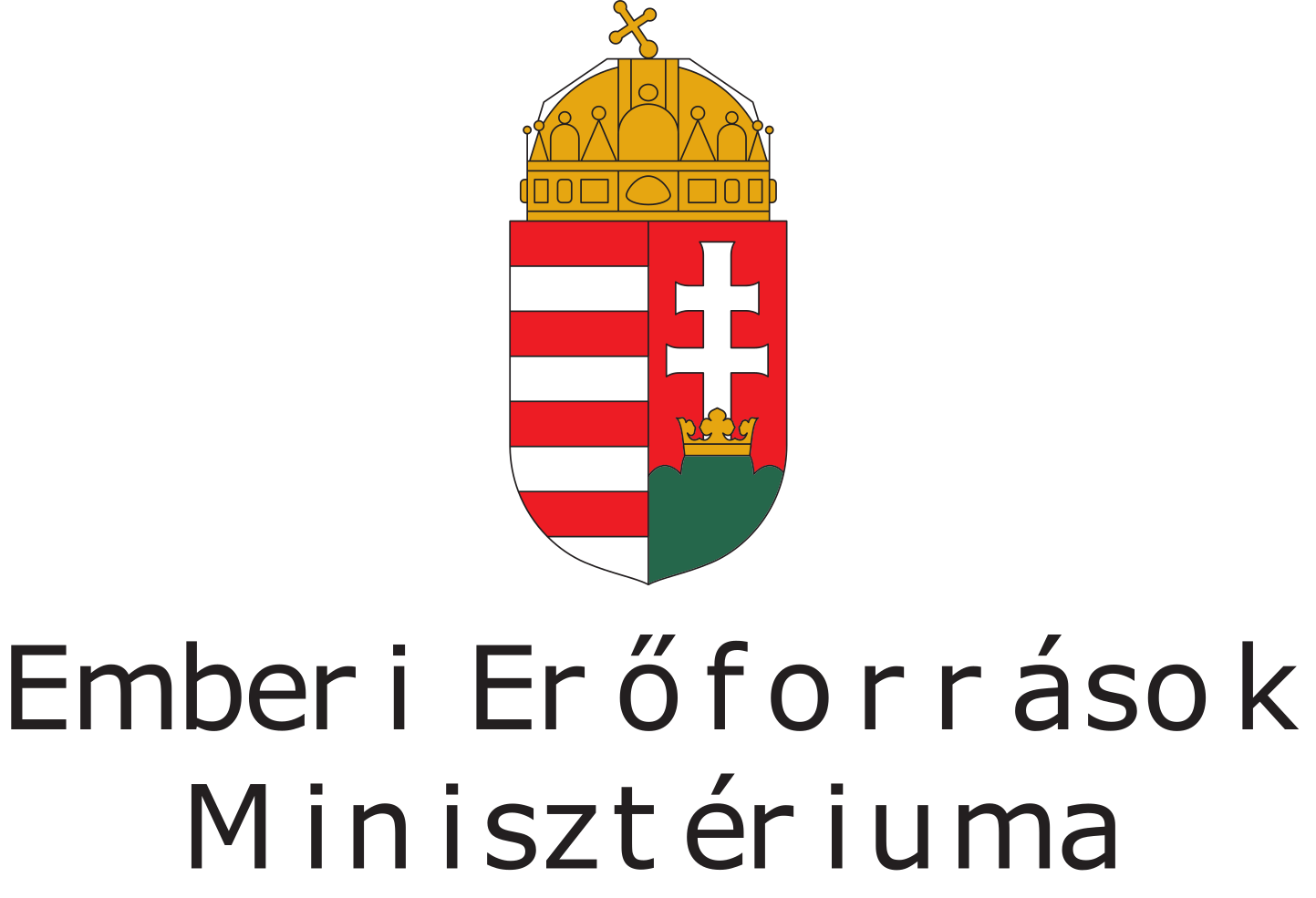} \raisebox{0.5cm}{\parbox[c]{11cm}{Supported through the New National Excellence Program of the Ministry of Human Capacities.}}}
\begin{abstract}
We survey results about Haar null subsets of (not necessarily locally compact) Polish groups. The aim of this paper is to collect the fundamental properties of the various possible definitions of Haar null sets, and also to review the techniques that may enable the reader to prove results in this area. We also present several recently introduced ideas, including the notion of Haar meager sets, which are closely analogous to Haar null sets. We prove some results in a more general setting than that of the papers where they were originally proved and prove some results for Haar meager sets which were already known for Haar null sets.
\end{abstract}

\maketitle

\tableofcontents

\section{Introduction and history}

Many results in various branches of mathematics state that certain properties hold for \emph{almost every} element of a space. In the continuous, large structures which are frequently studied in analysis it is common to encounter a property that is true for most points, but false on a negligibly small part of the structure. These situations mean that there are facts which can be grasped only by defining a suitable notion of smallness and stating that the exceptional elements form a small set.

In the Euclidean space $\mathbb{R}^n$ there is a generally accepted, natural notion of smallness: a set is considered to be small if it has Lebesgue measure zero. The Haar measure (which was introduced by Alfr\' ed Haar in 1933) generalizes the Lebesgue measure for arbitrary locally compact topological groups (see \autoref{ssec:loccptcase} for a brief introduction). Although \enquote{the} Haar measure is not completely unique (except in compact groups, where there exists a natural choice), the system of sets of Haar measure zero is well-defined. This yields a very useful notion of smallness in locally compact groups, but says nothing of non-locally-compact groups, where there is no Haar measure.

In the paper \cite{Chr} (which was published in 1972) Christensen introduced the notion of \emph{Haar null sets}, which is equivalent to having Haar measure zero in locally compact groups and defined in every abelian Polish group. (A topological group is Polish if it is separable and completely metrizable, for the definition of Haar nullness see \autoref{def:haarnull}.) Twenty years later Hunt, Sauer and Yorke independently introduced this notion under the name of \emph{shy sets} in the paper \cite{HSY}. Since then lots of papers were published which either study some property of Haar null sets or use this notion of smallness to state facts which are true for almost every element of some structure. It was relatively easy to generalize this notion to non-abelian groups, on the other hand, the assumption that the topology is Polish is still usually assumed, because it turned out to be convenient and useful.

There is another widely used notion of smallness, the notion of \emph{meager sets} (also known as sets of the first category). Meager sets can be defined in any topological space; a set is said to be meager if it is the countable union of nowhere dense sets (i.e.\ sets which are not dense in any open set). A topological space is called a Baire space if the nonempty open sets are non-meager; this basically means that one can consider the meager sets small in these spaces. The Baire category theorem states that all completely metrizable spaces and all locally compact Hausdorff spaces are Baire spaces (see \cite[Theorem 8.4]{Ke} for the proofs).

In locally compact groups the system of meager sets and the system of sets of Haar measure zero share many properties. For example the classical Erd\H os-Sierpi\' nski duality theorem states that it is consistent that there is a bijection $f: \mathbb{R} \to \mathbb{R}$ such that for every set $A\subseteq \mathbb{R}$, $f(A)$ is meager if and only if $A$ has Lebesgue measure zero and $f(A)$ has Lebesgue measure zero if and only if $A$ is meager. Despite this, there are sets that are small in one sense and far from being small in the other sense, for example every locally compact group can be written as the union of a meager set and a set of Haar measure zero.

In 2013, Darji defined the notion of \emph{Haar meager sets} in the paper \cite{Da} to provide a better analog of Haar null sets in the non-locally-compact case. (For the definition see \autoref{def:haarmeager}.) Darji only considered abelian Polish groups, but \cite{Do} generalized this notion to non-abelian Polish groups. Haar meager sets coincide with meager sets in locally compact Polish groups, and Haar meagerness is a strictly stronger notion than meagerness in non-locally-compact abelian groups (see \autoref{thm:haarmeagerismeagerlc} and \autoref{thm:haarmeagerismeagernlc}). The analogy between Haar nullness and Haar meagerness means that most of the results for Haar null sets are also true for Haar meager sets and it is often possible to prove them using similar methods.

First in \autoref{sec:notatterm}, we introduce the notions, definitions and conventions which are not related to our area, but used repeatedly in this paper. Then \autoref{sec:basicprop} defines the core notions and investigates their most important properties.

After these, \autoref{sec:alternatives} considers the modified variants of the definitions. This section starts with a large collection of equivalent definitions for our core notions, then lists and briefly describes most of the versions which appear in the literature and are not (yet proved to be) equivalent to the \enquote{plain} versions.

The next section, \autoref{sec:analogs}, considers the feasibility of generalizing four well-known results (Fubini's theorem, the Steinhaus theorem, the countable chain condition, writing the whole space as the union of a null and a meager set) for non-locally-compact Polish groups.

Then in \autoref{sec:techniques} we discuss some proof techniques for questions from this area. Some of these are essentially useful lemmas, the others are just ideas and ways of thinking which can be helpful in certain cases.

Finally, in \autoref{sec:outlook} we give a highly incomplete list of applications of Haar null sets in various fields of mathematics and in \autoref{sec:questions} we collect the open questions that are mentioned in this survey.

\section{Notation and terminology}\label{sec:notatterm}

This section is the collection of the miscellaneous notations, definitions and conventions that are used repeatedly in this paper.

The symbols $\mathbb{N}$ and $\omega$ both refer to the set of nonnegative integers. We write $\mathbb{N}$ if we consider this set as a topological space (with the discrete topology) and $\omega$ if we use it only as a cardinal, ordinal or index set. (For example we write the Polish space of the countably infinite sequences of natural numbers as $\mathbb{N}^\omega$.) We consider the nonnegative integers as von Neumann ordinals, i.e.\ we identify the nonnegative integer $n$ with the set $\{0, 1, 2, \ldots, n-1\}$.

$\mathcal{P}(S)$ denotes the power set of a set $S$. For a set $S \subseteq X\times Y$, $x\in X$ and $y\in Y$, $S_x$ is the $x$-section $S_x=\{y : (x, y)\in S\}$ and $S^y$ is the $y$-section $S^y=\{x: (x, y)\in S\}$.

If $S$ is a subset of a topological space, $\Int(S)$ is the interior of $S$ and $\overline{S}$ is the closure of $S$. We consider $\mathbb{N}$, $\mathbb{Z}$ and all finite sets to be topological spaces with the discrete topology. (Note that this convention allows us to simply write the Cantor set as $2^{\omega}=\{0, 1\}^\omega$.) If $X$ is a topological space, then
\begin{description}[labelwidth=1.4cm, leftmargin=1.6cm,labelsep=0cm,noitemsep,nolistsep]
\item[$\mathcal{B}(X)$] denotes its Borel subsets ($\mathcal{B}(X)$ is the $\sigma$-algebra generated by the open sets, see \cite[Chapter II]{Ke}),
\item[$\mathcal{M}(X)$] denotes its meager subsets (a set is \emph{meager} if it is the union of countably many nowhere dense sets and a set is \emph{nowhere dense} if the interior of its closure is empty, see \cite[$\S$8.A]{Ke}).
\end{description}
If the space $X$ is \emph{Polish} (that is, separable and completely metrizable), then 
\begin{description}[labelwidth=1.4cm, leftmargin=1.6cm,labelsep=0cm,noitemsep,nolistsep]
\item[$\mathbf{\Sigma}^1_1(X)$] denotes its analytic subsets (a set is \emph{analytic} if it is the continuous image of a Borel set, see \cite[Chapter III]{Ke}),
\item[$\mathbf{\Pi}^1_1(X)$] denotes its coanalytic subsets (a set is \emph{coanalytic} if its complement is analytic, see \cite[Chapter IV]{Ke}).
\end{description}
If the topological space $X$ is clear from the context, we simply write $\mathcal{B}$, $\mathcal{M}$, $\mathbf{\Sigma}^1_1$ and $\mathbf{\Pi}^1_1$.

In a metric space $(X, d)$, $\diam(S)=\sup \{d(x, y) : x, y\in S\}$ denotes the diameter of the subset $S$. If $x\in X$ and $r>0$, then $B(x, r)=\{x'\in X : d(x,x')<r\}$ and $\overline{B}(x, r)=\{x'\in X : d(x,x')\le r\}$ denotes respectively the open and the closed ball with center $x$ and radius $r$ in $X$.

If $\mu$ is an outer measure on a set $X$, we say that $A\subseteq X$ is \emph{$\mu$-measurable} if $\mu(B)=\mu(B\cap A)+\mu(B\setminus A)$ for every $B\subseteq X$. Unless otherwise stated, we identify an outer measure $\mu$ with its restriction to the $\mu$-measurable sets and \enquote{measure} means an outer measure or the corresponding complete measure. A measure $\mu$ is said to be \emph{Borel} if all Borel sets are $\mu$-measurable. The support of the measure $\mu$ is denoted by $\supp\mu$.

Almost all of our results will be about topological groups. A set $G$ is called a \emph{topological group} if is equipped with both a group structure and a Hausdorff topology and these structures are \emph{compatible}, that is, the multiplication map $G\times G \to G$, $(g,h)\mapsto gh$ and the inversion map $G\to G$, $g \mapsto g^{-1}$ are continuous functions. We use the convention that whenever we require a group to have some topological property (for example a \enquote{compact group}, a \enquote{Polish group}, \ldots), then it means that the group must be a topological group and have that property (as a topological space). The identity element of a group $G$ will be denoted by $1_G$.

Most of the results in this paper are about certain subsets of Polish groups. Unless otherwise noted, $(G, \cdot)$ denotes an arbitrary Polish group. We denote the group operation by multiplication even when we assume that the (abstract) group under consideration is abelian, but we write the group operation of well-known concrete abelian groups like $(\mathbb{R}, +)$ or $(\mathbb{Z}^\omega, +)$ as addition.

Some techniques only work in Polish groups that \emph{admit a two-sided invariant metric}. (A metric $d$ on $G$ is called \emph{two-sided invariant} (or simply \emph{invariant}) if $d(g_1 hg_2 , g_1 kg_2 ) = d(h, k)$ for any $g_1, g_2 , h, k \in G$.) Groups with this property are also called TSI groups. This class of groups properly contains all Polish, abelian groups, since each metric group $G$ admits a left-invariant metric which, obviously, is invariant when $G$ is abelian. Any invariant metric on a Polish group is automatically complete. For proofs of these facts and more results about TSI groups see for example \cite[\S8.]{HR}.

Some basic results can be generalized for non-separable groups, but we will only deal with the separable case. On the other hand, many papers about this topic only consider abelian groups or some class of vector spaces. When the proof of a positive result can be generalized for arbitrary Polish group, we will usually do so, but we will usually provide counterexamples only in the special case where their construction is the simplest. If we assume that $G$ is locally compact, our notions will coincide with simpler notions (see \autoref{ssec:loccptcase}) and the majority of the results in this paper become significantly easier to prove, so the interesting case is when $G$ is not locally compact.

\section{Basic properties}\label{sec:basicprop}

\subsection{Core definitions}\label{ssec:core}

This subsection defines Haar null sets and Haar meager sets. Both of these notions have several slightly different formalizations in the literature, these are discussed in \autoref{sec:alternatives}.

Haar null sets were first introduced by Christensen in \cite{Chr} in 1972 as a generalization of the null sets of the Haar measure. (The Haar measure itself cannot be generalized for groups that are non-locally-compact, see \autoref{thm:nohaarmeas}.) Twenty years later in \cite{HSY} Hunt, Sauer and Yorke independently introduced Haar null sets under the name \enquote{shy sets}.

\begin{definition}\label{def:haarnull}
A set $A \subseteq G$ is said to be \emph{Haar null} if there are a Borel set $B\supseteq A$ and a Borel probability measure $\mu$ on $G$ such that $\mu(gBh)=0$ for every $g, h\in G$. A measure $\mu$ satisfying this is called a \emph{witness measure} for $A$. The system of Haar null subsets of $G$ is denoted by $\mathcal{HN}=\mathcal{HN}(G)$.\end{definition}

\begin{remark}\label{alternativenames}
Using the terminology introduced in \cite{HSY}, a set $A \subseteq G$ is called \emph{shy} if it is Haar null, and \emph{prevalent} if $G\setminus A$ is Haar null.
\end{remark}

Some authors (including Christensen) write \enquote{universally measurable set} instead of \enquote{Borel set} when they define Haar null sets. Although most results can be proved for both notions in the same way, these two notions are not equivalent. When a paper uses both notions, sets satisfying this alternative definition are called \enquote{generalized Haar null sets}.

\begin{definition}
If $X$ is a Polish space, a set $A\subseteq X$ is called \emph{universally measurable} if it is $\mu$-measurable for any $\sigma$-finite Borel measure $\mu$ on $X$.
\end{definition}

\begin{definition}\label{def:genhaarnull}
A set $A\subseteq G$ is said to be a \emph{generalized Haar null} if there are a universally measurable set $B\supseteq A$ and a Borel probability measure $\mu$ on $G$ such that $\mu(gBh)=0$ for every $g, h\in G$. A measure $\mu$ satisfying this is called a \emph{witness measure} for $A$. The system of generalized Haar null subsets of $G$ is denoted by $\mathcal{GHN}=\mathcal{GHN}(G)$\end{definition}

\begin{remark}\label{remark:HNisGHN}
As every Borel set is universally measurable, every Haar null set is generalized Haar null.
\end{remark}

Haar meager sets were first introduced by Darji in \cite{Da} in 2013 as a topological counterpart to the Haar null sets. (Meagerness remains meaningful in non-locally-compact groups, but Haar meager sets are a better analogue for Haar null sets.) This original definition only considered the case of abelian groups, but it was straightforwardly generalized in \cite{Do} to work in arbitrary Polish groups.

\begin{definition}\label{def:haarmeager}
A set $A \subseteq G$ is said to be \emph{Haar meager} if there are a Borel set $B \supseteq A$, a (nonempty) compact metric space $K$ and a continuous function $f : K \to G$ such that $f^{-1}(gBh)$ is meager in $K$ for every $g, h \in G$. A function $f$ satisfying this is called a \emph{witness function} for $A$. The system of Haar meager subsets of $G$ is denoted by $\mathcal{HM}=\mathcal{HM}(G)$.
\end{definition}

We note that according to \autoref{thm:haarnullwitnessfunc}, it is also possible to define the notion of Haar null sets in a similar fashion, using witness functions instead of witness measures.

\subsection{Notions of smallness}\label{ssec:smallness}

Both \enquote{Haar null} and \enquote{Haar meager} are notions of smallness (i.e.\ we usually think of Haar null and Haar meager sets as small or negligible). This point of view is justified by the fact that both the system of Haar null sets and the system of Haar meager sets are $\sigma$-ideals.

\begin{definition}\label{def:sigmaideal}
A system $\mathcal{I}$ (of subsets of some set) is called a \emph{$\sigma$-ideal} if
\begin{propertylist}
\item $\emptyset \in \mathcal{I}$, 
\item $A \in \mathcal{I}, B \subseteq A \Rightarrow B \in \mathcal{I}$ and
\item if $A_n \in \mathcal{I}$ for all $n \in \omega$, then $\bigcup_n A_n \in \mathcal{I}$.
\end{propertylist}
\end{definition}

To prove that these systems are indeed $\sigma$-ideals we will need some technical lemmas.

\begin{lemma}\label{lem:posmeascpt}
If $\mu$ is a Borel probability measure on $G$ and $U$ is a neighborhood of $1_G$, then there are a compact set $C\subseteq G$ and $t \in G$ with $\mu(C)>0$ and $C \subseteq tU$. 
\end{lemma}

\begin{proof}
Applying \cite[Theorem 17.11]{Ke}, there exists a compact set $\tilde{C} \subseteq G$ with $\mu(\tilde{C})\ge \frac12$. Fix an open set $V$ with $1_G \in V \subset \overline{V} \subset U$. The collection of open sets $\{tV : t\in \tilde{C}\}$ covers $\tilde{C}$ and $\tilde{C}$ is compact, so $\tilde{C} = \bigcup_{t \in F} (tV \cap \tilde{C})$ for some finite set $F\subseteq \tilde{C}$. It is clear that $\mu(tV \cap \tilde{C})$ must be positive for at least one $t\in F$. Choosing $C= \overline{tV} \cap \tilde{C}$ clearly satisfies our requirements.
\end{proof}

\begin{corollary}\label{cptsupwitness}
If $\mu$ is a Borel probability measure on $G$, $V \subseteq G$ is an open set, $B\subseteq G$ is universally measurable and satisfies $\mu(gBh)=0$ for every $g, h\in G$, then there exists a Borel probability measure $\mu'$ that satisfies $\mu'(gBh)=0$ for every $g, h\in G$ and has a compact support that is contained in $V$.
\end{corollary}
\begin{proof}
For an arbitrary $v\in V$, the set $U = v^{-1} V$ is a neighborhood of $1_G$. Applying the previous lemma, it is easy to check that $\mu'(X) := \frac{\mu(tX \cap C)}{\mu(C)}$ satisfies our requirements.
\end{proof}

\begin{lemma}\label{correctionlemma}
Let $d$ be a metric on $G$ that is compatible with the topology of $G$. If $L \subseteq G$ is compact and $\varepsilon>0$ is arbitrary, then there exists a neighborhood $U$ of $1_G$ such that $d(x\cdot u, x) < \varepsilon $ for every $x \in L$ and $u\in U$.
\end{lemma}
\begin{proof}
(Reproduced from \cite[Lemma 2]{Do}.) By the continuity of the function $(x, u)\mapsto d(x \cdot u, x)$, for every $x \in L$ there are neighborhoods $V_x$ of $x$ and $U_x$ of $1_G$ such that the image of $ V_x \times U_x$ is a subset of $[0, \varepsilon)$. Let $F\subseteq L$ be a finite set such that $L \subseteq \bigcup_{x\in F} V_x$. It is easy to check that $U=\bigcap_{x\in F} U_x$ satisfies our conditions.
\end{proof}

\begin{samepage}
\begin{theorem}[Christensen, Mycielski]\label{thm:haarnullsigmaideal}
\begin{multistmt}
\item The system $\mathcal{HN}$ of Haar null sets is a $\sigma$-ideal.
\item The system $\mathcal{GHN}$ of generalized Haar null sets is a $\sigma$-ideal.
\end{multistmt}
\end{theorem}
\end{samepage}

\begin{proof} It is trivial that both $\mathcal{HN}$ and $\mathcal{GHN}$ satisfy (I) and (II) in \autoref{def:sigmaideal}. The proof of (III) that is reproduced here is from the appendix of \cite{CK}, where a corrected version of the proof in \cite{Myc} is given. Proving this fact is easier in abelian Polish groups (see \cite[Theorem 1]{Chr}) and when the group is metrizable with a complete left invariant metric (this would allow the proof of \cite[Theorem 3]{Myc} to work without modifications). The appendix of \cite{CK} mentions the other approaches and discusses the differences between them.

The proof of (III) for Haar null and for generalized Haar null sets is very similar. The following proof will be for Haar null sets, but if \enquote{Borel set} is replaced with \enquote{universally measurable set} and \enquote{Haar null} is replaced with \enquote{generalized Haar null}, it becomes the proof for generalized Haar null sets.

Let $A_n$ be Haar null for all $n\in \omega$. By definition there are Borel sets $B_n \subseteq G$ and Borel probability measures $\mu_n$ on $G$ such that $A_n \subseteq B_n$ and $\mu_n(gB_n h)=0$ for every $g, h \in G$. Let $d$ be a complete metric on $G$ that is compatible with the topology of $G$ (as $G$ is Polish, it is completely metrizable).

We construct for all $n \in \omega$ a compact set $C_n \subseteq G$ and a Borel probability measure $\tilde\mu_n$ such that the support of $\tilde\mu_n$ is $C_n$, $\tilde\mu_n(gB_n h)=0$ for every $g, h\in G$ (i.e. $\tilde\mu_n$ is a witness measure) and the \enquote{size} of the sets $C_n$ decreases \enquote{quickly}.

The construction will be recursive. For the initial step use \autoref{cptsupwitness} to find a Borel probability measure $\tilde\mu_0$ that satisfies $\tilde\mu_0(gB_0 h)=0$ for every $g, h\in G$ and that has compact support $C_0 \subseteq G$. Assume that $\tilde{\mu}_{n'}$ and $C_{n'}$ are already defined for all $n'<n$. By \autoref{correctionlemma} there exists a neighborhood $U_n$ of $1_G$ such that if $u\in U_n$, then $d(k\cdot u, k)<2^{-n}$ for every $k$ in the compact set $C_0C_1C_2\cdots C_{n-1}$. Applying \autoref{cptsupwitness} again we can find a Borel probability measure $\tilde{\mu}_n$ that satisfies $\tilde\mu_n(gB_n h)=0$ for every $g, h\in G$ and that has a compact support $C_n \subseteq U_n$.

If $c_n \in C_n$ for all $n\in \omega$, then it is clear that the sequence $(c_0c_1c_2\cdots c_n)_{n\in\omega}$ is a Cauchy sequence. As $(G, d)$ is complete, this Cauchy sequence is convergent; we write its limit as the infinite product $c_0c_1c_2\cdots$. The map $\varphi : \prod_{n\in\omega} C_n \to G$, $\varphi((c_0, c_1, c_2, \ldots)) = c_0c_1c_2\cdots$ is the uniform limit of continuous functions, hence it is continuous. 

Let $\mu^\Pi$ be the product of the measures $\tilde{\mu}_n$ on the product space $C^\Pi := \prod_{n\in\omega} C_n$. Let $\mu = \varphi_*(\mu^\Pi)$ be the push-forward of $\mu^\Pi$ along $\varphi$ onto $G$, i.e. 
\[\mu(X) = \mu^\Pi(\varphi^{-1}(X))= \mu^\Pi\left(\left\{(c_0, c_1, c_2, \ldots) \in C^\Pi : c_0c_1c_2\cdots \in X\right\}\right).\]

We claim that $\mu$ witnesses that $A=\bigcup_{n\in\omega} A_n$ is Haar null. Note that $A$ is contained in the Borel set $B = \bigcup_{n\in\omega} B_n$, so it is enough to show that $\mu(gBh)=0$ for every $g, h\in G$. As $\mu$ is $\sigma$-additive, it is enough to show that $\mu(g B_n h)=0$ for every $g, h\in G$ and $n\in\omega$. 

Fix $g, h \in G$ and $n \in \omega$. Notice that if $c_j\in C_j$ for every $j\neq n$, $j\in\omega$, then
\begin{align*}
&\tilde{\mu}_n\left(\{c_n \in C_n: c_0c_1c_2\cdots c_n\cdots \in g B_n h\}\right)=\\
&\quad=\tilde{\mu}_n\left(\left(c_0c_1\cdots c_{n-1}\right)^{-1}\cdot g B_n h\cdot \left(c_{n+1}c_{n+2}\cdots\right)^{-1}\right) = 0
\end{align*}
because $\tilde\mu_n(g'B_n h')=0$ for all $g', h'\in G$. Applying Fubini's theorem in the product space $\left(\prod_{j\neq n} C_j\right)\times C_n$ to the product measure  $\left(\prod_{j\neq n} \tilde{\mu}_j\right)\times \tilde{\mu}_n$ yields that 
\[0=\mu^\Pi\left(\left\{(c_0, c_1, \ldots, c_n, \ldots) \in C^\Pi : c_0c_1\cdots c_n\cdots \in gB_n h \right\}\right).\]
By the definition of $\mu$ this means that $\mu(gB_n h)=0$.
\end{proof}

The analogous statement for Haar meager sets was proved as \cite[Theorem 2.9]{Da} in the abelian case, and as \cite[Theorem 3]{Do} in the general case:

\begin{theorem}[Darji, Dole\v zal-Rmoutil-Vejnar-Vlas\' ak]\label{thm:haarmeagersigmaideal}
The system $\mathcal{HM}$ of Haar meager sets is a $\sigma$-ideal.
\end{theorem}

\begin{proof}
Again, the proof of (I) and (II) in \autoref{def:sigmaideal} is obvious. We reproduce the proof of (III) from \cite[Theorem 3]{Do}. This proof will be very similar to the proof of \autoref{thm:haarnullsigmaideal}, but restricting the witnesses to a smaller \enquote{part} of $G$ is simpler in this case (we do not need an analogue of \autoref{cptsupwitness}).

Let $A_n$ be Haar meager for all $n\in \omega$. By definition there are Borel sets $B_n \subseteq G$, compact metric spaces $K_n \neq \emptyset$ and continuous functions $f_n : K_n \to G$ such that $f_n^{-1}(gB_nh)$ is meager in $K_n$ for every $g, h\in G$. Let $d$ be a complete metric on $G$ that is compatible with the topology of $G$.

We construct for all $n \in \omega$ a compact metric space $\tilde{K}_n$ and a continuous function $\tilde{f}_n : \tilde{K}_n \to G$ satisfying that $\tilde{f}_n^{-1}(gB_n h)$ is meager in $\tilde{K}_n$ for every $g, h\in G$ (i.e. $\tilde{f}_n$ is a witness function) and the \enquote{size} of the images $\tilde{f}_n(\tilde{K}_n)\subseteq G$ decreases \enquote{quickly}.

Unlike the Haar null case, we do not have to apply recursion in this construction. By \autoref{correctionlemma} there exists a neighborhood $U_n$ of $1_G$ such that if $u\in U_n$, then $d(k\cdot u, k) <2^{-n}$ for every $k$ in the compact set $f_0(K_0)f_1(K_1)\cdots f_{n-1}(K_{n-1})$. Let $x_n \in f_n(K_n)$ be an arbitrary element and $\tilde{K}_n =\overline{f_n^{-1}(x_n U_n)}$. The set $\tilde{K}_n$ is compact (because it is a closed subset of a compact set) and nonempty. Let $\tilde{f}_n : \tilde{K}_n \to G$, $\tilde{f}_n(k) = x_n^{-1} f_n(k)$, this is clearly continuous. 

\begin{claim}\label{meagerwitness}
For every $n\in\omega$ and $g, h \in G$, $\tilde{f}_n^{-1}(gB_nh)$ is meager in $\tilde{K}_n$.
\end{claim}

\begin{proof}
Fix $n\in\omega$ and $g, h \in G$. The set $\tilde{f}_n^{-1}(U_n)$ is open in $K_n$ and because $f$ is a witness function, the set $\tilde{f}_n^{-1}(g B_n h) = f_n^{-1}(x_n g B_n h)$ is meager in $K_n$. This means that $\tilde{f}_n^{-1}(U_n)\cap \tilde{f}_n^{-1}(gB_n h)$ is meager in $\tilde{f}_n^{-1}(U_n)$. Since each open subset of $K_n$ is comeager in its closure and the closure of $\tilde{f}_n^{-1}(U_n)=f_n^{-1}(x_n U_n)$ is $\overline{f_n^{-1}(x_n U_n)}=\tilde{K}_n$, simple formal calculations yield that $\tilde{f}_n^{-1}(g B_n h) \cap \tilde{K}_n$ is meager in $\tilde{K}_n$.
\end{proof}

Let $K$ be the compact set $\prod_{n\in \omega} \tilde{K}_n$ and for $n\in\omega$ let $\psi_n$ be the continuous function $\psi_n : K \to G$, 
\[\psi_n(k) = \tilde{f}_0(k_0)\cdot \tilde{f}_1(k_1)\cdot \ldots \cdot \tilde{f}_{n-1}(k_{n-1}).\]
By the choice of $U_n$ we obtain $d(\psi_{n-1}(k), \psi_{n}(k)) \le 2^{-n}$ for every $k\in K$. Using the completeness of $d$ this means that the sequence of functions $(\psi_n)_{n\in\omega}$ is uniformly convergent. Let $f: K \to G$ be the limit of this sequence. $f$ is continuous, because it is the uniform limit of continuous functions.

We claim that $f$ witnesses that $A=\bigcup_{n\in \omega} A_n$ is Haar meager. Note that $A$ is contained in the Borel set $B = \bigcup_{n\in\omega} B_n$, so it is enough to show that $f^{-1}(gBh)$ is meager in $K$ for every $g, h\in G$. As meager subsets of $K$ form a $\sigma$-ideal, it is enough to show that $f^{-1}(g B_n h)$ is meager in $K$ for every $g, h\in G$ and $n\in\omega$. 

Fix $g, h\in G$ and $n\in \omega$. Notice that if $k_j \in \tilde{K}_j$ for every $j\neq n$, $j\in \omega$, then \autoref{meagerwitness} means that
\begin{align*}
&\{k_n \in \tilde{K}_n: f(k_0, k_1, \ldots, k_n, \ldots) \in g B_n h\}=\\
&\quad=\{k_n \in \tilde{K}_n : \tilde{f}_0(k_0)\cdot \tilde{f}_1(k_1)\cdot \ldots \cdot \tilde{f}_n(k_n)\cdot \ldots \in gB_n h\}\\
&\quad=\tilde{f}_n^{-1}\left(\left(\tilde{f}_0(k_0)\cdot \ldots \tilde{f}_{n-1}(k_{n-1})\right)^{-1}\cdot g B_n h\cdot \left(\tilde{f}_{n+1}(k_{n+1})\cdot\tilde{f}_{n+2}(k_{n+2})\cdot\ldots\right)^{-1}\right)
\end{align*}
is meager in $\tilde{K}_n$. Applying the Kuratowski-Ulam theorem (see e.g.\ \cite[Theorem 8.41]{Ke}) in the product space $\left(\prod_{j\neq n} \tilde{K}_j\right)\times \tilde{K}_n$, the Borel set $f^{-1}(gB_nh)$ is meager.
\end{proof}

As the group $G$ acts on itself via multiplication, it is useful if this action does not convert \enquote{small} sets into \enquote{large} ones. This means that a \enquote{nice} notion of smallness must be a translation invariant system.

\begin{definition} A system $\mathcal{I}\subseteq \mathcal{P}(G)$ is called \emph{translation invariant} if $A\in \mathcal{I} \Leftrightarrow gAh\in \mathcal{I}$ for every $A \subseteq G$ and $g, h\in G$.
\end{definition}

\begin{proposition}
The $\sigma$-ideals $\mathcal{HN}$, $\mathcal{GHN}$ and $\mathcal{HM}$ are all translation invariant.
\end{proposition}

\begin{proof} This is clear from \autoref{def:haarnull}, \autoref{def:genhaarnull} and \autoref{def:haarmeager}.
\end{proof}

If a nontrivial notion of smallness has these \enquote{nice} properties, then the following lemma states that countable sets are small and nonempty open sets are not small. Applying this simple fact for our $\sigma$-ideals is often useful in simple cases.

\begin{lemma}\label{lem:ctblsmallopenlarge}
Let $\mathcal{I}$ be a translation invariant $\sigma$-ideal that contains a nonempty set but does not contain all subsets of $G$. If $A \subseteq G$ is countable, then $A \in \mathcal{I}$, and if $U\subseteq G$ is nonempty open, then $U\notin \mathcal{I}$.
\end{lemma}

\begin{proof}
If $x\in G$, then any nonempty set in $\mathcal{I}$ has a translate that contains $\{x\}$ as a subset, hence $\{x\}\in \mathcal{I}$. Using that $\mathcal{I}$ is closed under countable unions, this yields that if $A \subseteq G$ is countable, then $A \in \mathcal{I}$. To prove the other claim, suppose for a contradiction that $U\subseteq G$ is a nonempty open set that is in $\mathcal{I}$. It is clear that $G=\bigcup_{g\in G} gU$, and as $G$ is Lindel\"of, $G=\bigcup_{n\in\omega} g_n U$ for some countable subset $\{g_n : n\in \omega\}\subseteq G$. But here $g_n U\in \mathcal{I}$ (because $\mathcal{I}$ is translation invariant) and thus $G \in \mathcal{I}$ (because $\mathcal{I}$ is closed under countable unions), and this means that $\mathcal{I}$ contains all subsets of $G$, and this is a contradiction.
\end{proof}

\begin{remark}
Let $\mathcal{I}$ be one of the $\sigma$-ideals $\mathcal{HN}$, $\mathcal{GHN}$ and $\mathcal{HM}$. If $G$ is countable, then $\mathcal{I}=\{\emptyset\}$, otherwise $\mathcal{I}$ contains a nonempty set and does not contain all subsets of $G$.
\end{remark}

\subsection{Connections to Haar measure and meagerness}\label{ssec:loccptcase}

This subsection discusses the connection between sets with Haar measure zero and Haar null sets and the connection between meager sets and Haar meager sets. In the simple case when $G$ is locally compact we will find that equivalence holds for both pairs, justifying the names \enquote{Haar null} and \enquote{Haar meager}. When $G$ is non-locally-compact, we will see that the first connection is broken by the fact that there is no Haar measure on the group. For the other pair we will see that Haar meager sets are always meager, but there are Polish groups where the converse is not true.

First we recall some well-known facts about Haar measures. For proofs and more detailed discussion see for example \cite[\S15]{HR}.

\begin{definition}
If $(X, \Sigma)$ is a measurable space with $\mathcal{B}(X) \subseteq \Sigma$, then a measure $\mu : \Sigma \to [0,\infty]$ is called \emph{regular} if $\mu(U)= \sup \{ \mu(K) : K\subseteq U, K\text{ is compact}\}$ for every $U$ open set and $\mu(A)=\inf \{ \mu(U) : A \subseteq U, U\text{ is open}\}$ for every set $A$ in the domain of $\mu$. 
\end{definition}

\begin{definition}\label{def:haarmeas}
If $G$ is a topological group (not necessarily Polish), a measure $\lambda : \mathcal{B}(G)\to [0,\infty]$ is called a \emph{left Haar measure} if it satisfies the following properties:
\begin{propertylist}
\item $\lambda(F)<\infty$ if $F$ is compact,
\item $\lambda(U)>0$ if $U$ is a nonempty open set,
\item $\lambda(gB)=\lambda(B)$ for all $B\in \mathcal{B(G)}$ and $g\in G$ (\emph{left invariance}),
\item $\lambda$ is regular.
\end{propertylist}
If left invariance is replaced by the property $\lambda(Bg)=\lambda(B)$ for all $B\in \mathcal{B}(G)$ and $g\in G$ (\emph{right invariance}), the measure is called a \emph{right Haar measure}.
\end{definition}

\begin{theorem}[existence of the Haar measure]\label{thm:haarmeasexists}
If $G$ is a locally compact group (not necessarily Polish), then there exists a left (right) Haar measure on $G$ and if $\lambda_1$, $\lambda_2$ are two left (right) Haar measures, then $\lambda_1=c \cdot \lambda_2$ for a positive real constant $c$.
\end{theorem}

If $G$ is compact, (I) means that the left and right Haar measures are finite measures, and this fact can be used to prove the following result:

\begin{theorem} If $G$ is a compact group, then all left Haar measures are right Haar measures and vice versa.
\end{theorem}

This result is also trivially true in abelian locally compact groups, but not true in all locally compact groups. However, the following result remains true:
\begin{theorem} \label{thm:samezeroset}
If $G$ is a locally compact group, then the left Haar measures and the right Haar measures are absolutely continuous relatively to each other, that is, for every Borel set $B\subseteq G$, either every left Haar measure and every right Haar measure assigns measure zero to $B$ or no left Haar measure and no right Haar measure assigns measure zero to $B$.
\end{theorem}

This allows us to define the following notion:
\begin{definition}
Suppose that $G$ is a locally compact group and fix an arbitrary left (or right) Haar measure $\lambda$. We say that a set $N\subseteq G$ has \emph{Haar measure zero} if $N\subseteq B$ for some Borel set $B$ with $\lambda(B)=0$. The collection of these sets is denoted by $\mathcal{N}=\mathcal{N}(G)$.
\end{definition}

\autoref{def:haarmeas} defines the Haar measures only on the Borel sets. If $\lambda$ is an arbitrary left (or right) Haar measure, we can complete it using the standard techniques. The domain of the completion will be $\sigma(\mathcal{B}(G)\cup \mathcal{N})$ (the $\sigma$-algebra generated by $\mathcal{N}$ and the Borel sets). For every set $A$ in this $\sigma$-algebra, let
\[\lambda(A) = \inf\left\{\sum_{j\in\omega} \lambda(B_j) : B_j\in \mathcal{B}(G), A \subseteq \bigcup_{j\in\omega}B_j\right\}.\]
This completion will be a complete measure that agrees with the original $\lambda$ on Borel sets and satisfies properties (I) -- (IV) from \autoref{def:haarmeas} (or right invariance instead of left invariance if $\lambda$ was a right Haar measure). 
We will identify a left (or right) Haar measure with its completion and we will also call this extension (slightly imprecisely) a left (or right) Haar measure.

The following theorem justifies the choice of the name of \enquote{Haar null} sets. The foundational paper \cite{Chr} shows this in the abelien case; \cite[Theorem 1]{Myc} gives a complete proof (using a slightly different definition of Haar measures).

\begin{theorem}[Christensen, Mycielski]\label{thm:loccptequivnotions}
If $G$ is a locally compact Polish group, then system of sets with Haar measure zero is the same as the system of Haar null sets and is the same as the system of generalized Haar null sets, that is, $\mathcal{N}(G)=\mathcal{HN}(G)=\mathcal{GHN}(G)$.
\end{theorem}

\begin{proof} 

$\mathcal{N}(G)\subseteq \mathcal{HN}(G)$:\\ 
Let $\lambda$ be a left Haar measure and $\lambda'$ be a right Haar measure. If $N\in \mathcal{N}(G)$ is arbitrary, then by definition there is a Borel set $B$ satisfying $N\subseteq B$ and $\lambda(B)=0$. The left invariance of $G$ means that $\lambda(gB)=0$ for every $g\in G$. Applying \autoref{thm:samezeroset} this means that $\lambda'(gB)=0$ for every $g\in G$, and applying the right invariance of $\lambda'$ we get that $\lambda'(gBh)=0$ for every $g, h\in G$.  Using the regularity of $\lambda'$, it is easy to see that there is a compact set $K$ with $0<\lambda'(K)<\infty$. The measure $\mu(X) = \frac{\lambda'(K \cap X)}{\lambda'(K)}$ is clearly a Borel probability measure. $\mu \ll \lambda'$ means that $\mu(gBh)=0$ for every $g, h\in G$, so $B$ and $\mu$ satisfy the requirements of \autoref{def:haarnull}.

$\mathcal{HN}(G)\subseteq\mathcal{GHN}(G)$:\\
This is obviously true in all Polish groups, see \autoref{remark:HNisGHN}.

$\mathcal{GHN}(G)\subseteq \mathcal{N}(G)$:\\
Suppose that $A\in \mathcal{GHN}(G)$. By definition there exists a universally measurable $B \subseteq G$ and a Borel probability measure $\mu$ such that $\mu(gBh)=0$ for every $g, h\in G$. Notice that we will only use that $\mu(Bh)=0$ for every $h\in G$, so we will also prove that (using the terminology of \autoref{ssec:leftrighthaarnull}) all generalized right Haar null sets have Haar measure zero. Let $\lambda$ be a left Haar measure on $G$. Let $m$ be the multiplication map $m : G\times G \to G$, $(x, y)\mapsto x\cdot y$.

Notice that the set $m^{-1}(B)=\{(x, y) \in G\times G : x\cdot y \in B\}$ is universally measurable in $G\times G$, because it is the preimage of a universally measurable set under the continuous map $m$. (This follows from that the preimage of a Borel set under $m$ is Borel, and for every $\sigma$-finite measure $\nu$ on $G\times G$, the preimage of a set of $m_*(\nu)$-measure zero under $m$ must be of $\nu$-measure zero. Here $m_*(\nu)$ is the push-forward measure: $m_*(\nu)(X) = \nu(\{(x,y) : x\cdot y \in X\})$.)

Applying Fubini's theorem in the product space $G\times G$ to the product measure $\mu\times \lambda$ (which is a $\sigma$-finite Borel measure) we get that
\[(\mu\times \lambda)(m^{-1}(B))= \int_G \lambda(\{y: x \cdot y \in B\})\intby\mu(x)=\int_G\mu(\{x : x \cdot y \in B\}\intby\lambda(y),\]
hence
\[\int_G \lambda(x^{-1} B)\intby\mu(x) = \int_G\mu(By^{-1}) \intby\lambda(y).\]
As $\mu$ is a witness measure, the right hand side is the integral of the constant 0 function. On the left hand side $\lambda(x^{-1}B)=\lambda(B)$, as $\lambda$ is left invariant (note that $B$ is $\lambda$-measurable, because $B$ is universally measurable and $\lambda$ is $\sigma$-finite). Thus $0=\int_G \lambda(B)\intby\mu(x)=\lambda(B)$. As $A\subseteq B$, this means that $\lambda(A)=0$, $A\in \mathcal{N}(G)$.
\end{proof}

We reproduce the proof of the classical theorem which shows that (left and right) Haar measures do not exist on topological groups that are not locally compact. We will apply the following generalized version of the Steinhaus theorem. (The original version by Steinhaus considered the Lebesgue measure on the group $(\mathbb{R}, +)$, Weil generalized this for the case of Haar measures \cite[page 50]{W}.)

\begin{theorem}[Steinhaus, Weil]\label{thm:SteinhausHaar}
If $G$ is a topological group, $\lambda$ is a left Haar measure on $G$ and $C \subseteq G$ is compact with $\lambda(C)>0$, then $1_G \in \Int(C\cdot C^{-1})$.
\end{theorem}

\begin{proof}
As $\lambda$ is a Haar measure and $C$ is compact, $\lambda(C)<\infty$. Using the regularity of $\lambda$, there is an open set $U \supseteq C$ that satisfies $\lambda(U)< 2\lambda(C)$.

\begin{claim} There exists an open neighborhood $V$ of $1_G$ such that $V\cdot C \subseteq U$. 
\end{claim}

\begin{proof}
For every $c\in C$ the multiplication map $m : G\times G\to G$ is continuous at $(1_G,c)$, so $c \in V_c \cdot W_c \subseteq U$ for some open neighborhood $V_c$ of $1_G$ and some open neighborhood $W_c$ of $c$. As $C$ is compact and $\bigcup_{c\in C} W_c \supseteq C$, there is a finite set $F$ with $\bigcup_{c\in F} W_c \supseteq C$. Then $V = \bigcap_{c\in F} V_c$ satisfies $V\cdot W_c \subseteq U$ for every $c\in F$, so $V\cdot C \subseteq U$.
\end{proof}

Now it is enough to prove that $V \subseteq C\cdot C^{-1}$. Choose an arbitrary $v\in V$. Then $v\cdot C$ and $C$ are subsets of $U$ and $\lambda(v\cdot C)=\lambda(C)>\frac{\lambda(U)}{2}$ (we used the left invariance of $\lambda$). This means that $v\cdot C\cap C\neq \emptyset$, so there exists $c_1, c_2\in C$ with $v c_1=c_2$, but this means that $v=c_2 c_1^{-1}\in C\cdot C^{-1}$.
\end{proof}

We note that in Polish groups it is possible to find a compact subset with positive Haar measure in every set with positive Haar measure. Hence the following version of the previous theorem is also true:
\begin{corollary}[Steinhaus, Weil]\label{cor:SteinhausHaarGen}
If $G$ is a locally compact Polish group, $\lambda$ is a left Haar measure on $G$ and $A \subseteq G$ is $\lambda$-measurable with $\lambda(A)>0$, then $1_G \in \Int(A\cdot A^{-1})$.
\end{corollary}
Other, more general variants of the Steinhaus theorem are examined in \autoref{ssec:steinhaus}.

Now we are ready to prove that Haar measures only exist in the locally compact case:

\begin{theorem}\label{thm:nohaarmeas}
If $G$ is a topological group and $\lambda$ is a left Haar measure on $G$, then $G$ is locally compact.
\end{theorem}

\begin{proof}
$\lambda(G)>0$ as $G$ is open. Using the regularity of $G$, there exists a compact set $C$ with $\lambda(C)>0$. The set $C\cdot C^{-1}$ is compact (it is the image of the compact set $C\times C$ under the continuous map $(x,y)\mapsto xy^{-1}$). Applying \autoref{thm:SteinhausHaar} yields that $C\cdot C^{-1}$ is a neighborhood of $1_G$, but then for every $g\in G$ the set $g\cdot C\cdot C^{-1}$ is a compact neighborhood of $g$, and this shows that $G$ is locally compact.
\end{proof}

The connection between meager sets and Haar meager sets is simpler. The following results are from \cite{Da} (the first paper about Haar meager sets, which only considers abelian Polish groups) and \cite{Do} (where the concept of Haar meager sets is extended to all Polish groups).
\begin{theorem}[Darji, Dole\v zal-Rmoutil-Vejnar-Vlas\' ak]\label{thm:haarmeagerismeager}
Every Haar meager set is meager, $\mathcal{HM}(G) \subseteq \mathcal{M}(G)$.
\end{theorem}

\begin{proof}
Let $A$ be a Haar meager subset of $G$. By definition there exists a Borel set $B\supseteq A$, a (nonempty) compact metric space $K$ and a continuous function $f : K \to G$ such that $f^{-1}(gBh)$ is meager in $K$ for every $g,h\in G$.       

Consider the set
\[S=\{(g, k) :f(k)\in gB\} \subseteq G\times K,\]
which is Borel because it is the preimage of $B$ under the continuous map $(g, k)\mapsto g^{-1} \cdot f(k)$. For every $g\in G$, the $g$-section of this set is $S_g= \{k \in K : f(k)\in gB\}=f^{-1}(gB)$, and this is a meager set in $K$. Hence, by the Kuratowski-Ulam theorem, $S$ is meager in $G\times K$. Using the Kuratowski-Ulam theorem again, for comeager many $k\in K$, the section $S^k=\{g\in G: f(k)\in gB\}=f(k)\cdot B^{-1}$ is meager in $G$. Since $K$ is compact, there is at least one such $k$. Then the inverse of the homeomorphism $b\mapsto f(k)\cdot b^{-1}$ maps the meager set $S^k$ to $B$, and this shows that $B$ is meager.
\end{proof}

\begin{theorem}[Darji, Dole\v zal-Rmoutil-Vejnar-Vlas\' ak]\label{thm:haarmeagerismeagerlc}
In a locally compact Polish group $G$ meagerness is equivalent to Haar meagerness, that is, $\mathcal{HM}(G) = \mathcal{M}(G)$.
\end{theorem}

\begin{proof}
We only need to prove the inclusion $\mathcal{M}(G) \subseteq \mathcal{HM}(G)$. As $G$ is locally compact, there is a nonempty open set $U \subseteq G$ such that $\overline{U}$ is compact. Let $f : \overline{U}\to G$ be the identity map restricted to $\overline{U}$. If $M$ is meager in $G$, then there exists a meager Borel set $B\supseteq M$. The set $gBh$ is meager in $G$ for every $g, h\in G$ (as $x\mapsto gxh$ is a homeomorphism), so $f^{-1}(gBh)=gBh\cap \overline{U}$ is meager in $\overline{U}$ for every $g, h\in G$.
\end{proof}

\begin{theorem}[Darji, Dole\v zal-Rmoutil-Vejnar-Vlas\' ak]\label{thm:haarmeagerismeagernlc}
In a non-locally-compact Polish group $G$ that admits a two-sided invariant metric meagerness is a strictly stronger notion than Haar meagerness, that is, $\mathcal{HM}(G) \subsetneqq \mathcal{M}(G)$.
\end{theorem}

\begin{proof}
We know that $\mathcal{HM}(G) \subseteq \mathcal{M}(G)$. To construct a meager but not Haar meager set, we will use a theorem from \cite{Sol}. As the proof of this purely topological theorem is relatively long, we do not reproduce it here.

\begin{theorem}[Solecki]\label{thm:solpuretopol}
Assume that $G$ is a non-locally-compact Polish group that admits a two-sided invariant metric. Then there exists a closed set $F\subseteq G$ and a continuous function $\varphi : F\to 2^\omega$ such that for any $x\in 2^\omega$ and any compact set $C \subseteq G$ there is a $g\in G$ with $gC\subseteq \varphi^{-1}(\{x\})$.
\end{theorem}

Using this we construct a closed nowhere dense set $M$ that is not Haar meager. The system $\{f^{-1}(\{x\}) : x \in 2^\omega\}$ contains continuum many pairwise disjoint closed sets. If we fix a countable basis in $G$, only countably many of these sets contain an open set from that basis. If for $x_0\in2^\omega$ the set $M:=f^{-1}(\{x_0\})$ does not contain a basic open set, then it is nowhere dense (as it is closed with empty interior). On the other hand, it is clear that $M$ is not Haar meager, as for every compact metric space $K$ and continuous function $f:K\to G$ there exists a $g\in G$ such that $gf(K) \subseteq M$, thus $f^{-1}(g^{-1} M)=K$.
\end{proof}

\section{Alternative definitions}\label{sec:alternatives}

In this section first we discuss various alternative definitions which are equivalent to the \enquote{normal} definitions, but may be easier to prove or easier to use in some situations. After this, we will briefly describe some other versions which appeared in papers about this topic.

\subsection{Equivalent versions}\label{ssec:equivvar}

In this subsection we mention some alternative definitions which are equivalent to \autoref{def:haarnull}, \autoref{def:genhaarnull} or \autoref{def:haarmeager}. Most of these equivalences are trivial, but even these trivial equivalences can be frequently used as lemmas. First we list some versions of the definition of Haar null sets.

\begin{theorem}\label{thm:haarnulleqv}
For a set $A\subseteq G$ the following are equivalent:
\begin{multistmt}
\item there exists a Borel set $B\supseteq A$ and a Borel probability measure $\mu$ on $G$ such that $\mu(gBh)=0$ for every $g, h\in G$ (i.e. $A$ is Haar null),
\item there exists a Borel Haar null set $B\supseteq A$,
\item there exists a Borel generalized Haar null set $B\supseteq A$,
\item there exists an analytic set $B\supseteq A$ and a Borel probability measure $\mu$ on $G$ such that $\mu(gBh)=0$ for every $g, h\in G$,
\item there exists an analytic generalized Haar null set $B \supseteq A$.
\end{multistmt}
\end{theorem}

\begin{proof}
First note that Lusin's theorem (see \cite[29.7]{Ke}) states that all analytic sets are universally measurable, hence $gBh$ is $\mu$-measurable in condition (4).

$(1) \Leftrightarrow (2) \Rightarrow (3)$ is trivial from the definitions. $(3) \Rightarrow (1)$ follows from the fact that if (3) is true, then there exists a Borel probability measure $\mu$ such that $\mu(gB'h)=0$ for some (universally measurable) $B'\supseteq B$ and every $g, h\in G$, but this means that $\mu(gBh)=0$ for every $g, h\in G$.

The implication $(1)\Rightarrow (4)$ follows from the fact that all Borel sets are analytic. The implication $(4) \Rightarrow (5)$ is trivial again, considering that all analytic sets are universally measurable.

Finally we prove $(5)\Rightarrow (1)$ to conclude the proof of the theorem. This proof is reproduced from \cite{Sol}. Without loss of generality we may assume that the set $A$ itself is analytic generalized Haar null. We have to prove that there exists a Borel set $B\supseteq A$ and a Borel probability measure $\mu$ on $G$ such that $\mu(gBh)=0$ for every $g, h\in G$.

By definition there exists a Borel probability measure $\mu$ such that $\mu(g\tilde{B}h)=0$ for some (universally measurable) $\tilde{B}\supseteq A$ and every $g, h\in G$, but this means that $\mu(gAh)=0$ for every $g, h\in G$.

\begin{claim} The family of sets
\[\Phi = \{X \subseteq G : X\text{{\upshape\ is analytic and $\mu(gXh)=0$ for every $g,h \in G$}}\}\]
is \emph{coanalytic on analytic}, that is, for every Polish space $Y$ and $P \in \mathbf{\Sigma}^1_1(Y\times G)$, the set $\{y \in Y : P_y\in \Phi \}$ is $\mathbf{\Pi}^1_1$.
\end{claim}

\begin{proof}
Let $Y$ be Polish space and $P \in \mathbf{\Sigma}^1_1(Y\times G)$ and let 
\[\tilde{P} =\{(g, h, y, \gamma) \in G\times G\times Y\times G : \gamma \in gP_yh\}.\]
Then $\tilde{P}$ is analytic, as it is the preimage of $P$ under $(g, h, y,\gamma)\mapsto (y, g^{-1}\gamma h^{-1})$. We will use the following result, which can be found as \cite[Theorem 29.26]{Ke}:
\begin{theorem}[Kond\^o-Tugu\'e]\label{thm:kondotugue}
Let $X, Y$ be standard Borel spaces and $A\subseteq X \times Y$ an analytic set. Then the set
\[\{(\nu, x, r) \in P(Y)\times X \times \mathbb{R} : \nu(A_x)>r\}\text{ is analytic.}\]
\end{theorem}
(Standard Borel spaces are introduced in \cite[Section 12]{Ke}, we only use the fact that every Polish space is a standard Borel space.)

Using this result yields that 
\[\{(\nu, g, h, y, r) \in P(G)\times G\times G\times Y\times\mathbb{R}: \nu(\tilde{P}_{(g, h, y)})> r\}\]
is analytic, therefore its section at $(\nu=\mu, r=0)$, which is the set
\[\{(g, h, y) \in G\times G\times Y: \mu(\tilde{P}_{(g, h, y)})> 0\}\]
is also analytic. Projecting this on $Y$ yields that 
\[\{y \in Y : \mu(\tilde{P}_{(g, h, y)})> 0\text{ for some $g, h \in G$}\}\]
is analytic, but then
\[\{y\in Y : \mu(\tilde{P}_{(g, h, y)})=0 \text{ for all $g, h\in G$}\}=\{y\in Y : P_y\in \Phi\}\]
is coanalytic.
\end{proof}
Now, since $A\in \Phi$, by the dual form of the First Reflection Theorem (see \cite[Theorem 35.10]{Ke} and the remarks following it) there exists a Borel set $B$ with $B\supseteq A$ and $B\in \Phi$, and this $B$ (together with $\mu$) satisfies our requirements.

We note that in \autoref{ssec:gnchn} we will use a modified variant of this method to prove the stronger result \autoref{thm:dodosborelhull}. 
\end{proof}

In \autoref{def:haarnull} and \autoref{def:genhaarnull} the witness measure is required to be a Borel probability measure, but some alternative conditions yield equivalent definitions. A set $A\subseteq G$ is Haar null (or generalized Haar null) if and only if there is a Borel (or universally measurable) set $B\supseteq A$ that satisfies the equivalent conditions listed in the following theorem.

\begin{theorem}\label{thm:witnessmeasequiv}
For a universally measurable set $B\subseteq G$ the following are equivalent: \begin{multistmt}
\item there exists a Borel probability measure $\mu$ on $G$ such that $\mu(gBh)=0$ for every $g, h\in G$,
\item there exists a Borel probability measure $\mu$ on $G$ such that $\mu$ has compact support and $\mu(gBh)=0$ for every $g, h\in G$,
\item there exists a Borel measure $\mu$ on $G$ such that $0<\mu(X)<\infty$ for some $\mu$-measurable set $X\subseteq G$ and $\mu(gBh)=0$ for every $g, h\in G$,
\item there exists a Borel measure $\mu$ on $G$ such that $0<\mu(C)<\infty$ for some compact set $C\subseteq G$ and $\mu(gBh)=0$ for every $g, h\in G$ (the paper \cite{HSY} calls a Borel set $B$ \emph{shy} if it has this property).
\end{multistmt}
\end{theorem}

\begin{proof}
The implications $(2)\Rightarrow (4)\Rightarrow (3)$ are trivial. $(3)\Rightarrow (1)$ is true, because if $\mu$ and $X$ satisfies the requirements of $(3)$, then $\tilde{\mu}(Y)=\frac{\mu(Y\cap X)}{\mu(X)}$ is a Borel probability measure and $\tilde{\mu} \ll \mu$ means that $\tilde{\mu}(gBh)=0$ for every $g, h\in G$. The equivalence $(1)\Rightarrow (2)$ follows from \autoref{cptsupwitness}.
\end{proof}

Generalizing \cite[Fact 4]{HSY} we may extend this equivalence further:
\begin{theorem}[Hunt-Sauer-Yorke]
For a universally measurable set $B\subseteq G$ the following are equivalent: \begin{multistmt}
\item there exists a Borel probability measure $\mu$ on $G$ such that $\mu(gBh)=0$ for every $g, h\in G$,
\setcounter{enumi}{4}
\item there exists a Borel probability measure $\mu$ on $G$ and a generalized Haar null set $N\subseteq G$ such that $\mu(gBh)=0$ for every $g, h \in G\setminus N$.
\end{multistmt}
\end{theorem}

\begin{proof}
The direction $(1)\Rightarrow(5)$ is trivial. To prove the other direction, assume that $(5)$ holds. We may assume without loss of generality that $N$ is universally measurable. Fix a Borel probability measure $\nu$ on $G$ such that $\nu(gNh)=0$ for all $g, h\in G$ and let $\overline{\nu}$ denote the Borel probability measure $\overline{\nu}(X) = \nu(\{x^{-1} : x\in X\})$.

Consider the measure $\tilde{\mu}$ defined by the convolution
\[\tilde{\mu}(X)=\overline{\nu}*\mu*\overline{\nu}(X) = (\overline{\nu}\times\mu\times\overline{\nu})(\{(p, q, r) \in G^3: pqr \in X\}),\]
which is clearly a Borel probability measure on $G$. To prove that $\tilde{\mu}$ satisfies condition $(1)$ we need to check that if we fix arbitrary $g, h \in G$, then
\[\tilde{\mu}(gBh) = (\overline{\nu}\times\mu\times\overline{\nu})(\{(p, q, r) \in G^3: pqr \in gBh\}) =0.\]
We may apply Fubini's theorem to see that
\begin{align*}
\tilde{\mu}(gBh) =& \int_G\int_G\int_G \chi_{p^{-1}gBhr^{-1}}(q) \intby\overline{\nu}(p)\intby\mu(q)\intby\overline{\nu}(r)=\\
=& \int_G\int_G \mu(p^{-1}gBhr^{-1}) \intby\overline{\nu}(p)\intby\overline{\nu}(r)=\\
=& \int_G\int_G \mu(pgBhr) \intby\nu(p)\intby\nu(r).
\end{align*}
We know that $\nu(h^{-1}N)= \nu(Ng^{-1})= 0$ and so the set $(G\setminus h^{-1}N) \times (G\setminus Ng^{-1})$ has full $(\nu\times\nu)$-measure in $G\times G$. This implies that
\[ \tilde{\mu}(gBh)= \int_{G\setminus h^{-1}N} \int_{G\setminus Ng^{-1}}\mu(pgBhr) \intby\nu(p)\intby\nu(r).\]
But here $p\in G\setminus Ng^{-1}$ and $r\in G\setminus h^{-1}N$ means that $pg, hr\in G\setminus N$ and therefore $\mu(pgBhr)=0$, concluding our proof.

We remark that the proof only used the fact that $N$ is a generalized left-and-right Haar null set. (This is a weaker notion than generalized Haar null sets, we discuss it in \autoref{ssec:leftrighthaarnull}.)
\end{proof}

The following result characterizes Haar null sets with witness functions, analogously to Haar meager sets:
\begin{theorem}[Banakh-G\l\k{a}b-Jab\l o\'nska-Swaczyna]\label{thm:haarnullwitnessfunc}
For a Borel set $B\subseteq G$ the following are equivalent:
\begin{multistmt}
\item $G$ is Haar null,
\item there exists an injective continuous map $f: 2^\omega \to G$ such that $f^{-1}(gBh) \in \mathcal{N}(2^\omega)$ for all $g, h\in G$,
\item there exists a continuous map $f: 2^\omega \to G$ such that $f^{-1}(gBh) \in \mathcal{N}(2^\omega)$ for all $g, h\in G$.
\end{multistmt}
\end{theorem}
Here $\mathcal{N}(2^\omega)$ is the $\sigma$-ideal of sets of Haar measure zero on the Cantor cube $2^\omega$.

The relatively long proof of this result can be found as \cite[Theorem 4.3]{BGJS}. The paper \cite{BGJS} only considers the case of abelian Polish groups, but the proof of this result remains valid in the case when $G$ is not necessarily abelian. For related results stated in this survey, see also \autoref{thm:witnessdomcantor} and \autoref{thm:genhaarnullwitnessfunc}.

The following result gives an equivalent characterization of Haar null sets which allows proving that a Borel set is Haar null by constructing measures that assign small, but not necessarily zero measures to the translates of that set. In \cite[Theorem 1.1]{Mat} Matou\v skov\' a proves this theorem for separable Banach spaces, but her proof can be generalized to work in arbitrary Polish groups.

\begin{theorem}[Matou\v skov\' a]\label{thm:matresult}
A Borel set $B$ is Haar null if and only if for every $\varepsilon>0$ and neighborhood $U$ of $1_G$, there exists a Borel probability measure $\mu$ on $G$ such that the support of $\mu$ is contained in $U$ and $\mu(gBh)<\varepsilon$ for every $g, h \in G$.
\end{theorem}

\begin{proof}
Let $P(G)$ be the set of Borel probability measures on $G$. As $G$ is Polish, \cite[Theorem 17.23]{Ke} states that $P(G)$ (endowed with the weak topology) is also a Polish space. In particular this means that it is possible to fix a compatible metric $d$ such that $(P(G), d)$ is a complete metric space. If $\mu,\nu\in P(G)$, let $(\mu*\nu)(X)=(\mu\times \nu)\left(\{(x, y):xy\in X\}\right)$ be their convolution. It is straightforward to see that $*$ is associative (but not commutative in general, as we did not assume that $G$ is commutative). The map $ * : P(G)\times P(G)\to P(G)$ is continuous, for a proof of this see e.g.\ \cite[Proposition 2.3]{HM}. Let $\delta(X)=1$ if $1_G\in X$, and $\delta(X)=0$ if $1_G\notin X$, then it is clear that $\delta\in P(G)$ is the identity element for $*$. 

First we prove the \enquote{only if} part. Let $(U_n)_{n\in\omega}$ be open sets with $\bigcap_n U_n=\{1_G\}$. For every $n\in\omega$ fix a Borel probability measure $\mu_n$ such that
\begin{propertylist}
\item $\supp\mu_n\subseteq U_n$ and 
\item $\mu_n(gBh)<\frac{1}{n+1}$ for every $g, h\in G$.
\end{propertylist}

It is easy to see from property (I) that the sequence $(\mu_n)_{n\in\omega}$ (weakly) converges to $\delta$, and this and the continuity of $*$ means that for any $\nu\in P(G)$ the sequence $d(\nu, \nu*\mu_n)$ converges to zero. This allows us to replace $(\mu_n)_{n\in\omega}$ with a subsequence which also satisfies that $d(\nu, \nu*\mu_n)<2^{-n}$ for every measure $\nu$ from the finite set \[\{\mu_{j_0}*\mu_{j_1}*\ldots*\mu_{j_{r}} : r< n \text{ and }0\le j_0<j_1<\ldots<j_{r}<n\}.\]
(Notice that property (II) clearly remains true for any subsequence.) Using this assumption and the completeness of $(P(G), d)$ we can define (for every $n\in\omega$) the \enquote{infinite convolution} $\mu_{n}*\mu_{n+1} * \ldots $ as the limit of the Cauchy sequence $(\mu_{n}*\mu_{n+1}*\mu_{n+2}*\ldots *\mu_{n+j})_{j\in\omega}$. We will show that the choice $\mu=\mu_0*\mu_1*\ldots $ witnesses that $B$ is Haar null.

We have to prove that $\mu(gBh)=0$ for every $g, h\in G$. To show this fix arbitrary $g, h\in G$ and $n\in\omega$; we will show that $\mu(gBh)\le\frac{1}{n+1}$. Let $\alpha_n=\mu_0*\mu_1*\ldots*\mu_{n-1}$ and $\beta_n=\mu_{n+1}*\mu_{n+2}*\ldots$ and notice the continuity of $*$ yields that
\begin{align*}
&\mu=\lim_{j\to\infty} (\alpha_n * \mu_n * (\mu_{n+1}*\mu_{n+2}*\ldots*\mu_{n+j}))= \\
&\quad=\alpha_n*\mu_n * \lim_{j\to\infty} (\mu_{n+1}*\mu_{n+2}*\ldots*\mu_{n+j})=\alpha_n*\mu_n*\beta_n.
\end{align*}
This means that
\begin{align*}
&\mu(gBh)=(\alpha_n*\mu_n*\beta_n)(gBh)=(\alpha_n\times\mu_n\times\beta_n)(\{(x,y,z)\in G^3: xyz\in gBh\})=\\
&\quad=((\alpha_n\times\beta_n)\times\mu_n)(\{((x,z),y)\in G^2\times G : y\in x^{-1}gBhz^{-1}\}).
\end{align*}
Notice that for every $x, z\in G$ property (II) yields that $\mu_n(x^{-1}gBhz^{-1})<\frac{1}{n+1}$. Applying Fubini's theorem in the product space $G^2\times G$ to the product measure $(\alpha_n\times\beta_n)\times\mu_n$ yields that $\mu(gBh)\le \frac{1}{n+1}$.

To prove the \enquote{if} part of the theorem, suppose that there exists a $\varepsilon>0$ and a neighborhood $U$ of $1_G$ such that for every Borel probability measure $\mu$ on $G$ if $\supp\mu\subseteq U$, then $\mu(gBh)\ge\varepsilon$ for some $g,h\in G$. Let $\mu$ be an arbitrary Borel probability measure. Applying \autoref{lem:posmeascpt} yields that there are a compact set $C\subseteq G$ and $c\in G$ with $\mu(C)>0$ and $C\subseteq cU$. Define $\mu'(X)=\frac{\mu(cX\cap C)}{\mu(C)}$, then $\mu'$ is a Borel probability measure with $\supp\mu'\subseteq U$, hence $\mu'(gBh)\ge\varepsilon$ for some $g,h\in G$. This means that $\mu(cgBh)\neq 0$, so $\mu$ is not a witness measure for $B$, and because $\mu$ was arbitrary, $B$ is not Haar null.

\end{proof}

For Haar meagerness the following analogue of \autoref{thm:haarnulleqv} holds:

\begin{theorem}\label{thm:haarmeagereqv}
For a set $A\subseteq G$ the following are equivalent:
\begin{multistmt}
\item there exists a Borel set $B\supseteq A$, a (nonempty) compact metric space $K$ and a continuous function $f : K\to G$ such that $f^{-1}(gBh)$ is meager in $K$ for every $g, h \in G$ (i.e. $A$ is Haar meager),
\item there exists a Borel Haar meager set $B\supseteq A$,
\item there exists an analytic set $B\supseteq A$, a (nonempty) compact metric space $K$ and a continuous function $f : K\to G$ such that $f^{-1}(gBh)$ is meager in $K$ for every $g, h \in G$.
\end{multistmt}
\end{theorem}

\begin{proof}
$(1) \Leftrightarrow (2)$ is trivial from \autoref{def:haarmeager}, $(1) \Rightarrow (3)$ follows from the fact that all Borel sets are analytic. Finally, the implication $(3) \Rightarrow (1)$ can be found as \cite[Proposition 8]{Do}, and the proof is a straightforward analogue of the proof of $(5)\Rightarrow (1)$ in \autoref{thm:haarnulleqv}.

Without loss of generality we may assume that the set $A$ itself is analytic and satisfies that $f^{-1}(gAh)$ is meager in $K$ for every $g, h \in G$ for some (nonempty) compact metric space $K$ and continuous function $f : K \to G$. We will prove that (for this $K$ and $f$) there exists a Borel set $B\supseteq A$ such that $f^{-1}(gBh)$ is meager in $K$ for every $g, h\in G$.

\begin{claim} The family of sets
\[\Phi = \{X \subseteq G : X\text{{\upshape\ is analytic and $f^{-1}(gXh)$ is meager in $K$ for every $g,h \in G$}}\}\]
is \emph{coanalytic on analytic}, that is, for every Polish space $Y$ and $P \in \mathbf{\Sigma}^1_1(Y\times G)$, the set $\{y \in Y : P_y\in \Phi \}$ is $\mathbf{\Pi}^1_1$.
\end{claim}

\begin{proof}
Let $Y$ be a Polish space and $P \in \mathbf{\Sigma}^1_1(Y\times G)$ and let 
\[\tilde{P} =\{(g, h, y, k) \in G\times G\times Y\times K : f(k) \in gP_yh\}.\]
Then $\tilde{P}$ is analytic, as it is the preimage of $P$ under $(g, h, y, k)\mapsto (y, g^{-1}f(k)h^{-1})$. Novikov's theorem (see e.g.\ \cite[Theorem 29.22]{Ke}) states that if $U$ and $V$ are Polish spaces and $A\subseteq U\times V$ is analytic, then $\{u\in U : A_u \text{ is not meager in }V\}$ is analytic. This yields that $\{(g, h, y) : \tilde{P}_{(g, h, y)}\text{ is meager in }K\}$ is coanalytic, but then
\[\{y\in Y : \tilde{P}_{(g, h, y)}\text{ is meager in $K$ for every $g, h\in G$}\}=\{y\in Y : P_y\in \Phi\}\]
is also coanalytic.
\end{proof}
Now, since $A\in \Phi$, by the dual form of the First Reflection Theorem (see \cite[Theorem 35.10 and the remarks following it]{Ke}) there exists a Borel set $B$ with $B\supseteq A$ and $B\in \Phi$, and this $B$ satisfies our requirements.
\end{proof}

To prove our next result we will need a technical lemma. This is a modified version of the well-known result that for every (nonempty) compact metric space $K$, there exists a continuous surjective map $\varphi: 2^\omega \to K$.

\begin{lemma}\label{lem:meagerpullback}
If $(K, d)$ is a (nonempty) compact metric space, then there exists a continuous function $\varphi : 2^\omega \to K$ such that if $M$ is meager in $K$, then $\varphi^{-1}(M)$ is meager in $2^\omega$. 
\end{lemma}
\begin{proof}
We will use a modified version of the usual construction. Note that if $\diam(K)<1$, then this function $\varphi$ will be surjective (we will not need this fact).

Let $2^{<\omega}=\bigcup_{n\in\omega} 2^n$ be the set of finite 0-1 sequences. The length of a sequence $s$ is denoted by $\lvert s \rvert$ and the sequence of length zero is denoted by $\emptyset$. If $s$ and $t$ are sequences (where $t$ may be infinite), $s\preceq t$ means that $s$ is an initial segment of $t$ (i.e.\ the first $\lvert s \rvert$ elements of $t$ form the sequence $s$) and $s\prec t$ means that $s\preceq t$ and $s\neq t$. For sequences $s, t$ where $\lvert s \rvert$ is finite, $s\concat t$ denotes the concatenation of $s$ and $t$. For a finite sequence $s$, let $[s] = \{x\in 2^\omega :s\preceq x\}$ be the set of infinite sequences starting with $s$. Note that $[s]$ is clopen in $2^\omega$ for every $s\in 2^{<\omega}$.

We will choose a set of finite 0-1 sequences $S\subseteq 2^{<\omega}$, and for every $s\in S$ we will choose a point $k_s\in K$. The construction of $S$ will be recursive: we recursively define for every $n\in \omega$ a set $S_n$ and let $S=\dot\bigcup_{n\in\omega} S_n$. Our choices will satisfy the following properties:
\begin{propertylist}
\item for every $n\in \omega$ and $x\in 2^\omega$ there exists a unique $s=s(x,n)\in S_n$ with $s\preceq x$, and $s(x,n)\prec s(x, n')$ if $n< n'$.
\item for every $n\in \omega$ and $x\in 2^\omega$ the set $C_{x, n}=\bigcap_{0\le j<n} \overline{B}\left(k_{s(x,j)}, \frac{1}{{j}+1}\right)$ is nonempty.
\end{propertylist}

First we let $S_0=\{\emptyset\}$ and choose an arbitrary $k_\emptyset\in K$, then these trivially satisfy (I) and (II). 

Suppose that we already defined $S_0, S_1, \ldots, S_{n-1}$ and let $s\in S_{n-1}$ be arbitrary. Notice that for $x,x'\in[s]$, $C_{x,n-1}=C_{x',n-1}$ and denote this common set with $C_s$. The set $C_s$ is (nonempty) compact and it is covered by the open sets $\{B\left(c, \frac{1}{{n}}\right): c\in C_s\}$, hence we can select a collection $(c^{(s)}_j)_{j\in I_s}$ where $I_s$ is a finite index set such that $\bigcup_{j\in I_s} B\left(c^{(s)}_j, \frac{1}{{n}}\right) \supseteq C_s$. We may increase the cardinality of this collection by repeating one element several times if necessary, and thus we can assume that the index set $I_s$ is of the form $2^{\ell_s}$ for some integer $\ell_s\ge 1$ (i.e.\ it consists of the 0-1 sequences with length $\ell_s$). Now we can define $S_n=\{s\concat t : s\in S_{n-1}, t\in 2^{\ell_s}\}$. If $s'\in S_n$, then there is a unique $s\in S_{n-1}$ such that $s\preceq s'$, if $t$ satisfies that $s'= s\concat t$ (i.e. $t$ is the final segment of $s'$), then let $k_{s'}= c^{(s)}_t$.

It is straightforward to check that these choices satisfy (I). Property (II) is satisfied because if $x\in 2^\omega$, and $s$, $s'$ and $t$ are the sequences that satisfy $s=s(x,n-1)$ and $s\concat t = s' = s(x,n)$, then $k_{s'}=c^{(s)}_t\in C_{x,n}$, because it is contained in both $C_{x,n-1}=C_s$ and $\overline{B}\left(k_{s'}, \frac{1}{{n}}\right)$, and this shows that $C_{x, n}$ is not empty.

We use this construction to define the function $\varphi$: let $\{\varphi(x)\}=\bigcap_{n\in\omega}C_{x, n}$ (as the system of nonempty compact sets $(C_{x,n})_{n\in\omega}$ is descending and $\diam(C_{x, n})\le \diam\left(B\left(k_{s(x,n-1)}, \frac{1}{{n}}\right)\right)\le 2 \cdot \frac{1}{{n}}\rightarrow 0$, the intersection of this system is indeed a singleton). This function is continuous, because if $\varepsilon>0$ and $\varphi(x) = k\in K$, then $2\cdot \frac{1}{{n_0}+1}<\varepsilon$ for some $n_0$, and then for every $x'$ in the clopen set  $[s(x, n_0)]$ the set $\overline{B}\left(k_{s(x,n_0)}, \frac{1}{{n_0}+1}\right)=\overline{B}\left(k_{s(x',n_0)}, \frac{1}{{n_0}+1}\right)$ has diameter $<\varepsilon$ and contains both $\varphi(x)$ and $\varphi(x')$.

Now we prove that if $U\subseteq 2^\omega$ is open, then $\varphi(U)$ contains an open set $V$. It is clear from (I) that $\{[s] : s\in S\}$ is a base of the topology of $2^\omega$. This means that there exist a $n\in \omega\setminus \{0\}$ and $s\in S_{n-1}$ satisfying $[s]\subseteq U$. Let $V=\bigcap_{0\le j< n} B\left(k_{t_j}, \frac{1}{{j}+1}\right)$ where $t_j=s(x, j)$ for an arbitrary $x\in [s]$ (this is well-defined as $t_j\in S_j$ is the only element of $S_j$ with $s(x, j)\preceq s(x, n-1)=s$). It is clear that $V$ is open and $V\subseteq C_{x, n-1}=C_s$ (where $x\in [s]$ is arbitrary, this is again well-defined), to show that $V\subseteq \varphi(U)$ let $v\in V$ be arbitrary. Define $u_0\in 2^{\ell_s}$ such that $v\in B\left(k_{s\concat u_0}, \frac{1}{{n}}\right)$ (this is possible as $\{B\left(k_{s\concat u_0}, \frac{1}{{n}}\right) : u_0\in 2^{\ell_s}\}$ was a cover of $C_{x, n-1}=C_s$). Repeat this to define $u_1\in 2^{\ell_{s\concat u_0}}$ such that $v\in B\left(k_{s\concat u_0\concat u_1}, \frac{1}{{n}+1}\right)$, then $u_2\in 2^{\ell_{s\concat u_0\concat u_1}}$ such that $v\in B\left(k_{s\concat u_0\concat u_1\concat u_2}, \frac{1}{n+2}\right)$ etc.\ and let $y$ be the infinite sequence $y=s\concat u_0\concat u_1\concat u_2\concat \ldots$. It is easy to see that $\varphi(y)=v$, as $d(\varphi(y), v)\le 2\cdot \frac{1}{{n}+1}$ for every $n\in\omega$.

Finally we show that if $M\in \mathcal{M}(K)$ is arbitrary, then $\varphi^{-1}(M)\in\mathcal{M}(2^\omega)$. As $M$ is meager $M\subseteq \bigcup_{n\in\omega} F_n$ for a system $(F_n)_{n\in\omega}$ of nowhere dense closed sets. If $\varphi^{-1}(M)$ is not meager, then $\varphi^{-1}(F_n)$ contains an open set for some $n \in\omega$. But then $F_n$ contains an open set, which is a contradiction.
\end{proof}

We use this lemma to show that in \autoref{def:haarmeager} we can also restrict the choice of the compact metric space $K$. In the following theorem the equivalence $(1)\Leftrightarrow (2)$ is \cite[Proposition 3]{DV} and the equivalence $(1)\Leftrightarrow (3)$ is \cite[Theorem 2.11]{Da}.

\begin{theorem}[Dole{\v z}al-Vlas{\'a}k / Darji]\label{thm:witnessdomcantor}
For a Borel set $B\subseteq G$ the following are equivalent:
\begin{multistmt}
\item there exists a (nonempty) compact metric space $K$ and a continuous function $f : K\to G$ such that $f^{-1}(gBh)$ is meager in $K$ for every $g, h \in G$ (i.e. $B$ is Haar meager),
\item there exists a continuous function $f : 2^\omega \to G$ such that $f^{-1}(gBh)$ is meager in $2^\omega$ for every $g, h \in G$,
\item there exists a (nonempty) compact set $C \subseteq G$, a continuous function $f : C \to G$ such that $f^{-1}(gBh)$ is meager in $C$ for every $g, h \in G$,
\end{multistmt}
\end{theorem}

\begin{proof}
$(1)\Rightarrow (2)$:\\
This implication is an easy consequence of \autoref{lem:meagerpullback}. If $K$ and $f$ satisfies the requirements of (1) and $\varphi$ is the function granted by \autoref{lem:meagerpullback}, then $\tilde{f} = f\circ \varphi : 2^\omega\to G$ will satisfy the requirements of (2), because it is continuous and for every $g, h\in G$ the set $f^{-1}(gBh)$ is meager in $K$, hence $\varphi^{-1}(f^{-1}(gBh))=\tilde{f}^{-1}(gBh)$ is meager in $2^\omega$.

$(2)\Rightarrow (3)$:\\
If $G$ is countable, the only Haar meager subset of $G$ is the empty set. In this case, any nonempty $C\subseteq G$ and continuous function $f: C\to G$ is sufficient. If $G$ is not countable, then it is well known that there is a (compact) set $C\subseteq G$ that is homeomorphic to $2^\omega$. Composing the witness function $f: 2^\omega\to G$ granted by (2) with this homeomorphism yields a function that satisfies our requirements (together with $C$).

$(3)\Rightarrow(1)$:\\
This implication is trivial.
\end{proof}

\subsection{Coanalytic hulls}
In \autoref{thm:haarnulleqv} and \autoref{thm:haarmeagereqv} we proved that the Borel hull in the definition of Haar null sets and Haar meager sets can be replaced by an analytic hull. The following theorems show that it cannot be replaced by a coanalytic hull. As the proofs of these theorems are relatively long, we do not reproduce them here. Note that these theorems were proved in the abelian case, but they can be generalized to the case when $G$ is TSI.

In the case of Haar null sets, the paper \cite{EVGdhull} proves the following result:
\begin{theorem}[Elekes-Vidny\' anszky]\label{thm:haarnullcoanal}
If $G$ is a non-locally-compact abelian Polish group, then  there exists a coanalytic set $A\subseteq G$ that is not Haar null, but there is a Borel probability measure $\mu$ on $G$ such that $\mu(gAh)=0$ for every $g,h\in G$.
\end{theorem}

In particular, this set $A$ is generalized Haar null, but not Haar null:

\begin{corollary}
If $G$ is non-locally-compact and abelian, then $\mathcal{GHN}(G) \supsetneqq \mathcal{HN}(G)$.
\end{corollary}

This (and its generalization to TSI groups) is the best known result for \cite[Question 5.4]{EVGdhull}, which asks the following:
\begin{question}[Elekes-Vidny\' anszky]\label{que:ghngthn} Is $\mathcal{GHN}(G) \supsetneqq \mathcal{HN}(G)$ in all non-locally-compact Polish groups?
\end{question}

The Haar meager case is answered by \cite[Theorem 13]{DV}:
\begin{theorem}[Dole\v zal-Vlas\' ak]\label{thm:haarmeagercoanal}
If $G$ is a non-locally-compact abelian Polish group, then there exists a coanalytic set $A\subseteq G$ that is not Haar meager, but there is a (nonempty) compact metric space $K$ and a continuous function $f: K \to G$ such that $f^{-1}(gAh)$ is meager in $K$ for every $g,h\in G$.
\end{theorem}

\subsection{Naive versions}
It is possible to eliminate the Borel/universally measurable hull from our definitions completely. Unfortunately, the resulting \enquote{naive} notions will not share the nice properties of a notion of smallness. The results in this subsection will show some of these problems. Most of these counterexamples are provided only in special groups or as a corollary of the Continuum Hypothesis, as even these \enquote{weak} results are enough to show that these notions are not very useful.

\begin{definition}
A set $A\subseteq G$ is called \emph{naively Haar null} if there is a Borel probability measure $\mu$ on $G$ such that $\mu(gAh)=0$ for every $g, h \in G$.
\end{definition}

\begin{definition}
A set $A\subseteq G$ is called \emph{naively Haar meager} if there is a (nonempty) compact metric space $K$ and a continuous function $f : K\to G$ such that $f^{-1}(gAh)$ is meager in $K$ for every $g, h\in G$.
\end{definition}

The following two examples show that in certain groups the whole group is the union of countably many sets which are both naively Haar null and naively Haar meager, and hence neither the system of naively Haar null sets, nor the system of naively Haar meager sets is a $\sigma$-ideal.

\begin{example}[Dougherty \cite{Doug}]\label{exa:twonaives}
Let $G$ be an uncountable Polish group. Assuming the Continuum Hypothesis, there exists a subset $W$ in the product group $G\times G$ such that both $W$ and $(G\times G)\setminus W$ are naively Haar null and naively Haar meager.
\end{example}

\begin{proof}
Let $<_W$ be a well-ordering of $G$ in order type $\omega_1$, and let $W=\{(g, h)\in G\times G: g<_W h\}$ be this relation considered as a subset of $G\times G$.

Let $\mu_1$ be an non-atomic measure on $G$ and $\mu_2$ be a measure on $G$ that is concentrated on a single point. It is clear that $(\mu_1\times \mu_2)(gWh)=0$ for every $g, h\in G$ (as $gWh\cap \supp (\mu_1\times\mu_2)$ is countable). Similarly $(\mu_2\times \mu_1)(g((G\times G)\setminus W)h)=0$ for every $g, h\in G$, but these mean that $W$ and $(G\times G)\setminus W$ are both naively Haar null.

Let $K\subseteq G$ be a nonempty perfect compact set (it is well known that such set exists) and let $f_1, f_2 : K\to G$ be the continuous functions $f_1(k)=(k, 1_G)$, $f_2(k)=(1_G, k)$ (here $1_G$ could be replaced by any fixed element of $G$). Then $f_1^{-1}(gWh)$ is countable (hence meager) for every $g, h\in G$, so $W$ is naively Haar meager, similarly $f_2$ shows that $(G\times G)\setminus W$ is also naively Haar meager.
\end{proof}

\begin{example}
The Polish group $(\mathbb{R}^2,+)$ is the union of countably many sets that are both naively Haar null and naively Haar meager.
\end{example}

\begin{proof}
The paper \cite{Dav} constructs a decomposition of the plane with the following properties:
\begin{theorem}[Davies]
Suppose that $(\theta_i)_{i\in\omega}$ is a countably infinite system of directions, such that $\theta_i$ and $\theta_j$ are not parallel if $i\neq j$. Then the plane can be decomposed as $\mathbb{R}^2=\dot{\bigcup}_{i\in \omega} S_i$ such that each line in the direction $\theta_i$ intersects the set $S_i$ in at most one point.
\end{theorem}
We do not include the proof of this theorem.

Later \cite[Example 5.4]{EVnaivhnull} notices that the sets $S_i$ ($i\in\omega$) in this decomposition are naively Haar null: Let $\mu_i$ be the 1-dimensional Lebesgue measure on an arbitrary line with direction $\theta_i$. Then $\mu_i(g+S_i+h)=0$ for every $g, h\in \mathbb{R}^2$, because at most one point of $g+S_i+h$ is contained in the support of $\mu_i$. This means that $S_i$ is indeed naively Haar null.

A similar argument shows that these sets are also naively Haar meager: Let $K_i$ be a nonempty perfect compact subset of an arbitrary line with direction $\theta_i$, and let $f_i : K_i\to \mathbb{R}^2$ be the restriction of the identity function. Then $f_i^{-1}(g+S_i+h)$ contains at most one point for every $g, h\in \mathbb{R}^2$, hence it is meager. This proves that $S_i$ is naively Haar meager.
\end{proof}

The following theorem shows that in a relatively general class of groups the naive notions are strictly weaker than the corresponding \enquote{canonical} notions.
\begin{theorem}[Elekes-Vidny\' anszky / Dole\v zal-Rmoutil-Vejnar-Vlas\' ak]
Let $G$ be an uncountable abelian Polish group.
\begin{multistmt}
\item There exists a subset of $G$ that is naively Haar null but not Haar null.
\item There exists a subset of $G$ that is naively Haar meager but not Haar meager.
\end{multistmt}
\end{theorem}

We do not reproduce the relatively long proof of these results. The proof of (1) can be found in \cite[Theorem 1.3]{EVnaivhnull}, the proof of (2) can be found in \cite[Theorem 16]{Do}. Also note that when $G$ is non-locally-compact, these results are corollaries of \autoref{thm:haarnullcoanal} and \autoref{thm:haarmeagercoanal}.

The following example from \cite[Proposition 17]{Do} yields the results of \autoref{exa:twonaives} in a different class of groups. The cited paper only proves this for the naively Haar meager case, but states that it can be proved analogously in the naively Haar null case. We do not reproduce this proof, as it is significantly longer than the proof of \autoref{exa:twonaives}.

\begin{example}[Dole\v zal-Rmoutil-Vejnar-Vlas\' ak]
Let $G$ be an uncountable abelian Polish group. Assuming the Continuum Hypothesis, there exists a subset $X\subseteq G$ such that both $X$ and $G\setminus X$ are naively Haar null and naively Haar meager.
\end{example}

\subsection{Left and right Haar null sets}\label{ssec:leftrighthaarnull}

When we defined Haar null sets in \autoref{def:haarnull}, we used multiplication from both sides by arbitrary elements of $G$. If we replace this by multiplication from one side, we get the following notions:

\begin{definition}\label{def:lefthaarnull}
A set $A \subseteq G$ is said to be \emph{left Haar null} (or \emph{right Haar null}) if there are a Borel set $B\supseteq A$ and a Borel probability measure $\mu$ on $G$ such that $\mu(gB)=0$ for every $g\in G$ ($\mu(Bg)=0$ for every $g\in G$).
\end{definition}

\begin{definition}\label{def:leftrighthaarnull}
A set $A \subseteq G$ is said to be \emph{left-and-right Haar null} if there are a Borel set $B\supseteq A$ and a Borel probability measure $\mu$ on $G$ such that $\mu(gB)=\mu(Bg)=0$ for every $g\in G$.
\end{definition}

If \enquote{Borel set} is replaced by \enquote{universally measurable set}, we can naturally obtain the generalized versions of these notions. As the papers about this topic happen to follow Christensen in defining \enquote{Haar null set} to mean generalized Haar null set in the terminology of our paper, most results in this subsection were originally stated for these generalized versions.

There are situations where these \enquote{one-sided} notions are more useful than the usual, symmetric one. (For example \cite{SolAmen} uses them to state results related to automatic continuity and the generalizations of the Steinhaus theorem. We mention some of these results in \autoref{ssec:steinhaus}.) However in some groups these are not \enquote{good} notions of smallness in the sense of \autoref{ssec:smallness}. 

Translation invariance is not problematic even for the \enquote{one-sided} notions: Suppose that e.g.\ $B$ is Borel left Haar null and a Borel probability measure $\mu$ satisfies $\mu(gB)=0$ for every $g\in G$. If we consider a right translate $Bh$ (this is the interesting case, invariance under left translation is trivial), then $\mu'(X)=\mu(Xh^{-1})$ is a Borel probability measure which satisfies $\mu'(g\cdot Bh)=0$ for every $g\in G$. An analogous argument shows that the system of right Haar null sets is translation invariant; using \autoref{pro:leftandrighthaarnull} it follows from these that the system of left-and-right Haar null sets is also translation invariant. It is clear that this reasoning works for the generalized versions, too.

Unfortunately there are Polish groups where these notions fail to form $\sigma$-ideals. In \cite{SolAmen} Solecki gives a sufficient condition which guarantees that the generalized left Haar null sets form a $\sigma$-ideal and gives another sufficient condition which guarantees that the left Haar null sets do not form a $\sigma$-ideal. We state these results and show examples of groups satisfying these conditions without proofs:

\begin{definition}[Solecki]\label{def:amenone}
A Polish group $G$ is called \emph{amenable at 1} if for any sequence $(\mu_n)_{n\in\omega}$ of Borel probability measures on $G$ with $1_G\in\supp\mu_n$, there are Borel probability measures $\nu_n$ and $\nu$ such that
\begin{propertylist}
\item $\nu_n\ll\mu_n$,
\item if $K\subseteq G$ is compact, then $\lim_n (\nu *\nu_n) (K) =\nu(K)$.
\end{propertylist}
\end{definition}

This class is closed under taking closed subgroups and continuous homomorphic images. A relatively short proof also shows that (II) is equivalent to
\begin{propertylist}
\item[(II')] if $K\subseteq G$ is compact, then $\lim_n (\nu_n * \nu) (K) =\nu(K)$.
\end{propertylist}
This means that if $G$ is amenable at 1, then so is the opposite group $G^{\mathrm{opp}}$. ($G^{\mathrm{opp}}$ is the set $G$ considered as a group with $(x, y)\mapsto y\cdot x$ as the multiplication.)

Examples of groups which are amenable at 1 include:
\begin{multistmt}
\item abelian Polish groups,
\item locally compact Polish groups,
\item countable direct products of locally compact Polish groups such that all but finitely many factors are amenable,
\item inverse limits of sequences of amenable, locally compact Polish groups with continuous homomorphisms as bonding maps.
\end{multistmt}

\begin{theorem}[Solecki]\label{thm:amenonesigmaid}
If $G$ is amenable at 1, then the generalized left Haar null sets form a $\sigma$-ideal.
\end{theorem}

Our note about the opposite group (and the fact that the intersection of two $\sigma$-ideals is also a $\sigma$-ideal) means that generalized right Haar null sets and generalized left-and-right Haar null sets also form $\sigma$-ideals.

The following definition is the sufficient condition for the \enquote{bad} case. Note that this condition is also symmetric (in the sense that if $G$ satisfies it, then $G^{\mathrm{opp}}$ also satisfies it).

\begin{definition}[Solecki]\label{def:freesgone}
A Polish group $G$ is said to \emph{have a free subgroup at 1} if it has a non-discrete free subgroup whose all finitely generated subgroups are discrete.
\end{definition}

The paper \cite{SolAmen} lists several groups which all have a free subgroup at 1, we mention some of these:
\begin{multistmt}
\item countably infinite products of Polish groups containing discrete free non-Abelian subgroups,
\item $S_\infty$, the group of permutations of $\mathbb{N}$ with the topology of pointwise convergence,
\item $\mathrm{Aut}(\mathbb{Q},\le)$, the group of order-preserving self-bijections of the rationals with the topology of pointwise convergence on $\mathbb{Q}$ viewed as discrete (i.e.\ a sequence $(f_n)_{n\in\omega}\in G^\omega$ is said to be convergent if for every $q\in \mathbb{Q}$ there is a $n_0\in\omega$ such that the sequence $(f_n(q))_{n\ge n_0}$ is constant),
\end{multistmt}

\begin{theorem}[Solecki]\label{thm:freesgnotsigmaid}
If $G$ has a free subgroup at 1, then there are a Borel left Haar null set $B\subseteq G$ and $g\in G$ such that $B\cup Bg=G$. As the left Haar null sets are translation invariant and $G$ is not left Haar null, this means that they do not form an ideal.
\end{theorem}
As the notion of having a free subgroup at 1 is symmetric, the same is true for right Haar null sets.

In abelian groups it is trivial that the notions introduced by \autoref{def:lefthaarnull} and \autoref{def:leftrighthaarnull} are all equivalent to the notion of Haar null sets (and the generalized versions are equivalent to generalized Haar null sets).

The following remark is also trivial, but sometimes useful:
\begin{remark}
If $A\subseteq G$ is a conjugacy invariant set (that is, $gAg^{-1}=A$ for every $g\in G$), then $gAh=gh\cdot h^{-1}Ah=ghA$ (and similarly $gAh=Agh$) for every $g, h\in G$, hence if $A$ is left (or right) Haar null, then $A$ is Haar null.
\end{remark}

For locally compact groups \autoref{thm:loccptequivnotions} and the remark in its proof shows that all generalized right Haar null sets are Haar null. This implies that in locally compact groups the eight variants of Haar null (generalized or not, left or right or left-and-right or \enquote{plain}) are all equivalent to having Haar measure zero.

The following diagram summarizes the situation in the general case:

\begin{tikzpicture}[implic/.style={double equal sign distance,-{Implies[]},every to/.append style={bend right=10}}]
\node at (0,0) (hn) {Haar null};
\node[left=3em of hn,align=center] (lar) {left-and-right\\Haar null};
\node[left=3em of lar,align=center] (lr) {left Haar null and\\right Haar null};
\node[coordinate,left=5em of lr] (lrmid) {};
\node[above=1em of lrmid,align=center] (lhn) {left Haar null};
\node[below=1em of lrmid,align=center] (rhn) {right Haar null};
\draw[implic,transform canvas={yshift=1ex}] (hn.west) to (lar.east);
\draw[implic,negated,transform canvas={yshift=-1ex}] (lar.east) to node[below=2pt] {(3)} (hn.west);
\draw[implic,transform canvas={yshift=1ex}] (lar.west) to (lr.east);
\draw[implic,transform canvas={yshift=-1ex}] (lr.east) to node[below=2pt] {(2)} (lar.west);
\draw[implic,negated] (lhn.east) to[bend left] node[auto] {(1)} (lr);
\draw[implic,negated] (rhn.east) to node[auto,'] {(1)} (lr);
\draw[implic] (lr.west)+(0,0.5em) to[bend left] (lhn);
\draw[implic] (lr.west)+(0,-0.5em) to (rhn);
\draw[implic,negated,transform canvas={xshift=0.5em}] (lhn.south) -- node[left=1.25em,anchor=center] {(1)} (rhn.north);
\draw[implic,negated,transform canvas={xshift=-2em}] (rhn.north) -- (lhn.south);
\end{tikzpicture}

(Here $\nRightarrow$ means that there is a counterexample in a suitable Polish group and $\Rightarrow$ means that the implication is true in all Polish groups.)

The not-implications marked by (1) are demonstrated by the already stated \autoref{thm:freesgnotsigmaid}. We will prove the implication (2) as \autoref{pro:leftandrighthaarnull} and demonstrate the not-implication (3) by \autoref{exa:notHNleftrightHN}.

\begin{proposition}\label{pro:leftandrighthaarnull}
A set $A\subseteq G$ is left-and-right Haar null if and only if it is both left Haar null and right Haar null.
\end{proposition}

This simple result was mentioned in \cite{ST}; it can be proved by modifying the proof of \cite[Theorem 2]{Myc}.

\begin{proof}
We only have to show that if $A$ is both left Haar null and right Haar null, then it is left-and-right Haar null, as the other direction is trivial. By definition there exist a Borel set $B\supseteq A$ and Borel probability measures $\mu_1, \mu_2$ on $G$ such that $\mu_1(gB)=\mu_2(Bg)=0$ for every $g\in G$.

Define \[\mu(X)=(\mu_1\times \mu_2)\left(\{(x, y)\in G^2 : yx\in X\}\right),\]
then $\mu$ is clearly a Borel probability measure and if the characteristic function of a set $S$ is denoted by $\chi_S$, then using Fubini's theorem we have
\[\mu(gB)=\int_G\int_G \chi_{gB}(yx)\intby\mu_1(x)\intby\mu_2(y)=\int_G \mu_1(y^{-1}gB)\intby\mu_2(y)=0\]
and
\[\mu(Bg)=\int_G\int_G \chi_{Bg}(yx)\intby\mu_2(y)\intby\mu_1(x)=\int_G \mu_2(Bgx^{-1})\intby\mu_1(x)=0\]
and these show that $\mu$ satisfies the requirements of \autoref{def:leftrighthaarnull}.
\end{proof}

The following two counterexamples appear in \cite{ST}. This paper uses relatively elementary techniques and examines the structure of the group
\[\mathcal{H}[0,1]=\{f : f \text{ is continuous, strictly increasing}, f(0)=0, f(1)=1\}.\]
(This is the group of order-preserving self-homeomorphisms of $[0,1]$, the group operation is composition, the topology is the compact-open topology.) We state these results without proofs:

\begin{example}[Shi-Thomson]\label{exa:notHNleftrightHN}
The group $\mathcal{H}[0,1]$ has a Borel subset that is left-and-right Haar null but not Haar null.
\end{example}

\begin{example}[Shi-Thomson]
There exists a Borel set $B\subseteq \mathcal{H}[0,1]$ and a Borel probability measure $\mu$ such that $\mu(Bg)=0$ for every $g\in \mathcal{H}[0,1]$ (this implies that $B$ is right Haar null), but $\mu(gB)\neq 0$ for some $g\in \mathcal{H}[0,1]$ (i.e. $\mu$ does not witness that $B$ is left Haar null).
\end{example}

When the Polish group $G$ has a free subgroup at 1, then \autoref{thm:freesgnotsigmaid} implies that the left Haar null sets are distinct from the right Haar null sets; \cite[Question 5.1]{SolAmen} asks if the opposite of this is true for groups that are amenable at 1:
\begin{question}[Solecki]\label{que:lhneqrhnamenone}
Let $G$ be an amenable at 1 Polish group. Do left Haar null subsets of $G$ coincide with right Haar null subsets of $G$?
\end{question}

We conclude this part by stating \cite[Theorem 6.1]{SolLR}, which provides a necessary and sufficient condition for the equivalence of the notions generalized left Haar null and generalized Haar null in a special class of groups.
\begin{theorem}[Solecki]\label{thm:soleckiLRresult}
Let $H_n$ ($n\in \omega$) be countable groups and consider the group $G=\prod_n H_n$. The following conditions are equivalent:
\begin{multistmt}
\item In $G$ the system of generalized left Haar null sets is the same as the system of generalized Haar null sets.
\item For each universally measurable set $A\subseteq G$ that is not generalized Haar null, $1_{G}\in \Int(AA^{-1})$.
\item For each closed set $F\subseteq G$ that is not (generalized) Haar null, $FF^{-1}$ is dense in some non-empty open set.
\item For all but finitely many $n\in\omega$ all elements of $H_n$ have finite conjugacy classes in $H_n$, that is, for all but finitely many $n\in\omega$ and for all $x\in H_n$ the set $\{yxy^{-1} : y\in H_n\}$ is finite.
\end{multistmt}
\end{theorem}

\subsection{Openly Haar null sets}\label{ssec:openlyhn}

\begin{definition}
A set $A\subseteq G$ is said to be \emph{openly Haar null} if there is a Borel probability measure $\mu$ on $G$ such that for every $\varepsilon>0$ there is an open set $U\supseteq A$ such that $\mu(gUh)<\varepsilon$ for every $g, h\in G$. If $\mu$ has these properties, we say that $\mu$ witnesses that $A$ is openly Haar null.
\end{definition}

The notion of openly Haar null sets was introduced in \cite{SolND}, and more thoroughly examined in \cite{CK}.

The following simple proposition shows that this is a stronger property than Haar nullness:
\begin{proposition}\label{pro:openlyHNinGdHN}
Every openly Haar null set is contained in a $G_\delta$ Haar null set.
\end{proposition}

\begin{proof}
For every $n\in\omega$ there is an open set $U_n \supseteq A$ such that $\mu(gU_nh)<\frac1{n+1}$ for every $g, h\in G$. Then $B=\bigcap_{n\in\omega} U_n$ is a $G_\delta$ set that is Haar null because it satisfies $\mu(gBh)\le \mu(gU_nh)<\frac1{n+1}$ for every $g, h\in G$, hence $\mu(gBh)=0$ for every $g, h\in G$. 
\end{proof}

\begin{remark}
In non-locally-compact abelian Polish groups \enquote{openly Haar null} is a \emph{strictly} stronger property than \enquote{Haar null}, because \cite[Theorem 1.3]{EVGdhull} states that in these groups there is a Haar null set that is not contained in a $G_\delta$ Haar null set.
\end{remark}

To prove that the system of openly Haar null sets forms a $\sigma$-ideal, we will need the following technical lemma:

\begin{lemma} \label{lem:cptsupwitnessopen}
If $\mu$ witnesses that $A\subseteq G$ is openly Haar null and $V$ is a neighborhood of $1_G$, then there exists a measure $\mu'$ which also witnesses that $A$ is openly Haar null and has a compact support that is contained in $V$. 
\end{lemma}

\begin{proof}
\autoref{lem:posmeascpt} states that there are a compact set $C\subseteq G$ and $c \in G$ such that $\mu(C)>0$ and $C \subseteq cV$. Let $\mu'(X) := \frac{\mu(cX \cap C)}{\mu(C)}$, this is clearly a Borel probability measure and has a compact support that is contained in $V$. Fix an arbitrary $\varepsilon>0$. We will find an open set $U\supseteq A$ such that $\mu'(gUh)<\varepsilon$ for every $g, h\in G$. Notice that \[\mu'(gUh)<\varepsilon\quad \Leftrightarrow\quad \mu(cgUh\cap C)<\mu(C)\cdot\varepsilon \quad\Leftarrow\quad \mu(cgUh)<\mu(C)\cdot\varepsilon.\]
There exists an open set $U\supseteq A$ with $\mu(gUh)<\mu(C)\cdot\varepsilon$ for every $g, h\in G$, and this satisfies $\mu(cgUh)<\mu(C)\cdot \varepsilon$ for every $g, h\in G$, hence $\mu'$ has the required properties.
\end{proof}

\begin{theorem}[Cohen-Kallman]
The system of openly Haar null sets is a translation invariant $\sigma$-ideal.
\end{theorem}

\begin{proof}
It is trivial that the system of openly Haar null sets satisfy (I) and (II) in \autoref{def:sigmaideal}. The following proof of (III) is described in the appendix of \cite{CK} and is similar to the proof of \autoref{thm:haarnullsigmaideal}.

Let $A_n$ be openly Haar null for all $n\in \omega$, we prove that $A=\bigcup_{n\in\omega} A_n$ is also openly Haar null. For every set $A_n$ fix a measure $\mu_n$ which witnesses that $A_n$ is openly Haar null. Let $d$ be a complete metric on $G$ that is compatible with the topology of $G$.

We construct for all $n \in \omega$ a compact set $C_n \subseteq G$ and a Borel probability measure $\tilde\mu_n$ such that the support of $\tilde\mu_n$ is $C_n$, $\tilde\mu_n(gU_n h)<\varepsilon$ for every $g, h\in G$ and the \enquote{size} of the sets $C_n$ decreases \enquote{quickly}.

The construction will be recursive. For the initial step use \autoref{lem:cptsupwitnessopen} to find a measure $\tilde\mu_0$ witnessing that $A_0$ is openly Haar null and has compact support $C_0 \subseteq G$. Assume that $\tilde{\mu}_{n'}$ and $C_{n'}$ are already defined for all $n'<n$. By \autoref{correctionlemma} there exists a neighborhood $V_n$ of $1_G$ such that if $v\in V_n$, then $d(k\cdot v, k)<2^{-n}$ for every $k$ in the compact set $C_0C_1C_2\cdots C_{n-1}$. Applying \autoref{lem:cptsupwitnessopen} again we can find a measure $\tilde{\mu}_n$ with compact support $C_n\subseteq V_n$ which is witnessing that $A_n$ is openly Haar null.

If $c_n \in C_n$ for all $n\in \omega$, then it is clear that the sequence $(c_0c_1c_2\cdots c_n)_{n\in\omega}$ is a Cauchy sequence. As $(G, d)$ is complete, this Cauchy sequence is convergent; we write its limit as the infinite product $c_0c_1c_2\cdots$. The map $\varphi : \prod_{n\in\omega} C_n \to G$, $\varphi((c_0, c_1, c_2, \ldots)) = c_0c_1c_2\cdots$ is the pointwise limit of continuous functions, hence it is Borel. 

Let $\mu^\Pi$ be the product of the measures $\tilde{\mu}_n$ on the product space $C^\Pi := \prod_{n\in\omega} C_n$. Let $\mu = \varphi_*(\mu^\Pi)$ be the push-forward of $\mu^\Pi$ along $\varphi$ onto $G$, i.e. 
\[\mu(X) = \mu^\Pi(\varphi^{-1}(X))= \mu^\Pi\left(\left\{(c_0, c_1, c_2, \ldots) \in C^\Pi : c_0c_1c_2\cdots \in X\right\}\right).\]

We claim that $\mu$ witnesses that $A=\bigcup_{n\in\omega} A_n$ is openly Haar null. Fix an arbitrary $\varepsilon>0$, we will show that there is an open set $U\supseteq A$ such that $\mu(gUh)<\varepsilon$ for every $g, h\in G$. It is enough to find open sets $U_n\supseteq A_n$ such that $\mu(g U_n h)<\varepsilon\cdot 2^{-(n+2)}$ for every $g, h\in G$ and $n\in\omega$, because then $U=\bigcup_{n\in\omega} U_n$ satisfies that for every $g, h\in G$ 
\[\mu(gUh) \le \sum_{n\in\omega}\mu(g U_n h) \le \frac{\varepsilon}{2}<\varepsilon.\]

Fix $g, h \in G$ and $n \in \omega$. Choose an $U_n\supseteq A_n$ open set satisfying that $\tilde{\mu}_n(gU_nh)<\varepsilon\cdot 2^{-(n+2)}$ for every $g, h\in G$. Notice that if $c_j\in C_j$ for every $j\neq n$, $j\in\omega$, then
\begin{align*}
&\tilde{\mu}_n\left(\{c_n \in C_n: c_0c_1c_2\cdots c_n\cdots \in g U_n h\}\right)=\\
&\quad=\tilde{\mu}_n\left(\left(c_0c_1\cdots c_{n-1}\right)^{-1}\cdot g U_n h\cdot \left(c_{n+1}c_{n+2}\cdots\right)^{-1}\right) <\varepsilon\cdot 2^{-(n+2)}
\end{align*}
because $\tilde\mu_n(g'U_n h')<\varepsilon\cdot 2^{-(n+2)}$ for all $g', h'\in G$. Applying Fubini's theorem in the product space $\left(\prod_{j\neq n} C_j\right)\times C_n$ to the product measure  $\left(\prod_{j\neq n} \tilde{\mu}_j\right)\times \tilde{\mu}_n$ yields that 
\[\varepsilon\cdot 2^{-(n+2)}>\mu^\Pi\left(\left\{(c_0, c_1, \ldots, c_n, \ldots) \in C^\Pi : c_0c_1\cdots c_n\cdots \in gU_n h \right\}\right).\]
By the definition of $\mu$ this means that $\mu(gU_n h)<\varepsilon\cdot 2^{-(n+2)}$.
\end{proof}

Unfortunately, there are groups where this ideal contains only the empty set. The following results about this \enquote{collapse} are proved as \cite[Propositons 5 and 3]{CK}:

\begin{proposition}[Cohen-Kallman]\label{prop:ohntriv}
Assume that $G$ is a Polish group. If for every compact subset $C\subseteq G$ and every nonempty open subset $U\subseteq G$ there are $g, h\in G$ with $gCh\subseteq U$, then the empty set is the only openly Haar null subset of $G$.
\end{proposition}

\begin{proposition}[Cohen-Kallman]\label{prop:ohntrivH01}
In particular $G=\mathcal{H}[0,1]$, the group of order-preserving self-homeomorphisms of $[0,1]$ (endowed with the compact-open topology) has this property, hence in $\mathcal{H}[0,1]$ only the empty set is openly Haar null.
\end{proposition}

On the other hand, \cite[Proposition 2]{CK} shows several groups and classes of groups where the ideal of openly Haar null sets is nontrivial:

\begin{proposition}[Cohen-Kallman]\label{prop:ohnnontriv}
In the Polish group $G$ there is a nonempty openly Haar null subset if at least one of the following conditions holds:
\begin{multistmt}
\item $G$ is uncountable and admits a two-sided invariant metric,
\item $G=S_\infty$ is the group of permutations of $\mathbb{N}$ with the topology of pointwise convergence,
\item $G=\mathrm{Aut}(\mathbb{Q},\le)$ is the group of order-preserving self-bijections of the rationals with the topology of pointwise convergence on $\mathbb{Q}$ viewed as discrete (i.e.\ a sequence $(f_n)_{n\in\omega}\in G^\omega$ is said to be convergent if for every $q\in \mathbb{Q}$ there is a $n_0\in\omega$ such that the sequence $(f_n(q))_{n\ge n_0}$ is constant),
\item $G=\mathcal{U}(\ell^2)$ is the unitary group on the separable infinite-dimensional complex Hilbert space with the strong operator topology,
\item $G$ admits a continuous surjective homomorphism onto a group which has a nonempty openly Haar null subset.
\end{multistmt}
\end{proposition}

In \autoref{ssec:decomposition} we will mention \autoref{thm:ohndecon}, which is an interesting application of this notion.

\subsection{Generically Haar null sets}\label{ssec:gnchn}

\begin{definition}
Assume that $A\subseteq G$ is universally measurable. Let 
\[T(A) = \{\mu\in P(G) : \mu(gAh)=0 \text{ for all } g, h \in G\}\]
be the set of its witness measures. We say that $A$ is \emph{generically Haar null} if $T(A)$ is comeager in $P(G)$ (the Polish space of Borel probability measures).
\end{definition}

Dodos introduced this notion in the paper \cite{DodosDi} (under the name \enquote{strongly Haar null sets} and only considering abelian groups), and later generalized them for arbitrary Polish groups in \cite{DodosSt} (where they are called \enquote{generically Haar null sets}).

The \enquote{trick} of this definition is that the system of generically Haar null sets is obviously closed under countable unions because $T\left(\bigcup_{n\in\omega} A_n\right) \supseteq \bigcap_{n\in \omega} T(A_n)$ and the countable intersection of comeager sets is still comeager. The following proposition states this fact and the other useful properties which are evident from the definition:

\begin{proposition}
The system of generically Haar null sets is translation-invariant and forms a $\sigma$-ideal in the $\sigma$-algebra of universally measurable sets. Every generically Haar null set is a generalized Haar null set.
\end{proposition}

Notice that we only defined generically Haar null sets among the universally measurable sets. 

In the special case when $G$ is abelian, this notion coincides with the \enquote{generic} version of \autoref{thm:haarnullwitnessfunc} part (3):
\begin{theorem}[Banakh-G\l\k{a}b-Jab\l o\'nska-Swaczyna]\label{thm:genhaarnullwitnessfunc}
Assume that $G$ is an abelian Polish group and $B\subseteq G$ is a Borel set. Then the following are equivalent:
\begin{multistmt}
\item in the Polish space $\mathcal{C}(2^\omega, G)$ of continuous functions from $2^\omega$ to $G$ (endowed with the compact-open topology) the set
\[\{f\in \mathcal{C}(2^\omega, G) : f^{-1}(gB)\in\mathcal{N}(2^\omega)\}\quad\text{is comeager},\]
\item in the Polish space $\mathcal{P}(G)$ of Borel probability measures on $G$ the set
\[\{\mu\in \mathcal{P}(G) : \text{$\mu$ witnesses that $B$ is Haar null}\}\quad\text{is comeager}.\]
\end{multistmt}
\end{theorem}
As in \autoref{thm:haarnullwitnessfunc}, $\mathcal{N}(2^\omega)$ denotes the $\sigma$-ideal of sets of Haar measure zero on the Cantor cube $2^\omega$. The proof of this result can be found as part (2) of \cite[Theorem 13.8]{BGJS}.

It is also possible to define an analogous one-sided notion:

\begin{definition}
Assume that $A\subseteq G$ is universally measurable. Let 
\[T_l(A) = \{\mu\in P(G) : \mu(gA)=0 \text{ for all } g \in G\}\]
be the set of its left witness measures. We say that $A$ is \emph{generically left Haar null} if $T_l(A)$ is comeager in $P(G)$.
\end{definition}

It is clear that generically Haar null sets are always generically left Haar null and the two notions coincide in abelian groups.

\begin{proposition}
The system of generically left Haar null sets is translation-invariant and forms a $\sigma$-ideal in the $\sigma$-algebra of universally measurable sets. Every generically left Haar null set is a generalized left Haar null set.
\end{proposition}

Notice that generically left Haar null sets are closed under countable unions even in those groups where left Haar null sets (or generalized left Haar null sets) do not have this property (see \autoref{thm:freesgnotsigmaid}). 

The following trichotomy gives additional motivation for this notion:
\begin{theorem}[Dodos]\label{thm:dodostri}
Each analytic set $A\subseteq G$ satisfies one of the following:
\begin{multistmt}
\item $T(A)=\emptyset$ (i.e.\ $A$ is not generalized Haar null),
\item $T(A)$ is meager and dense in $P(G)$,
\item $T(A)$ is comeager in $P(G)$ (i.e.\ $A$ is generically Haar null).
\end{multistmt}
\end{theorem}
\begin{proof}
First we prove that if $A$ is generalized Haar null, then $T(A)$ is dense in $P(G)$ (this part of the proof is from \cite[Proposition 5 (i)]{DodosRe}). Let $\delta_x$ denote the Dirac measure concentrated at $x\in G$ (that is, $\delta_x(X)= 1$ if $x\in X$ and $\delta_x(X)=0$ otherwise). The system
\[\mathcal{F}=\{\sum_{i=1}^n \alpha_i \delta_{x_i} : n\in \omega, 0\le \alpha_i\le 1, \sum_{i=1}^n\alpha_i = 1, x_i \in G\}\]
is a dense subset of $P(G)$ (see \cite[Theorem 17.19]{Ke}). It is enough to prove that if $\varphi=\sum_{i=1}^n\alpha_i \delta_{x_i}\in \mathcal{F}$ is arbitrary, then any of its neighborhoods contains an element of $T(A)$. If $U\subset P(G)$ is a neighbourhood of $\varphi$, then there are open sets $U_i \subset G$ ($1\le i\le n$) such that $U_i$ is a neighbourhood of $x_i$ and $U$ contains all measures of the form $\sum_{i=1}^n\alpha_i \mu_i$ where $\mu_i\in P(G)$ and $\supp \mu_i \subset U_i$. \autoref{cptsupwitness} states that there are witness measures $\mu_i\in T(A)$ with $\supp \mu_i\subset U_i$; the convex combination $\sum_{i=1}^n\alpha_i \mu_i$ is also a witness measure, concluding this proof.

Now it is enough to prove that $T(A)$ is always either meager in $P(G)$ or comeager in $P(G)$ (this part of the proof is from \cite[Theorem A]{DodosDi}). This proof relies on the following fact about the geometrical structure of the space of probability measures:
\begin{definition}
Let $S$ be a set equipped with a function $c: [0,1]\times S\times S\to S$ that defines the convex combinations of elements of $S$ (for example $S=P(X)$ with $c(t,\mu,\nu) = t\mu + (1-t)\nu$). A subset $F\subseteq S$ is called a \emph{face} of $S$ if it is convex (that is, $0\le t\le 1, p, q\in F \Rightarrow c(t,p,q)\in F$) and extremal (that is, $0<t<1, c(t,p,q)\in F \Rightarrow p \in F \& q \in F$).
\end{definition}
\begin{theorem}[Dodos]
Let $X$ be a Polish space and $F$ be a face of $P(X)$ with the Baire property. Then $F$ is either a meager or a comeager subset of $P(X)$.
\end{theorem}
The proof of this result can be found as \cite[Theorem B]{DodosDi}.

It is easy to see that if $A$ is a universally measurable generalized Haar null set, then $T(A)$ is a face of $P(G)$.

To finish the proof of the theorem it is enough to prove the following claim:
\begin{claim}\label{cla:TAbaire}
For an analytic set $A\subseteq G$, $T(A)$ has the Baire property.
\end{claim}
\begin{proof}
It is clear that the set
\[S = \{((g, h), gah) : g, h\in G, a\in A\}\subseteq (G\times G)\times G\]
is analytic. If we apply \autoref{thm:kondotugue} for the set $S$, $X=G\times G$ and $Y=G$, then we can see that
\[\{(\mu, (g, h), r) \in P(G)\times (G\times G)\times \mathbb{R}: \mu(gAh) >r\}\]
is analytic. If we take the section of this set at $r=0$ and project this section on o$P(G)$, then we can see that
\[\{\mu \in P(G) : \exists g, h\in G : \mu(gAh)>0\} = P(G)\setminus T(A)\]
is analytic, $T(A)$ is coanalytic. It follows from \cite[Corollary 29.14]{Ke} that analytic and coanalytic sets have the Baire property, concluding the proof.
\end{proof}
\end{proof}
                                                     
It is clear that the following one-sided variant of this trichotomy can be proved analogously:
\begin{theorem}[Dodos]
Each analytic set $A\subseteq G$ satisfies one of the following:
\begin{multistmt}
\item $T_l(A)=\emptyset$ (i.e.\ $A$ is not generalized left Haar null),
\item $T_l(A)$ is meager and dense in $P(G)$,
\item $T_l(A)$ is comeager in $P(G)$ (i.e.\ $A$ is generically left Haar null).
\end{multistmt}
\end{theorem}

Following \cite{DodosDi} and \cite{DodosRe}, we stated the trichotomy \autoref{thm:dodostri} for analytic sets. This is formally stronger than stating it for Borel sets, but it turns out that the analytic generically Haar null sets are just the analytic subsets of the Borel generically Haar null sets:

\begin{theorem}[Dodos]\label{thm:dodosborelhull}
Let $A \subseteq G$ be an analytic generalized Haar null set. Then there exists a Borel Haar null set $B \supseteq A$ such that $T(A) \setminus T(B)$ is meager in $P(G)$.
\end{theorem}
\begin{proof}
This result originally appears as \cite[Corollary 12]{DodosRe}, in a paper which only considers the question in the abelian case. We use a different method to prove this result (in all Polish groups $G$); this method is from \cite{Sol} where it was used to prove the result which we stated as the equivalence $(1)\Leftrightarrow (5)$ in \autoref{thm:haarnulleqv}. 

In the proof of \autoref{thm:dodostri} we proved \autoref{cla:TAbaire} which states that $T(A)$ has the Baire property. This implies that $T(A)$ can be written as $T(A) = R\cup M$ where $R$ is a nonempty Borel subset of $P(G)$ and $M$ is meager in $P(G)$.

\begin{claim} The family of sets
\[\Phi = \{X \subseteq G : X\text{{\upshape\ is analytic and $\mu(gXh)=0$ for every $g,h \in G$ and $\mu\in R$}}\}\]
is \emph{coanalytic on analytic}, that is, for every Polish space $Y$ and $P \in \mathbf{\Sigma}^1_1(Y\times G)$, the set $\{y \in Y : P_y\in \Phi \}$ is $\mathbf{\Pi}^1_1$.
\end{claim}

\begin{proof}
Let $Y$ be Polish space and $P \in \mathbf{\Sigma}^1_1(Y\times G)$ and let 
\[\tilde{P} =\{(g, h, y, \gamma) \in G\times G\times Y\times G : \gamma \in gP_yh\}.\]
Then $\tilde{P}$ is analytic, as it is the preimage of $P$ under $(g, h, y,\gamma)\mapsto (y, g^{-1}\gamma h^{-1})$. Applying \autoref{thm:kondotugue} yields that
\[\{(\mu, g, h, y, r) \in P(G) \times G\times G\times Y \times \mathbb{R}: \mu(\tilde{P}_{(g, h, y)})> r\}\]
is analytic, therefore its section
\[\{(\mu, g, h, y) \in P(G)\times G \times G\times Y : \mu(\tilde{P}_{(g, h, y)})> 0\}\]
is also analytic. If we intersect this with the Borel set $R \times G\times G\times Y$, we get the analytic set
\[\{(\mu, g, h, y) \in R\times G \times G\times Y : \mu(\tilde{P}_{(g, h, y)})> 0\}.\]
Projecting this on $Y$ yields that 
\[\{y : \mu(\tilde{P}_{(g, h, y)})> 0\text{ for some $g, h \in G$ and $\mu\in R$}\}\]
is analytic. Using this and the fact that $\tilde{P}_{(g,h, y)} = g P_h y$ (by definition) yields that
\[\{y\in Y : \mu(\tilde{P}_{(g, h, y)})=0 \text{ for all $g, h\in G$ and $\mu \in R$}\}=\{y\in Y : P_y\in \Phi\}\]
is coanalytic.
\end{proof}
Now, since $A\in \Phi$, by the dual form of the First Reflection Theorem (see \cite[Theorem 35.10]{Ke} and the remarks following it) there exists a Borel set $B$ with $B\supseteq A$ and $B\in \Phi$. $B\in \Phi$ means that $R\subset T(B)$, in particular, $B$ is Haar null because $R$ is nonempty. It is clear from the definition that $T(B)\subset T(A)$, therefore
\[T(A)\setminus T(B)\subseteq T(A)\setminus R \subseteq M,\]
which shows that $T(A)\setminus T(B)$ is indeed meager.
\end{proof}

The following result is an immediately corollary of this theorem:
\begin{corollary}
An analytic set $A\subseteq G$ is generically Haar null if and only if there is a Borel generically Haar null set $B \subseteq G$ such that $A\subseteq B$.
\end{corollary}

It is relatively easy to prove that a Borel generically Haar null set is not only Haar null, but also meager. (This idea appears as part (ii) of \cite[Proposition 5]{DodosRe}.) In fact, this is also true for generically left Haar null sets:
\begin{proposition}[Dodos]
If $B$ is a Borel generically left Haar null set, then $B$ is meager.
\end{proposition}
\begin{proof}
Assume for the contrary that $B$ is not meager. This implies that $B$ is comeager in some open subset of $G$ (see e.g.\ \cite[8.26]{Ke}). If $D$ is a countable dense subset of $G$, then the product $DB=\{db: d\in D, b\in B\}$ is comeager in $G$. It follows from \cite[Theorem 10]{CrK} that $\mu(DB)=1$ for the generic measure $\mu\in P(G)$. On the other hand, the generic measure $\mu\in P(G)$ satisfies that $\mu(gB)=0$ for all $g\in G$ (because $B$ is generically left Haar null), thus $\mu(DB) \le \sum_{g\in D} \mu(gB) = 0$, a contradiction.

Notice that this proof works even if we replace \enquote{Borel} with \enquote{universally measurable and with the Baire property} in the proposition.
\end{proof}

The following proposition implies that this $\sigma$-ideal is nontrivial (contains nonempty sets) in abelian groups:
\begin{proposition}[Dodos]\label{pro:cptisgnchn}
Assume that $G$ is an abelian Polish group. If $A\subset G$ is a $\sigma$-compact Haar null set, then $A$ is generically Haar null.
\end{proposition}
The proof of this result can be found as part (iii) of \cite[Proposition 5]{DodosRe}.

Notice that this proposition implies that if (the abelian Polish group) $G$ is locally compact, then the closed Haar null sets are generically Haar null. The converse of this is also true:
\begin{theorem}[Dodos]
Assume that $G$ is an abelian Polish group. Then $G$ is locally compact if and only if every closed Haar null set is generically Haar null.
\end{theorem}
The proof of this result can be found as \cite[Corollary 9]{DodosDi}.

In \autoref{ssec:steinhaus} we will state a variant of the Steinhaus theorem (\autoref{thm:steinhausgnchn}) which uses the generically left Haar null sets as small sets.

For more information about related notions in the special case of abelian Polish groups, see \cite{BGJS}.

\subsection{Strongly Haar meager sets}\label{ssec:stronghm}

Strongly Haar meager sets are the variant of Haar meager sets where we require the witness function to be the identity function of $G$ restricted to a compact subset $C\subseteq G$. (Despite the similar name, this notion is unrelated to the generically Haar null sets, which are also called strongly Haar null sets in \cite{DodosDi}.) This is motivated by the fact that when we prove that some set is Haar meager, we frequently use witness functions of this kind.

\begin{definition}\label{def:stronglyhaarmeager}
A set $A \subseteq G$ is said to be \emph{strongly Haar meager} if there are a Borel set $B \supseteq A$ and a (nonempty) compact set $C\subseteq G$ such that $gBh\cap C$ is meager in $C$ for every $g, h \in G$.
\end{definition}

The following basic question is \cite[Problem 2]{Da}:

\begin{question}[Darji]\label{que:HMeqSHM}
Is every Haar meager set strongly Haar meager?
\end{question}

(Note that the paper \cite{Da} only considered the case of abelian groups.)

The result \cite[Theorem 5.13]{BGJS} shows that in a certain class of abelian Polish groups the Haar meager sets and the strongly Haar meager sets coincide:

\begin{definition}
A topological group $G$ is called \emph{hull-compact} if each compact subset of $G$ is contained in a compact subgroup of $G$.
\end{definition}

\begin{theorem}[Banakh-G\l\k{a}b-Jab\l o\'nska-Swaczyna]
If the abelian Polish group $G$ is hull-compact, then every Haar meager subset of $G$ is strongly Haar meager.
\end{theorem}

\begin{example}\label{bgjsexample}
It is not very hard to verify that the abelian Polish group $(\mathbb{Q}/\mathbb{Z})^\omega$ is hull-compact (where we endow $\mathbb{Q}$ with the discrete topology).
\end{example}

However, the result \cite[Theorem 1.8]{ENPV} answers this \autoref{que:HMeqSHM} negatively:

\begin{theorem}[Elekes-Nagy-Po\'or-Vidny\'anszky]\label{HMneqSHM}
In the abelian Polish group $\mathbb{Z}^\omega$, there exists a $G_\delta$ set that is Haar meager but not strongly Haar meager.
\end{theorem}

The other basic question of Darji related to strongly Haar meager sets is \cite[Problem 3]{Da}:

\begin{question}[Darji]\label{que:shmsigmaideal}
Is the system of strongly Haar meager sets a $\sigma$-ideal?
\end{question}

(Note that while the paper \cite{Da} only considered the case of abelian groups, this question seems to be interesting in other Polish groups as well.)

Even the following variant of this question seems to be open:

\begin{question}\label{que:shmideal}
Is the system of strongly Haar meager sets an ideal?
\end{question}

We conclude this subsection by stating part (2) of \cite[Theorem 13.8]{BGJS} without proof:

\begin{theorem}[Banakh-G\l\k{a}b-Jab\l o\'nska-Swaczyna]
Assume that $G$ is an abelian Polish group and $B\subseteq G$ is a Borel set. Then the following are equivalent:
\begin{multistmt}
\item in the Polish space $\mathcal{C}(2^\omega, G)$ of continuous functions from $2^\omega$ to $G$ (endowed with the compact-open topology) the set
\[\{f\in \mathcal{C}(2^\omega, G) : \text{$f$ witnesses that $B$ is Haar meager}\}\quad\text{is comeager},\]
\item in the Polish space $\mathcal{K}(G)$ of nonempty, compact subsets of $G$
\[\{K\in \mathcal{K}(G) : \text{$K$ witnesses that $B$ is strongly Haar meager}\}\quad\text{is comeager}.\]
\end{multistmt}
\end{theorem}

Sets that satisfy these equivalent conditions are called \emph{generically Haar meager sets}; it is easy to see that the system of these sets is closed under countable unions.

\section{Analogs of the results from the locally compact case}\label{sec:analogs}

In this section we discuss generalizations and analogs of a few theorems that are well-known for locally compact groups. Unfortunately, although these are true in locally compact groups for the sets of Haar measure zero and the meager sets, neither of them remains completely valid for Haar null sets and Haar meager sets. However, in some cases weakened versions remain true, and these often prove to be useful.

\subsection{Fubini's theorem and the Kuratowski-Ulam theorem}\label{ssec:fubinikurulam}

Fubini's theorem, and its topological analog, the Kuratowski-Ulam theorem (see e.g.\ \cite[8.41]{Ke}) describes small sets in product spaces. They basically state that a set (which is measurable in the appropriate sense) in the product of two spaces is small if and only if co-small many sections of it are small. Notice that because the product of left (or right) Haar measures of two locally compact groups of is trivially a left (or right) Haar measure of the product group, (a special case of) Fubini's theorem connects the sets of Haar measure zero in the two groups and the sets of Haar measure zero in the product group.

Unfortunately, analogs of these theorems are proved only in very special cases, and there are counterexamples known in otherwise \enquote{nice} groups. We provide a simple counterexample (which can be found as \cite[Example 20]{Do}) that works in both the Haar null and Haar meager case.

\begin{example}[folklore] There exists a closed set $A\subseteq \mathbb{Z}^\omega\times \mathbb{Z}^\omega$ that is neither Haar null nor Haar meager, but in one direction all its sections are Haar null and Haar meager. (In the other direction, non-Haar-null and non-Haar-meager many sections are non-Haar-null and non-Haar-meager.)
\end{example}

\begin{proof}
The group operation of $\mathbb{Z}^\omega\times\mathbb{Z}^\omega$ is denoted by $+$.

The set with these properties will be 
\[A=\{(s, t)\in\mathbb{Z}^\omega\times \mathbb{Z}^\omega : t_n\ge s_n\ge 0\text{ for every }n\in\omega\}.\]
It is clear from the definition that $A$ is closed.

Note that for $t\in \mathbb{Z}^\omega$, the section $A^t = \{s\in\mathbb{Z}^\omega: t_n \ge s_n\ge 0\text{ for every }n\in\omega\}$ is compact (as it is the product of finite sets with the discrete topology), and it follows from \autoref{thm:tsicpthaarnull} and \autoref{thm:tsicpthaarmeager} that all compact sets are Haar null and Haar meager in $\mathbb{Z}^\omega$.

To show that $A$ is not Haar null and not Haar meager, we will use the technique described in \autoref{ssec:cpttranslates} and show that for every compact set $C\subseteq \mathbb{Z}^\omega\times\mathbb{Z}^\omega$ the set $A$ contains a translate of $C$. As $C=\emptyset$ satisfies this, we may assume that $C\neq\emptyset$. Let $\pi^1_n(s,t)=s_n$ and $\pi^2_n(s,t)=t_n$, then $\pi^1_n, \pi^2_n : \mathbb{Z}^\omega\times\mathbb{Z}^\omega\to\mathbb{Z}$ are continuous functions. Let $a$ and $b$ be the sequences satisfying $a_n=-\min \pi^1_n(C)$ and $b_n=-\min \pi^2_n(C)+\max \pi^1_n(C)-\min\pi^1_n(C)$. It is straightforward to check that this choice guarantees that if $(s,t)\in C+(a,b)$, then $t_n\ge s_n\ge 0$ for every $n\in\omega$.

Finally, if $s\in \mathbb{Z}^\omega$ satisfies $s_n\ge 0$ for every $n\in\omega$, then similar, but simpler arguments show that the section $A_s =\{t\in \mathbb{Z}^\omega : t_n\ge s_n\text{ for every }n\in\omega\}$ contains a translate of every compact set $C\subseteq \mathbb{Z}^\omega$, hence it is neither Haar null nor Haar meager. Similarly, the set $\{s\in \mathbb{Z}^\omega : s_n\ge 0\text{ for every }n\in\omega\}$ is also neither Haar null nor Haar meager, so we proved the statement about the sections in the other direction.
\end{proof}

The following counterexample appears as \cite[Theorem 6]{Chr}: 
\begin{example}[Christensen]
Let $H$ be a separable infinite dimensional Hilbert space (with addition as the group operation) and let $\mathbb{S}^1$ be the unit circle in the complex plane (with complex multiplication as the group operation). There exists in the product group $H \times \mathbb{S}^1$ a Borel set $A$ such that
\begin{propertylist}
\item For every $h\in H$, the section $A_h$ has Haar measure one in $\mathbb{S}^1$.
\item For every $s\in \mathbb{S}^1$, the section $A^s$ is Haar null in $H$.
\item The complement of $A$ is Haar null in the product group $H\times \mathbb{S}^1$.
\end{propertylist}
\end{example}
Note that here $\mathbb{S}^1$ is a compact group.

The connection between (I) and (III) in this example is not accidental:
\begin{theorem}[Christensen, Borwein-Moors]
If $(G_1, \cdot)$ is an abelian Polish group and $(G_2, \cdot)$ is a locally compact abelian Polish group and $A\subset G_1\times G_2$ is universally measurable, then the following are equivalent:
\begin{multistmt}
\item $A$ is generalized Haar null in $G_1\times G_2$,
\item there is a generalized Haar null set $N \subset G_1$ such that for all $g_1\in G_1\setminus N$, the section $A_{g_1}$ is Haar null in $G_2$ (that is, the Haar measure on $G_2$ assigns measure zero to it).
\end{multistmt}
\end{theorem}
This connection is mentioned in \cite{Chr} and proved in \cite[Theorem 2.3]{BM}.

We prove the following generalized version of this result:
\begin{theorem}\label{thm:weakfubini}
Suppose that $G$ and $H$ are Polish groups and $B\subseteq G\times H$ is a Borel subset. Then if $G$ is locally compact, then the following conditions are all equivalent. Moreover, the implication $(1) \Rightarrow (2)$ remains valid even if $G$ is not necessarily locally compact.
\begin{multistmt}
\item there exists a Haar null set $E\subseteq H$ and a Borel probability measure $\mu$ on $G$ such that every $h\in H\setminus E$ and $g_1, g_2\in G$ satisfies $\mu(g_1 B^h g_2)=0$ (i.e.\ these sections are Haar null and $\mu$ is witness measure for each of them),
\item $B$ is Haar null in $G\times H$,
\item there exists a Haar null set $E\subseteq H$ such that for every $h\in H\setminus E$ the section $B^h$ has Haar measure zero in $G$.
\end{multistmt}
\end{theorem}

\begin{proof}
$(1)\Rightarrow (2)$:\\
As the set $E$ is Haar null, there exists a Borel set $E'\supseteq E$ and a Borel probability measure $\nu$ on $H$ such that $\nu(h_1 E' h_2)=0$ for every $h_1, h_2\in H$. We show that the measure $\mu\times\nu$ witnesses that $B$ is Haar null in $G\times H$. Fix arbitrary $g_1,g_2\in G$ and $h_1, h_2\in H$, we have to prove that $(\mu\times\nu)((g_1, h_1)\cdot B\cdot (g_2, h_2))=0$.

Notice that
\[(g_1,h_1)\cdot B\cdot (g_2, h_2)\subseteq \left(G\times (h_1\cdot E'\cdot h_2)\right)\cup\left((g_1, h_1)\cdot (B\setminus (G\times E'))\cdot (g_2, h_2)\right)\]
and in this union the $(\mu\times\nu)$-measure of the first term is zero (by the choice of $\mu$ and $E'$). For every $h\in H$, the $h$-section of the Borel set $(g_1, h_1)\cdot (B\setminus (G\times E'))\cdot (g_2, h_2)$ is $g_1\cdot (B\setminus (G\times E'))^{h_1^{-1} h h_2^{-1}} \cdot g_2$, and this is either the empty set (if $h_1^{-1} h h_2^{-1}\in E'$) or a set of $\mu$-measure zero (if $h_1^{-1} h h_2^{-1}\in H\setminus E'$). Applying Fubini's theorem in the product space $G\times H$ to the product measure $\mu\times\nu$ yields that $(\mu\times\nu)((g_1, h_1)\cdot B\cdot (g_2, h_2))=0$, and this is what we had to prove.

$(2)\Rightarrow (3)$ (when $G$ is locally compact):\\
Let $B\subseteq G\times H$ be Haar null and suppose that $\mu$ is a witness measure, i.e. $\mu((g_1, h_1)\cdot B\cdot (g_2, h_2))=0$ for every $g_1, g_2\in G$ and $h_1, h_2\in H$. Fix a left Haar measure $\lambda$ on $G$, let $\delta$ denote the Dirac measure at ${1_H}$ (i.e.\ for $X\subseteq H$, $\delta(X)=1$ if $1_H\in X$ and $0$ otherwise) and let $\tilde{\lambda}= \lambda\times \delta$. Let $E=\{h\in H : \lambda(B^h)\neq 0\}$, it is clearly enough to prove that $E\subseteq H$ is Haar null. First we use a standard argument to show that $E$ is a Borel set.

If $(X, \mathcal{S})$ is a measurable space, $Y$ is a separable metrizable space, $P(Y)$ is the space of Borel probability measures on $Y$ and $A\subseteq X\times Y$ is measurable (i.e. $A\in \mathcal{S}\times \mathcal{B}(Y)$), then \cite[Theorem 17.25]{Ke} states that the map $X\times P(Y)\to \mathbb{R}_+, (x,\varrho)\mapsto \varrho(A_x)$ is measurable (for $\mathcal{S}\times \mathcal{B}(P(Y))$). Applying this for $(X, \mathcal{S}):= (H, \mathcal{B}(H))$, $Y:=G$, and $A:=B$ yields that $\tilde{E}=\{(h, \varrho)\in H\times P(G) : \varrho(B^h)\neq 0\}$ is Borel (the Borel preimage of an open set in $\mathbb{R}_+$). If $\varrho$ is a Borel probability measure on $G$ that is equivalent to $\lambda$ in the sense that they have the same zero sets, then $E=\tilde{E}^\varrho$ is Borel, because it is the section of a Borel set.

We will show that the measure $\nu(X)=\mu(G\times X)$ witnesses that $E$ is Haar null. Fix arbitrary $h_1, h_2\in H$, we have to prove that $0=\nu(h_1 E h_2)=\mu(G\times (h_1 E h_2))$.

Consider the set
\begin{align*}
&S=\{((u_G,u_H),(v_G, v_H))\in((G\times H)\times (G\times H)) : 
\\&\hspace{3cm}(u_G, u_H)\cdot(v_G, v_H)\in (1_G, h_1)\cdot B\cdot(1_G, h_2)\},
\end{align*}
it is easy to see that this is a Borel set. Applying Fubini's theorem in the product space $(G\times H)\times (G\times H)$ to the product measure $\mu\times \tilde\lambda$ yields that 
\[(\mu\times\tilde\lambda)(S)=\int_{G\times H} \mu((1_G, h_1)\cdot B\cdot (v_G^{-1}, h_2 v_H^{-1}))\intby \tilde\lambda((v_G, v_H))=0\]
because $\mu$ witnesses that $B$ is Haar null. Applying Fubini's theorem again for the other direction yields that
\begin{align*}
&0=(\mu\times\tilde\lambda)(S)=\int_{G\times H}\tilde\lambda\left((u_G^{-1}, u_H^{-1}h_1)\cdot B\cdot (1_G, h_2)\right)\intby \mu((u_G, u_H))=
\\&\quad=\int_{G\times H} \lambda\left(\{g \in G : (g, 1_H)\in (u_G^{-1}, u_H^{-1}h_1)\cdot B\cdot (1_G, h_2)\}\right)\intby \mu((u_G, u_H))=
\\&\quad=\int_{G\times H} \lambda\left(\{g\in G : (u_G g, h_1^{-1} u_Hh_2^{-1})\in B\}\right)\intby \mu((u_G, u_H))=
\\&\quad=\int_{G\times H} \lambda\left(u_G^{-1} \cdot B^{h_1^{-1} u_H h_2^{-1}}\right)\intby \mu((u_G, u_H))=
\\&\quad=\int_{G\times H} \lambda\left(B^{h_1^{-1} u_H h_2^{-1}}\right)\intby\mu((u_G, u_H)).
\end{align*}
It is clear from the definition of $E$ that the function \[\varphi : G\times H\to \mathbb{R}_+,\qquad(u_G, u_H)\mapsto \lambda\left(B^{h_1^{-1} u_H h_2^{-1}}\right)\] takes strictly positive values on the set $G\times (h_1 E h_2)$. However, our calculations showed that $\int \varphi \intby \mu$ is zero, hence the $\mu$-measure of the set $G\times (h_1 E h_2)$ must be zero.

$(3)\Rightarrow(1)$ (when $G$ is locally compact):\\
Fix a left Haar measure $\lambda$ on $G$. $(3)$ states that $\lambda(B^h)=0$ for every $h\in H\setminus E$. Using \autoref{thm:samezeroset} it is easy to see that $\lambda(g_1 B^h g_2)=0$ for every $g_1, g_2\in G$. Let $C\subseteq G$ be a Borel set such that $0<\lambda(C)<\infty$ (the regularity of the Haar measure guarantees a compact $C$ satisfying this) and let $\mu$ be the Borel probability measure $\mu(X)=\frac{\lambda(X\cap C)}{\lambda(C)}$. Then $\mu\ll\lambda$ guarantees that $\mu$ satisfies condition $(1)$.
\end{proof}

We also prove the analog of this for Haar meager sets:

\begin{theorem}\label{thm:weakkurulam}
Suppose that $G$ and $H$ are Polish groups and $B\subseteq G\times H$ is a Borel subset. 
Then if $G$ is locally compact, then the following conditions are all equivalent. Moreover, the implication $(1) \Rightarrow (2)$ remains valid even if $G$ is not necessarily locally compact.
\begin{multistmt}
\item there exists a Haar meager set $E\subseteq H$, a (nonempty) compact metric space $K$ and a continuous function $f: K\to G$ such that every $h\in H\setminus E$ and $g_1, g_2\in G$ satisfies that $f^{-1}(g_1 B^h g_2)$ is meager in $K$ (i.e.\ these sections are Haar meager and $f$ is a witness function for each of them),
\item $B$ is Haar meager in $G\times H$,
\item there exists a Haar meager set $E\subseteq H$ such that for every $h\in H\setminus E$ the section $B^h$ is meager in the locally compact group $G$.
\end{multistmt}
\end{theorem}

\begin{proof}
$(1)\Rightarrow (2)$:\\
As the set $E$ is Haar meager, there exist a Borel set $E'\supseteq E$, a (nonempty) compact metric space $K'$ and a continuous function $\varphi : K'\to H$ such that $\varphi^{-1}(h_1 E' h_2)$ is meager in $K'$ for every $h_1, h_2\in H$.

Let $f\times \varphi : K\times K'\to G\times H$ be the function $(f\times\varphi)(k, k')=(f(k), \varphi(k'))$. We show that the function $f\times\varphi$ witnesses that $B$ is Haar meager in $G\times H$. Fix arbitrary $g_1,g_2\in G$ and $h_1, h_2\in H$, we have to prove that $(f\times\varphi)^{-1}((g_1, h_1)\cdot B\cdot (g_2, h_2))$ is meager in $K\times K'$.

Notice that
\[(g_1,h_1)\cdot B\cdot (g_2, h_2)\subseteq \left(G\times (h_1\cdot E'\cdot h_2)\right)\cup\left((g_1, h_1)\cdot (B\setminus (G\times E'))\cdot (g_2, h_2)\right)\]
and in this union
\[(f\times\varphi)^{-1}\left(G\times (h_1\cdot E'\cdot h_2)\right)=K\times \varphi^{-1}(h_1\cdot E'\cdot h_2)\] and using the choice of $\varphi$ and the Kuratowski-Ulam theorem it is clear that this is a meager subset of $K\times K'$. For every $h\in H$, the $h$-section of the Borel set $(g_1, h_1)\cdot (B\setminus (G\times E'))\cdot (g_2, h_2)$ is $g_1\cdot (B\setminus (G\times E'))^{h_1^{-1} h h_2^{-1}} \cdot g_2$, and this is either the empty set (if $h_1^{-1} h h_2^{-1}\in E'$) or a set whose preimage under $f$ is meager in $K$ (if $h_1^{-1} h h_2^{-1}\in H\setminus E'$). This means that for every $k'\in K'$, the $k'$-section of the set \[(f\times\varphi)^{-1}\left((g_1, h_1)\cdot (B\setminus (G\times E'))\cdot (g_2, h_2)\right)\subseteq K\times K'\] is meager in $K$. Applying the Kuratowski-Ulam theorem in the product space $K\times K'$ yields that the set $(f\times\varphi)^{-1}((g_1, h_1)\cdot B\cdot (g_2, h_2))$ is indeed meager in $K\times K'$, and this is what we had to prove.

$(2)\Rightarrow (3)$ (when $G$ is locally compact):\\
Let $B\subseteq G\times H$ be Haar meager and suppose that $f : K\to G\times H$ is a witness function (where $K$ is a nonempty compact metric space), i.e. $f^{-1}((g_1, h_1)\cdot B\cdot (g_2, h_2))$ is meager in $K$ for every $g_1, g_2\in G$ and $h_1, h_2\in H$. Let $f(k)=(f_G(k), f_H(k))$ for every $k\in K$, then $f_G : K\to G$ and $f_H : K\to H$ are continuous functions. Let $E=\{h\in H : B^h\text{ is not meager in }G\}$, then \cite[Theorem 16.1]{Ke} states that $E$ is Borel. It is enough to prove that $E\subseteq H$ is Haar meager. We show that this is witnessed by the function $f_H : K \to H$. Fix arbitrary $h_1, h_2\in H$, we have to prove that $f_H^{-1}(h_1 E h_2)$ is meager in $K$.

As in the case of measure, define the Borel set
\begin{align*}
&S=\{((u_G,u_H),(v_G, v_H))\in((G\times H)\times (G\times H)) : 
\\&\hspace{3cm}(u_G, u_H)\cdot(v_G, v_H)\in (1_G, h_1)\cdot B\cdot(1_G, h_2)\},
\end{align*}
and the function $\psi : G\to G\times H$, $\psi(g)=(g, 1_H)$ ($\psi$ is the analog of $\tilde{\lambda}$ from the case of measure). Let $f\times \psi : K\times G \to (G\times H)\times (G\times H)$ be the function $(f\times\psi)(k, g)=(f(k), \psi(g))$. First notice that for every $g\in G$, the $g$-section of the set $(f\times \psi)^{-1}(S)\subseteq K\times G$ is
\begin{align*}
&((f\times\psi)^{-1}(S))^g= \{k\in K : (f\times \psi)((k, g))\in S\}=
\\&\quad= \{k\in K : (f(k), \psi(g))\in S\}= \{k\in K : (f(k), (g, 1_H))\in S\}=
\\&\quad=\{k\in K : f(k)\in (1_G, h_1)\cdot B\cdot (1_G, h_2)\cdot(g, 1_H)^{-1}\}=
\\&\quad=f^{-1}((1_G, h_1)\cdot B\cdot (g^{-1}, h_2))
\end{align*}
and this is meager in $K$ (as $f$ witnesses that $B$ is Haar meager), hence the Kuratowski-Ulam theorem yields that $(f\times\psi)^{-1}(S)$ is meager in $K\times G$.

Applying the Kuratowski-Ulam theorem in the other direction yields that the set
\[\{k \in K : ((f\times\psi)^{-1}(S))_k \text{ is not meager in $G$}\}\]
is meager in $K$. Here if we let $u_G=f_G(k)$ and $u_H=f_H(k)$, then 
\begin{align*}
&((f\times\psi)^{-1}(S))_k= \{g\in G : (f\times \psi)((k, g))\in S\}=
\\&\quad= \{g\in G : (f(k), \psi(g))\in S\}= \{g\in G : ((u_G, u_H), (g, 1_H))\in S\}=
\\&\quad=\{g\in G : (g, 1_H)\in (u_G, u_H)^{-1}\cdot (1_G, h_1)\cdot B\cdot (1_G, h_2)\}=
\\&\quad=\{g\in G : (g, 1_H)\in (u_G^{-1}, u_H^{-1} h_1)\cdot B\cdot (1_G, h_2)\}=
\\&\quad=\{g\in G : (u_G g, h_1^{-1} u_H h_2^{-1})\in B\}=
\\&\quad=u_G^{-1} \cdot B^{h_1^{-1} u_H h_2^{-1}} = (f_G(k))^{-1} \cdot B^{h_1^{-1} f_H(k) h_2^{-1}}.
\end{align*}
Thus we know that 
\[\{k \in K : (f_G(k))^{-1} \cdot B^{h_1^{-1} f_H(k) h_2^{-1}} \text{ is not meager in $G$}\} \text{ is meager in $K$},\]
and because meagerness is translation invariant
\[\{k \in K : B^{h_1^{-1} f_H(k) h_2^{-1}} \text{ is not meager in $G$}\} \text{ is meager in $K$}.\]
But notice that 
\begin{align*}
&f_H^{-1}(h_1 E h_2)= \{k\in K : f_H(k)\in h_1 E h_2\}=\{k\in K : h_1^{-1} f_H(k)h_2^{-1}\in E\}=
\\&\quad=\{k\in K : B^{h_1^{-1} f_H(k)h_2^{-1}}\text{ is not meager in }G\}
\end{align*}
so we proved that $f_H$ is indeed a witness measure for $E$.

$(3)\Rightarrow(1)$ (when $G$ is locally compact):\\
Notice that the proof of \autoref{thm:haarmeagerismeagerlc} shows that in a locally compact Polish group every meager set is Haar meager and there is a function which is a witness function for each of them. The implication that we have to prove is clearly a special case of this observation.
\end{proof}

Finally, we state another special case when the analog of the Kuratowski-Ulam theorem is valid. The proof of this result can be found in \cite[Theorem 18]{Do}.
\begin{theorem}[Dole\v zal-Rmoutil-Vejnar-Vlas\' ak]
Suppose that $G$ and $H$ are Polish groups and $A\subseteq G$ and $B\subseteq H$ are analytic sets. Then $A\times B$ is Haar meager in $G\times H$ if and only if at least one of $A$ and $B$ is Haar meager in the respective group.
\end{theorem}

\subsection{The Steinhaus theorem}\label{ssec:steinhaus}

The Steinhaus theorem and its generalizations state that if a set $A\subseteq G$ is not small (and satisfies some measurability condition, e.g.\ it is in the $\sigma$-ideal generated by the Borel sets and the small sets), then $AA^{-1}=\{xy^{-1} : x,y\in A\}$ contains a neighborhood of $1_G$.

The original result of Steinhaus stated this in the group $(\mathbb{R}, +)$ and used the sets of Lebesgue measure zero as the small sets. Weil extended this for an arbitrary locally compact group, using the sets of Haar measure zero as the small sets.

\begin{restate}{cor:SteinhausHaarGen}[Steinhaus, Weil]
If $G$ is a locally compact Polish group, $\lambda$ is a left Haar measure on $G$ and $A \subseteq G$ is $\lambda$-measurable with $\lambda(A)>0$, then $1_G \in \Int(A A^{-1})$.
\end{restate}

We already proved this result and used it to prove that there are no (left or right) Haar measures in non-locally-compact groups.

\begin{remark}
The proof of this result also works for non-locally-compact groups, but in those groups the result is vacuously true.
\end{remark}

It is natural to ask the following question (a question of this type appears as e.g.\ \cite[Problem 2.7]{Ros}):
\begin{question}\label{que:steinhausgen}
Let $A$ be a Borel (or universally measurable) subset in the Polish group $G$ that is not Haar null (or not generalized Haar null, or not Haar meager). What can we say about the groups $G$ where $1_G \in \Int(AA^{-1})$ is neccessarily satisfied?
\end{question}

Unfortunately, there are groups where the answer is negative. This is demonstrated by \cite[Theorem 6.1]{SolLR}, which we already stated as \autoref{thm:soleckiLRresult} in the section about left and right Haar null sets. However, there are weakened versions of the Steinhaus theorem which turn out to be true and useful.

One of these is stated as \cite[Theorem 2.8]{Ros} and is a slight variation of a result in \cite{ChrBS}. This result does not claim that $AA^{-1}$ will be a neighborhood of $1_G$, only that finitely many conjugates of it will cover a neighborhood of $1_G$. Also, it uses generalized right Haar null sets as small sets (see \autoref{ssec:leftrighthaarnull} for the definition), which is a weaker notion than generalized Haar null sets (hence this result does \emph{not} imply the variant where \enquote{right} is omitted from the text).

\begin{theorem}[Christensen, Rosendal]\label{thm:SteinhausRHN}
Suppose that $A\subseteq G$ is a universally measurable subset which is not generalized right Haar null. Then for any neighborhood $W$ of $1_G$ there are $n\in\omega$ and $h_0, h_1, h_2, \ldots, h_{n-1}\in W$ such that
\[h_0 AA^{-1} h_0^{-1}\cup h_1 AA^{-1} h_1^{-1}\cup \ldots \cup h_{n-1} AA^{-1} h_{n-1}^{-1}\]
is a neighborhood of $1_G$.
\end{theorem}

\begin{proof}
Suppose that the conclusion fails for $A$ and $W$, that is, for every $n\in\omega$ and $h_0, h_1, \ldots h_{n-1}\in W$ and any neighborhood $V\owns 1_G$, there is some 
\[g\in V\setminus \left(h_0 AA^{-1} h_0^{-1}\cup h_1 AA^{-1} h_1^{-1}\cup \ldots \cup h_{n-1} AA^{-1} h_{n-1}^{-1}\right).\]
Then we can inductively choose a sequence $(g_j)_{j\in\omega}$ such that $g_j\to 1_G$ and for every $j_0< j_1 < j_2 < \ldots$ index sequence and $r\in\omega$
\begin{propertylist}
\item the infinite product $g_{j_0} g_{j_1} g_{j_2}\cdots$ converges (this can be achieved e.g.\ by requiring $d(g_{j_0}\cdots g_{j_{r-1}}, g_{j_0}\cdots g_{j_{r-1}}\cdot g_{j_r})<2^{-j_r}$ where $d$ is a fixed complete metric on $G$).
\item $g_{j_r}\notin (g_{j_0} \cdots g_{j_{r-1}})^{-1} AA^{-1} (g_{j_0} \cdots g_{j_{r-1}})$
\end{propertylist}
Using (I) we can define a continuous map $\varphi : 2^\omega \to G$ by
\[\varphi(\alpha)=g_0^{\alpha(0)}g_1^{\alpha(1)}g_2^{\alpha(2)}\cdots,\]
where $g^0=1_G$ and $g^1=g$.
Let $\lambda$ be the Haar measure on the Cantor group $(\mathbb{Z}_2)^\omega=2^\omega$ and notice that as $A$ is not generalized right Haar null, there is some $g\in G$ such that
\[\lambda(\varphi^{-1}(Ag))=\varphi_*(\lambda)(Ag)>0\]

So by \autoref{cor:SteinhausHaarGen},
\[\varphi^{-1}(Ag)(\varphi^{-1}(Ag))^{-1}\]
contains a neighborhood of the identity $(0,0,\ldots)$ in $2^\omega$. This neighborhood must contain an element with exactly one \enquote{1} coordinate, so there are $\alpha, \beta\in \varphi^{-1}(Ag)$ and $m\in\omega$ such that $\alpha(m)=1$, $\beta(m)=0$ and $\alpha(j)=\beta(j)$ for every $j\in\omega, j\neq m$. This means that there are $h=g_{j_0}g_{j_1}\cdots g_{j_{r-1}}$, $j_0<j_1<\ldots<j_{r-1}<m$ and $k\in G$ such that $\varphi(\alpha)=hg_m k$ and $\varphi(\beta)=hk$. It follows that \[hg_m h^{-1}= hg_m k \cdot k^{-1}h^{-1} \in Agg^{-1}A= AA^{-1}\]
and so $g_m \in h^{-1}AA^{-1}h=(g_{j_0} \cdots g_{j_{r-1}})^{-1} AA^{-1} (g_{j_0} \cdots g_{j_{r-1}})$, contradicting the choice of $g_m$.
\end{proof}

The following special case is often useful:
\begin{corollary}\label{cor:SteinhausAbel}
Suppose that $A\subseteq G$ is a universally measurable subset which is \emph{conjugacy invariant} (that is, $gAg^{-1}=A$ for every $g\in G$; in abelian groups every subset has this property) and not generalized Haar null. Then $AA^{-1}$ is a neighborhood of the identity.
\end{corollary}

\begin{proof}
If $A$ is conjugacy invariant, then as we noted, $gAh=gAg^{-1}\cdot gh=Agh$ and thus $A$ is generalized right Haar null if and only if $A$ is generalized Haar null. Moreover, if $A$ is conjugacy invariant, then $gAA^{-1}g^{-1}=(gAg^{-1})\cdot(gAg^{-1})^{-1}=AA^{-1}$, hence $AA^{-1}$ is also conjugacy invariant and thus in this case \autoref{thm:SteinhausRHN} states that $AA^{-1}$ is a neighborhood of identity. 
\end{proof}

Variants of the Steinhaus theorem can be used to prove results about automatic continuity (results stating that all homomorphisms $\pi : G\to H$ which satisfy certain properties are continuous). For example the classical result \autoref{cor:SteinhausHaarGen} yields that any universally measurable homomorphism from a locally compact Polish group into another Polish group is continuous (see e.g.\ \cite[Corollary 2.4]{Ros}; a map is said to be \emph{universally measurable} if the preimages of open sets are universally measurable). We prove the following automatic continuity result as a corollary of \autoref{thm:SteinhausRHN}.

\begin{corollary}[Christensen]\label{cor:autocont}
If $G$ and $H$ are Polish groups, $H$ admits a two-sided invariant metric and $\pi : G\to H$ is a universally measurable homomorphism, then $\pi$ is continuous.
\end{corollary}

\begin{proof}
It is enough to prove that $\pi$ is continuous at $1_G$ (a homomorphism is continuous if and only if it is continuous at the identity element). Let $V\subseteq H$ be an arbitrary neighborhood of $1_H$.

Using the continuity of the map $(x, y)\mapsto xy^{-1}$ at $(1_H, 1_H)\in H\times H$, there is an open set $U$ with $1_H\in U$ satisfying that $UU^{-1} \subseteq V$. If $d$ is a two-sided invariant metric on $H$, then the open balls $B(1_H, r)$ are conjugacy invariant sets and form a neighborhood base of $1_H$, hence we may also assume that $U$ is conjugacy invariant.

$G$ can be covered by countably many right translates of $\pi^{-1}(U)$, because there is a countable set $S\subseteq \pi(G)$ that is dense in $\pi(G)$ and thus the system $\{U \cdot s : s\in S\}$ covers $\pi(G)$. This means that the universally measurable set $\pi^{-1}(U)$ is not right Haar null, hence we may apply \autoref{thm:SteinhausRHN} to see that for some $n\in\omega$ and $h_0, h_1, \ldots, h_{n-1}\in G$ the set
\[W=\bigcup_{0\le j<n}h_j\pi^{-1}(U)(\pi^{-1}(U))^{-1}h_j^{-1}\]
is a neighborhood of $1_G$. For every $g\in W$ there is a $0\le j<n$ such that $g\in h_j\pi^{-1}(U)(\pi^{-1}(U))^{-1} h_j^{-1}$, but then \[\pi(g)\in\pi(h_j) UU^{-1} \pi(h_j)^{-1}= \pi(h_j) U\pi(h_j)^{-1}\cdot( \pi(h_j) U\pi(h_j)^{-1})^{-1}=UU^{-1},\]
thus $\pi(W)\subseteq UU^{-1}\subseteq V$ and this shows that $\pi$ is continuous at $1_G$.
\end{proof}

In \cite{SolAmen} Solecki proved that if $G$ is amenable at 1 (see \autoref{def:amenone}), then a simpler variant of the Steinhaus theorem is true, but if $G$ has a free subgroup at 1 (see \autoref{def:freesgone}) and satisfies some technical condition, then this variant is false. The following theorems state these results.

\begin{theorem}[Solecki]
If $G$ is amenable at 1 and $A\subseteq G$ is universally measurable and not generalized left Haar null, then $A^{-1}A$ contains a neighborhood of $1_G$.
\end{theorem}

\begin{definition}
A Polish group $G$ is called \emph{strongly non-locally-compact} if for any neighborhood $U$ of $1_G$ there exists a neighborhood $V$ of $1_G$ such that $U$ cannot be covered by finitely many sets of the form $gVh$ with $g,h\in G$.
\end{definition}
Note that the examples we mentioned after \autoref{def:freesgone} are all strongly non-locally-compact.
\begin{theorem}[Solecki]
If $G$ has a free subgroup at 1 and is strongly non-locally-compact, then there is a Borel set $A\subseteq G$ which is not left Haar null and satisfies $1_G\notin \Int(A^{-1}A)$.
\end{theorem}

Using the notion of generically left Haar null sets (which we describe in \autoref{ssec:gnchn}), it is possible to prove another weakened version of \autoref{que:steinhausgen}:
\begin{theorem}[Dodos]\label{thm:steinhausgnchn}
Suppose that $A\subseteq G$ is an analytic set which is not generically left Haar null. Then $A^{-1}A$ is not meager in $G$.
\end{theorem}
The relatively long proof of this result can be found as \cite[Theorem A]{DodosSt}. 

Combining this and Pettis' theorem yields the following corollary:
\begin{corollary}[Dodos]\label{cor:steinhausgnchnII}
If $A\subseteq G$ is analytic and not generically left Haar null, then $1_G\in \Int(A^{-1}AA^{-1}A)$.
\end{corollary}

The above mentioned Pettis' theorem (which is also known as Piccard's theorem and appears as e.g.\ \cite[Theorem 2.9.6]{Kuc} and \cite[Theorem 9.9]{Ke}) is the analog of the Steinhaus theorem which uses the meager sets as small sets.

As \enquote{Haar meager} is a stonger notion than \enquote{meager}, the following result strengthens this classical result in the abelian case:
\begin{theorem}[Jab\l o\' nska]\label{thm:SteinhausJab}
Let $G$ be an abelian Polish group. If $A\subseteq G$ is a Borel set that is not Haar meager, then $AA^{-1}$ is a neighborhood of $1_G$.
\end{theorem}
The proof of this theorem can be found in \cite{Jab}; the method of the proof is similar to that of \autoref{thm:SteinhausRHN}.

For more information about this question in the special case of abelian Polish groups, see the paper \cite{BGJS}.

\subsection{The countable chain condition}

The countable chain condition (often abbreviated as ccc) is a well-known property of partially ordered sets or structures with associated partially ordered sets. A partially ordered set $(P, \le)$ is said to satisfy the countable chain condition if every strong antichain in $P$ is countable. (A set $A\subseteq P$ is a strong antichain if $\forall x, y\in A : (x\neq y\Rightarrow \nexists z\in P : (z\le x\text{ and }z\le y))$.) This is called countable \enquote{chain} condition because in some particular cases this condition happens to be equivalent to a condition about lengths of certain chains.

We will apply the countable chain condition to notions of smallness in the following sense:
\begin{definition}
Suppose that $X$ is a set and $\mathcal{S}\subseteq \mathcal{A} \subseteq \mathcal{P}(X)$. We say that \emph{$\mathcal{S}$ has the countable chain condition in $\mathcal{A}$} if there is no uncountable system $\mathcal{U}\subseteq \mathcal{A}\setminus\mathcal{S}$ such that $U\cap V\in \mathcal{S}$ for any two distinct $U, V\in \mathcal{U}$.
\end{definition}
In our cases $\mathcal{A}$ will be a $\sigma$-algebra and $\mathcal{S}$ will be the $\sigma$-ideal of \enquote{small} sets. Notice that $\mathcal{S}$ has the countable chain condition in $\mathcal{A}$ if and only if the partially ordered set $(\mathcal{A}\setminus \mathcal{S}, \subseteq)$ satisfies the countable chain condition. Also notice that if $\tilde{\mathcal{S}}\subseteq \mathcal{S}$ and $\tilde{\mathcal{A}}\supseteq \mathcal{A}$, then \enquote{$\tilde{\mathcal{S}}$ has the countable chain condition in $\tilde{\mathcal{A}}$} is a stronger statement than \enquote{$\mathcal{S}$ has the countable chain condition in $\mathcal{A}$}.

We will generalize the following two classical results (these are stated as \cite[Exercise 17.2]{Ke} and \cite[Exercise 8.31]{Ke}).
\begin{proposition}\label{pro:nullccc}
If $\mu$ is a $\sigma$-finite measure, the $\sigma$-ideal of sets with $\mu$-measure zero has the countable chain condition in the $\sigma$-algebra of $\mu$-measurable sets.
\end{proposition}

\begin{proposition}\label{pro:meagerccc}
In a second countable Baire space, the $\sigma$-ideal of meager sets has the countable chain condition in the $\sigma$-algebra of sets with the Baire property.
\end{proposition}

The theorems in \autoref{ssec:loccptcase} state that if $G$ is locally compact, then $\mathcal{N}=\mathcal{HN}=\mathcal{GHN}$ (i.e.\ of Haar measure zero $\Leftrightarrow$ Haar null $\Leftrightarrow$ generalized Haar null) and $\mathcal{M}=\mathcal{HM}$ (i.e.\ meager $\Leftrightarrow$ Haar meager). The special case of \autoref{pro:nullccc} where $\mu=\lambda$ for a left Haar measure $\lambda$ on $G$ means that the $\sigma$-ideal $\mathcal{HN}=\mathcal{GNH}$ has the countable chain condition in the $\sigma$-algebra of $\lambda$-measurable sets. As every universally measurable set is $\lambda$-measurable, this clearly implies that $\mathcal{HN}=\mathcal{GHN}$ has the countable chain condition in the $\sigma$-algebra of universally\ measurable sets. Analogously, \autoref{pro:meagerccc} means that in a locally compact Polish group $G$ the $\sigma$-ideal $\mathcal{HM}$ has the countable chain condition in the $\sigma$-algebra of sets with the Baire property. 

For the case of measure Christensen asked in \cite[Problem 2]{Chr} whether is this true in the non-locally-compact case (the paper \cite{Chr} considers only abelian groups, but the problem is interesting in general).

The following simple example shows that the answer for this is negative in the group $\mathbb{Z}^\omega$ (a variant of this is stated in \cite[Proposition 1]{Doug}). It also answers the analogous question in the case of category.

\begin{example}
For $A\subseteq \omega$ let 
\[S(A)=\{s\in\mathbb{Z}^\omega : s_n\ge 0\text{{\upshape{} if $n\in A$ and }}s_n<0\text{{\upshape{} if $n\notin A$}}\}.\]
Then the system $\{S(A) : A\in \mathcal{P}(\omega)\}$ consists of continuum many pairwise disjoint Borel (in fact, closed) subsets of $\mathbb{Z}^\omega$ which are neither generalized Haar null nor Haar meager.
\end{example}

\begin{proof}
It is clear that $S(A)$ is closed for every $A\subseteq \omega$. If $A$ and $B$ are two different subsets of $\omega$, then some $n\in \omega$ satisfies for example $n\in A\setminus B$ and thus $\forall s\in S(A) : s_n\ge 0$, but $\forall s\in S(B) : s_n<0$.

Finally, for every $A\subseteq \omega$ the set $S(A)$ contains a translate of every compact subset $C\subseteq \mathbb{Z}^\omega$, because if we define $t^{(C)}\in \mathbb{Z}^\omega$ by \[t^{(C)}_n=\begin{cases}\min\{c_n: c\in C\}&\text{ if $n\in A$,}\\-1-\max\{c_n:c\in C\}&\text{ if $n\notin A$,}\end{cases}\] then clearly $C+t^{(C)}\subseteq A$. Applying \autoref{lem:translateall} concludes our proof.
\end{proof}

In \cite{Sol} Solecki showed that the situation is the same in all non-locally-compact groups that admit a two-sided invariant metric. This is a corollary of \autoref{thm:solpuretopol}, which we already stated without proof. As we will use \autoref{lem:translateall} again, this also answers the question in the case of category.

\begin{example}[Solecki]
Suppose that $G$ is non-locally-compact and admits a two-sided invariant metric. Then none of $\mathcal{HN}$, $\mathcal{GHN}$ and $\mathcal{HM}$ has the countable chain condition in $\mathcal{B}(G)$.
\end{example}

\begin{proof}
By \autoref{thm:solpuretopol} there exists a closed set $F\subseteq G$ and a continuous function $\varphi: F\to 2^\omega$ such that for any $x\in 2^\omega$ and any compact set $C\subseteq G$ there is a $g\in G$ with $gC\subseteq\varphi^{-1}(\{x\})$. Then the system $\{\varphi^{-1}(\{x\}) : x\in 2^\omega\}$ consists of continuum many pairwise disjoint closed sets and they all satisfy the requirements of \autoref{lem:translateall}, hence they are neither generalized Haar null nor Haar meager.
\end{proof}

\subsection{Decomposition into a Haar null and a meager set}\label{ssec:decomposition}
In a locally compact group $G$ the regularity of the Haar measures implies that the group can be written as $G=N\cup M$ where $N$ is of Haar measure zero and $M$ is meager. 

In \cite{DaOP} Darji asks if this holds for non-locally-compact groups:
\begin{question}[Darji]\label{que:darji}
Can every uncountable Polish group be written as the union of two sets, one meager and the other Haar null?
\end{question}

\begin{remark}
More precisely, \cite{DaOP} asks whether every uncountable Polish group can be written as the union of a meager and a \textit{generalized} Haar null set (in the terminology of this paper). However, it is easy to see that this distinction is inconsequential, as every meager set is contained in an $F_\sigma$ (in particular, Borel) meager set.
\end{remark}

Although this question is open, there are known results which answer it affirmatively in various groups or classes of groups.

For example, the following surprisingly short calculation from \cite{CK} shows that the existence of a nonempty openly Haar null set in the group implies the existence of a decomposition. (Openly Haar null sets are introduced in \autoref{ssec:openlyhn}.)

\begin{theorem}[Cohen-Kallman]\label{thm:ohndecon}
If there is a nonempty openly Haar null set in $G$, then every countable subset $C\subseteq G$ is contained in a comeager Haar null set. In particular, $G$ may be written as the union $G=A\cup B$ where $A$ is a Haar null set and $B$ is meager in $G$.
\end{theorem}

\begin{proof}
If there is a nonempty openly Haar null set in $G$, then applying \autoref{lem:ctblsmallopenlarge} yields that every countable set in $G$ is openly Haar null. In particular a dense countable set $C'\supseteq C$ is openly Haar null. Then \autoref{pro:openlyHNinGdHN} yields that $C'\subseteq A$ for a $G_\delta$ Haar null set $A$. The dense $G_\delta$ set $A$ is a countable intersection of dense open sets, hence $B:= G\setminus A$ is a countable union of nowhere dense sets, i.e. $B$ is ($F_\sigma$ and) meager.
\end{proof}

Although there are Polish groups where only the empty set is openly Haar null, the combination of this fact and \cite[Proposition 2]{CK} (which we already stated without proof as \autoref{prop:ohnnontriv}) answers the question of Darji affirmatively in several well-known groups:

\begin{theorem}[Cohen-Kallman]\label{thm:decompohn}
The following uncountable Polish groups can be written as the union of a meager and a Haar null set:
\begin{multistmt}
\item uncountable groups that admit a two-sided invariant metric,
\item $S_\infty$, the group of permutations of $\mathbb{N}$,
\item $\mathrm{Aut}(\mathbb{Q},\le)$, the group of order-preserving self-bijections of the rationals,
\item $\mathcal{U}(\ell^2)$, the unitary group on the separable infinite-dimensional complex Hilbert space.
\end{multistmt}
\end{theorem}
Before these, (2) had been already proved in \cite{DM} (using a different approach which is briefly described in \autoref{ssec:random}).

The result \cite[Corollary 5.2]{DEKKVACS} also gives an affirmative answer in two well-known automorphism groups in addition to the already mentioned $\mathrm{Aut}(\mathbb{Q},\le)$:
\begin{theorem}[Darji-Elekes-Kalina-Kiss-Vidny\'anszky] \label{thm:DEKKVAresults}
The following uncountable Polish groups can be written as the union of a meager and a Haar null set:
\begin{multistmt}
\item $\mathrm{Aut}(\mathcal{R})$, where $\mathcal{R}$ is \emph{the} countably infinite random graph (also called the Rad\'{o} graph and the Erd\H{o}s-R\'{e}nyi graph),
\item $\mathrm{Aut}(\mathcal{B}_\infty)$, where $\mathcal{B}_\infty$ is the countable atomless Boolean algebra.
\end{multistmt}
\end{theorem}

These concrete results are special cases of \cite[Corollary 5.1]{DEKKVACS}, which states the following:
\begin{theorem}[Darji-Elekes-Kalina-Kiss-Vidny\'anszky]\label{thm:DEKKVAgeneral}
Let $G$ be a closed subgroup of $S_\infty$. Assume that $G$ satisfies the \emph{FACP} (finite algebraic closure property), that is, for every finite $S\subset \mathbb{N}$ the set $\{b: \lvert G_{(S)}(b)\rvert <\infty\}$ is finite, where $G_{(S)}$ is the pointwise stabilizer of $S$ under the action of $G$. Moreover, suppose that the set $F = \{g\in G : \mathrm{Fix}(g)\text{ is inifinte}\}$ is dense in $G$. Then $G$ can be written as the union of a meager and a Haar null set.
\end{theorem}

These results from \cite{DEKKVACS} are corollaries of the examination of the \enquote{size} of the conjugacy classes in these groups; we omit their relatively long proofs.

It is well known that $\mathrm{Aut}(\mathcal{B}_\infty)$ (the group considered in part (2) of \autoref{thm:DEKKVAresults}) is isomorphic to the homeomorphism group of the Cantor set. The paper \cite{DEKKVH} answers \autoref{que:darji} in two other important homeomorphism groups:
\begin{theorem}[Darji-Elekes-Kalina-Kiss-Vidny\'anszky]\label{thm:DEKKVHresults}
The following uncountable Polish groups can be written as the union of a meager and a Haar null set:
\begin{multistmt}
\item $\mathcal{H}([0,1])$, the group of order-preserving self-homeomorphisms of the unit interval,
\item $\mathcal{H}(\mathbb{S}^1)$, the group of order-preserving self-homeomorphisms of the circle.
\end{multistmt}
Both groups are endowed with the compact-open topology, which coincides with the topology of uniform convergence in these cases.
\end{theorem}

These results are also corollaries of the examination of the size of the conjugacy classes, for their proofs see \cite[Corollary 6.1]{DEKKVH} and \cite[Corollary 6.2]{DEKKVH}. Notice that the example of $\mathcal{H}([0,1])$ shows that decomposition into a meager and a Haar null set may be possible even if there are no nonempty openly Haar null sets in the group (\autoref{prop:ohntrivH01}).

The result \cite[Proposition 6.3]{DEKKVH} allows us to pull back decompositions from factor groups:
\begin{proposition}[Darji-Elekes-Kalina-Kiss-Vidny\'anszky]\label{prop:decompullback}
Let $G$ and $H$ be Polish groups and suppose that there exists a continuous, surjective homomorphism $\varphi : G \to H$. If $H$ can be written as the union of a meager and a Haar null set then $G$ can be decomposed in such a way as well.
\end{proposition}
\begin{proof}
If $\varphi$ satisfies these conditions, then it follows from \cite[Theorem 2.3.3]{Gao} that $\varphi$ is necessarily an open mapping.

Let $H= N\cup M$ where $N$ is Haar null and $M$ is meager. We will show that $G=\varphi^{-1}(N)\cup \varphi^{-1}(M)$ is a good decomposition of $G$. 

By \cite[Proposition 8]{Doug}, the inverse image of a Haar null set under a continuous, surjective homomorphism is also Haar null, hence $\varphi^{-1}(N)$ is Haar null in $G$.

As $M$ is meager, there are closed, nowhere dense sets $(S_n)_{n\in\omega}$ such that $M\subseteq\bigcup_{n\in \omega} S_n$. For each $n$, the set $\varphi^{-1}(S_n)$ is closed (because $\varphi$ is continuous) and nowhere dense (because $\varphi$ is open); these imply that $\varphi^{-1}(M)$ is meager.
\end{proof}

Applying this, it is easy to prove the following:
\begin{corollary}
The following uncountable Polish groups can be written as the union of a meager and a Haar null set:
\begin{multistmt}
\item $(\mathrm{Diff}^k_{+}[0,1], \circ)$, the group of $k$-times differentiable order-preserving self-homeomorphisms of $[0,1]$ (where $k\ge1$ is an integer),
\item $\mathcal{H}(\mathbb{D}^2)$, the group of orientation-preserving self-homeomorphisms of the closed disc $\mathbb{D}^2$.
\end{multistmt}
\end{corollary}
\begin{proof}
(1) The map $\varphi: (\mathrm{Diff}^k_{+}[0,1], \circ) \to (\mathbb{R}^+, \cdot)$, $f\mapsto f'(0)$ is a continuous, surjective homomorphism. As $\mathbb{R}^+$ is locally compact, the existence of a decomposition is well-known.

(2) The map $\psi: \mathcal{H}(\mathbb{D}^2)\to \mathcal{H}(\mathbb{S}^1)$, $h\mapsto h|_{\mathbb{S}^1}$ is a continuous homomorphism. It is easy to see that if we restrict $\psi$ to the set 
\[\{h: h\text{ preserves the center and is linear on the radiuses}\}\subset \mathcal{H}(\mathbb{D}^2),\]
then we get a bijective map. This shows that $\psi$ is surjective. The other condition of \autoref{prop:decompullback} follows from \autoref{thm:DEKKVHresults} (2), which states that $\mathcal{H}(\mathbb{S}^1)$ can be written as the union of a meager and a Haar null set.

We note that part (1) was originally proved in \cite{CK} using part (e) of \cite[Proposition 2]{CK} (which we already stated without proof as part (5) of \autoref{prop:ohnnontriv}); while part (2) is \cite[Corollary 6.5]{DEKKVH}.
\end{proof}

The questions \cite[Question 6.6]{DEKKVH} and \cite[Question 6.7]{DEKKVH}  highlight the following special cases of \autoref{que:darji}:
\begin{question}[Darji-Elekes-Kalina-Kiss-Vidny\'anszky]\label{que:darjidekkvspeccases}
Is it possible to write the following uncountable Polish groups as the union of a meager and a Haar null set:
\begin{multistmt}
\item $\mathcal{H}(\mathbb{D}^n)$, the group of orientation-preserving self-homeomorphisms of the closed $n$-ball $\mathbb{D}^n$ (where $n\ge 3$ is an integer),
\item $\mathcal{H}(\mathbb{S}^n)$, the group of orientation-preserving self-homeomorphisms of the $n$-sphere $\mathbb{S}^n$ (where $n\ge 2$ is an integer),
\item the homeomorphism group of the Hilbert cube $[0,1]^\omega$.
\end{multistmt}
\end{question}

After these particular cases, we would like to mention an equivalent form of \autoref{que:darji}, which connects it to a different family of problems. The following, seemingly unrelated question appears as \cite[Question 5.5]{EVGdhull} and later in \cite{BanakhII}:
\begin{question}[Elekes-Vidny\'anszky, Banakh]\label{que:banakh}
Is each countable subset of an uncountable Polish group $G$ contained in a $G_\delta$ Haar null subset of $G$?
\end{question}

However, it turns out that this is equivalent to \autoref{que:darji}:

\begin{theorem}\label{thm:darjiqueeqv}
In an uncountable Polish group $G$ the following are equivalent:
\begin{multistmt}
\item Each countable subset of $G$ is contained in a $G_\delta$ Haar null subset of $G$.
\item There exists a countable subset of $G$ that is dense and contained in a $G_\delta$ Haar null subset of $G$.
\item $G$ can be written as the union of a meager and a Haar null set.
\end{multistmt}
\end{theorem}

\begin{proof}

$(1) \Rightarrow (2)$: As $G$ is Polish, there exists a countable dense subset $D$ in $G$.

$(2) \Rightarrow (3)$
Let $D\subset G$ be countable dense set that is contained in a $G_\delta$ Haar null set $N\subset G$. $N$ is dense, because it contains the dense set $D$ as a subset. It is well-known that the complement of a dense $G_\delta$ set is meager; therefore $G=(G\setminus N) \cup N$ is a suitable decomposition.

$(3) \Rightarrow (1)$: Assume that $G=M_0\cup N_0$ where $M_0$ is meager and $N_0$ is Haar null. As $M_0$ is meager, there exists an $F_\sigma$ meager set $M$ such that $M_0 \subseteq M \subset G$ (this follows from the fact that the closure of a nowhere dense set is still nowhere dense). Then the $G_\delta$ set $N = G\setminus M$ is Haar null, because it is a subset of $N_0$.

Let $C\subset G$ be an arbitrary countable set. We will show that there exists a $g\in G$ such that $C\subset gN =\{g\cdot x: x\in N\}$ (this is sufficient, because the translates of a $G_\delta$ Haar null set are also $G_\delta$ Haar null sets). The set of \enquote{bad} translations is
\[\{g\in G: C\nsubseteq gN \} = \bigcup_{c\in C} \{g\in G: c\notin gN \} =\bigcup_{c\in C} \{g\in G: c\cdot g^{-1}\in M\} = \bigcup_{c\in C} Mg\]
which is a meager set (because it is a countable union of meager sets); this clearly shows that there exists a $g\in G$ such that $C\subset gN$. 
\end{proof}

As every Haar meager set is meager (but the converse is not true in general -- see \autoref{ssec:loccptcase}), it is natural to ask the following stronger version of \autoref{que:darji}, which appeared first as \cite[Question 4]{Jab}:
\begin{question}[Jab\l o\'nska]\label{que:jablonska}
Can every uncountable abelian Polish group be written as the union of two sets, one Haar meager and the other Haar null?
\end{question}

We state the following partial answer, its proof can be found as \cite[Theorem 25]{Do}:
\begin{theorem}[Dole\v zal-Rmoutil-Vejnar-Vlas\' ak]\label{thm:hnhmdecon}
Let $G$ be an abelian Polish group such that its identity element has
a local basis consisting of open subgroups. Then $G$ is the union of a Haar null set and a Haar meager set.
\end{theorem}

Recall that the condition of this result has several equivalent formulations:
\begin{theorem}
Let $G$ be a Polish group. Then the following are equivalent:
\begin{multistmt}
\item $G$ is isomorphic to a closed subgroup of $S_\infty$;
\item $G$ admits a countable neighborhood basis at the identity consisting of open subgroups;
\item $G$ admits a countable basis closed under left multiplication;
\item $G$ admits a compatible left-invariant ultrametric.
\end{multistmt}
\end{theorem}
The proof of this classical result can be found as \cite[Theorem 1.5.1]{BK}.

The following result shows that this kind of decomposition can be lifted from a certain kind of subgroup onto the whole group:
\begin{theorem}[Dole\v zal-Rmoutil-Vejnar-Vlas\' ak]
Let $G$ be a Polish group and $H \le G$ be an uncountable closed
subgroup of the center of $G$. Assume that $H$ can be written as the union of a Haar null set in $H$ and a Haar meager set in $H$. Then $G$ is also the union of a Haar null set in $G$ and a Haar meager set in $G$.

In particular, this holds for $G = \mathbb{R}^\omega$ or $G = X$ where $X$ is a Banach space.
\end{theorem}
The proof of this result can be found as \cite[Theorem 22]{Do}.

\section{Common techniques}\label{sec:techniques}
In this section we introduce six techniques, which are frequently useful in practice. The first five of these can be used to show that a set is small, and the last one can be used to show that a set is not small. Note that some of the results from the earlier sections (for example, the basic properties in \autoref{ssec:smallness} or the equivalent definitions in \autoref{ssec:equivvar}) are also very useful in practice.

\subsection{Probes}

Probes are a very basic technique for constructing witness measures. The core of this idea is fairly straightforward and the only surprising thing about probes is the fact that despite their simplicity they are often useful.

Probes were introduced with the following definition in \cite{HSY} (this paper examines the Haar null sets in completely metrizable linear spaces).
\begin{definition}
Suppose that $V$ is an (infinite-dimensional) completely metrizable linear space. A finite dimensional subspace $P\subseteq V$ is called a \emph{probe for} a set $A\subseteq V$ if the Lebesgue measure on $P$ witnesses that $A$ is Haar null.
\end{definition}

\begin{remark}
Strictly speaking, these Lebesgue measures are not probability measures, therefore they cannot be witness measures in the sense of \autoref{def:haarnull}. However, \autoref{thm:witnessmeasequiv} implies that if $B$ is Borel and the Lebesgue measure $\lambda_P$ on a subspace $P$ satisfies that $\lambda_P(B+v)=0$ for every $v\in V$ then, then $B$ is Haar null.
\end{remark}

There is nothing \enquote{magical} about this definition, but it is easy to handle these simple witness measures in the calculations and if there is a probe for a set $A\supseteq V$, then by definition $A$ is Haar null. In arbitrary Polish groups it is easy to generalize this idea and consider a witness measure which is the \enquote{natural} measure supported on a small and well-understood subgroup or subset. If the considered set and the candidate for the probe are not too contrived, then it is often easy to see that it is indeed a probe.

For Haar meager sets the analogue of this is basically proving that the set is strongly Haar meager (see \autoref{ssec:stronghm}) and this is witnessed by a \enquote{naturally chosen} set.

The proof of the following example illustrates the usage of probes.

\def\Mon{M}\begin{example}\label{exa:probeex}
In the Polish group $(\mathcal{C}[0,1],+)$ of continuous real-valued functions on $[0,1]$, the set $\Mon=\{f \in \mathcal{C}[0,1] : f \text{{\upshape{}  is monotone on some interval}}\}$ is Haar null.
\end{example}

\begin{proof}
For a proper interval $I\subseteq [0,1]$ let \[\Mon(I)=\{f\in \mathcal{C}[0,1] : f\text{ is monotone on }I\}.\] As the Haar null sets form a $\sigma$-ideal and \[\Mon=\bigcup\{\Mon([q,r]) : 0\le q <r\le 1\text{ and } q,r \in \mathbb{Q}\},\] it is enough to show that $\Mon(I)$ is Haar null for every proper interval $I\subseteq [0,1]$. It is straightforward to check that $\Mon(I)$ is Borel (in fact, closed).

Fix a function $\varphi \in \mathcal{C}[0,1]$ such that its restriction to $I$ is not of bounded variation. We show that the one-dimensional subspace $\mathbb{R}\varphi=\{c\cdot \varphi : c\in \mathbb{R}\}$ is a probe for $\Mon(I)$, that is, the measure $\mu$ on $\mathcal{C}[0,1]$ that is defined by \[\mu(X) = \lambda(\{c\in \mathbb{R} : c\cdot \varphi\in X\})\]
(where $\lambda$ is the Lebesgue measure on $\mathbb{R}$) is a witness measure for $\Mon(I)$.

We have to prove that $\mu(\Mon(I)+f)=0$ for every $f\in \mathcal{C}[0,1]$. By definition
\[\mu(\Mon(I)+f)=\lambda(\{c\in \mathbb{R} : c\cdot \varphi \in \Mon(I)+f\})\]
and here the set $S_f=\{c\in \mathbb{R} : c\cdot \varphi \in \Mon(I)+f\}$ has at most one element, because if $c_1, c_2\in S_f$, then $c_1\cdot \varphi= m_1 +f$ and $c_2\cdot \varphi= m_2 +f$ for some functions $m_1, m_2\in \Mon(I)$ that are monotone on $I$ and hence $(c_1-c_2)\cdot\varphi= m_1-m_2$ is of bounded variation when restricted to $I$, but this is only possible if $c_1=c_2$. Thus $\lambda(S_f)=\mu(\Mon(I)+f)=0$ and this shows that $\Mon(I)$ is Haar null.
\end{proof}

For additional examples which demonstrate the usage of probes, see the paper \cite{HSY} (which introduces probes and gives several simple examples) or the paper \cite{Hunt} which proves \autoref{exa:wienerex} using a two-dimensional probe. The latter paper is of particular interest because it also proves (at the beginning of section 2) that \autoref{exa:wienerex} cannot be proved using a one-dimensional probe.

\subsection{Application of the Steinhaus theorem}

Sometimes the application of one of the results in \autoref{ssec:steinhaus} can yield very short proofs for the Haar nullness and Haar meagerness of certain sets. Unfortunately, this technique is restricted in the sense that \enquote{good} analogs of the Steinhaus theorem are known only in special classes of groups.

We illustrate this technique by proving \autoref{exa:probeex} again.

\begin{example}\label{exa:probeexagain}
In the Polish group $(\mathcal{C}[0,1],+)$ of continuous real-valued functions on $[0,1]$, the set $\Mon=\{f \in \mathcal{C}[0,1] : f \text{{\upshape{}  is monotone on some interval}}\}$ is Haar null and Haar meager.
\end{example}

\begin{proof}
As we noted in the proof using probes, if $I\subseteq[0,1]$ is a proper interval, then the set
\[\Mon(I)=\{f\in \mathcal{C}[0,1] : f\text{ is monotone on }I\}\]
is Borel and it is enough to see that this set is Haar null and Haar meager for every proper interval $I\subseteq[0,1]$ (we use the fact that the Haar meager sets also form a $\sigma$-ideal).

Assume for contradiction that there exists a proper interval $I\subseteq [0,1]$ such that $\Mon(I)$ is either not Haar null or not Haar meager. If $\Mon(I)$ is not Haar null, then \autoref{cor:SteinhausAbel} implies that $\Mon(I)-\Mon(I)$ is a neighborhood of the constant 0 function. Similarly, if $\Mon(I)$ is not Haar meager, then \autoref{thm:SteinhausJab} implies that $\Mon(I)-\Mon(I)$ is a neighborhood of the constant 0 function.

It is well-known and easy to prove that the difference of two monotone functions is a function of bounded variation; this implies that every $f\in\Mon(I)-\Mon(I)$ satisfies that the restriction $f|_I$ is of bounded variation. However is easy to construct a continuous function $f$ such that $\|f\|_\infty$ is small and $f|_I$ is not of bounded variation; this clearly contradicts the fact that $\Mon(I)-\Mon(I)$ is a neighborhood of the constant 0 function.
\end{proof}

The proof of \autoref{pro:cptsetsmallfromsteinhaus} is another example of this technique.

\subsection{The Wiener measure as witness}

The Wiener measure (which we will denote by $\mu$ in this section) is a well-studied Borel probability measure on the Polish group $(\mathcal{C}[0,1],+)$. There are lots of results which show that something is true for \enquote{most} continuous functions by proving that $\mu(E)=0$ where $E$ is the set of \enquote{exceptional} functions.

Although $\mu(E)=0$ does not necessarily imply that $E$ is Haar null, it is often possible to use the same methods to show that there is a Borel set $B$ such that $E\subseteq B \subseteq \mathcal{C}[0,1]$ and $\mu(B+g)=0$ for every $g\in \mathcal{C}[0,1]$. 

We illustrate this technique on the main result of \cite{Hunt}. This original proof by Hunt relied on constructing a two-dimensional probe. Later, Holick\'y and Zaj\'\i\v{c}ek gave an alternative proof in \cite{HZ} using the Wiener measure as a witness measure; here we reproduce this second proof.

\newcommand{\Deriv}{E}
\begin{example}[Hunt]\label{exa:wienerex}
In the Polish group $(\mathcal{C}[0,1],+)$ of continuous real-valued functions on $[0,1]$, the set \[\Deriv=\{f \in \mathcal{C}[0,1] : f \text{{\upshape{} has a derivative $f'(x)\in \mathbb{R}$ at some point $0\le x\le 1$}}\}\] is Haar null.
\end{example}

\begin{proof}[Proof (Holick\'y-Zaj\'\i\v{c}ek).]
We will prove that the set
\[\Deriv_R=\{f \in \mathcal{C}[0,1] : f \text{{\upshape{} has a finite right derivative at some point $0\le x< 1$}}\}\]
is Haar null; then by symmetry the set
\[\Deriv_L=\{f \in \mathcal{C}[0,1] : f \text{{\upshape{} has a finite left derivative at some point $0< x\le 1$}}\}\]
is also Haar null and this is enough because clearly $\Deriv\subseteq\Deriv_R\cup\Deriv_L$.

For a function $f\in \mathcal{C}[0,1]$, we say that $f$ is Lipschitz from the right at $x$ if
\[\limsup_{y\to {x+}} \left\lvert \frac{f(y)-f(x)}{y-x}\right\rvert <\infty.\]

It is clearly sufficient to prove the following claim:

\def\LipR{B}\begin{claim}\label{pro:wienerprop}
The set
\[\LipR=\{f \in \mathcal{C}[0,1] : f \text{{\upshape{} is Lipschitz from the right at some point $0\le x<1$}}\}\]
has the following properties:
\begin{multistmt}
\item $\Deriv_R \subseteq \LipR$,
\item $\LipR$ is Borel (in fact, $F_\sigma$),
\item $\mu(\LipR +g)=0$ for every $g\in \mathcal{C}[0,1]$ (where $\mu$ is the Wiener measure).
\end{multistmt}
\end{claim}

Here (1) is true because if $f$ has a finite right derivative $d$ at some point $0\le x<1$ then \[\limsup_{y\to x+} \left\lvert\frac{f(x)-f(y)}{x-y}\right\rvert = \lvert d\rvert <\infty.\]

(2) follows from the fact that $\LipR = \bigcup_n E_n$ where
\begin{align*}
E_n = \big\{f\in \mathcal{C}[0,1] :\ &\text{there is a $\textstyle0\le x<1-\frac1n$ such that}\\
&\text{for all $0< h<1-x$, $\lvert f(x+h) -f(x)\rvert \le nh$}\big\}.
\end{align*}
Elementary calculations (which can be found in \cite[Chapter 11]{O}) show that $E_n$ is closed.

(3) is the nontrivial part of this claim. To prove it, first notice that if $f$ is Lipschitz from the right at $x$, then there exists $k\in\mathbb{N}$ and $\delta>0$ such that if $x\le u\le v\le x+\delta$, then
\[\lvert f(v)-f(u)\rvert<k\cdot (v-x).\]
If $n\in\mathbb{N}$ is large enough ($n>\frac{1}{5\delta}$), then there exists an $i\in \{0,1, \ldots, n-1\}$ such that $\frac{i-1}{n}<x\le \frac{i}{n}<\frac{i+1}{n}<\frac{i+2}{n}<\frac{i+3}{n}<x+\delta$, and hence for all $j\in\{0,1,2\}$,
\[\left\lvert f\left(\frac{i+j+1}{n}\right)-f\left(\frac{i+j}{n}\right)\right\rvert <\frac{4k}{n}.\]

If we formalize this observation, then we get
\[\LipR\subseteq\bigcup_{k\in \mathbb{N}}\bigcup_{m\in\mathbb{N}}\bigcap_{n\ge m} \bigcup_{i=0}^{n-1}\bigcap_{j=0}^2\left\{f \in \mathcal{C}[0,1] : \left\lvert f\left(\frac{i+j+1}{n}\right)-f\left(\frac{i+j}{n}\right)\right\rvert <\frac{4k}{n}\right\}.\]

Now we fix an arbitrary $g\in \mathcal{C}[0,1]$ and conclude the proof of the claim by showing that $\mu(\LipR+g)=0$. As $\LipR+g = \{h \in \mathcal{C}[0,1] : h-g \in \LipR\}$,
\[\LipR+g\subseteq\bigcup_{k\in \mathbb{N}}\bigcup_{m\in\mathbb{N}}\bigcap_{n\ge m} \bigcup_{i=0}^{n-1} M_{n, k, i},\]
where
\[M_{n, k, i} = \bigcap_{j=0}^2\left\{h \in \mathcal{C}[0,1] : \left\lvert h\left(\frac{i+j+1}{n}\right)-h\left(\frac{i+j}{n}\right)+\Delta_{i, j,n}\right\rvert <\frac{4k}{n}\right\}\]
and $\Delta_{i, j, n}$ is the difference of two values of $g$ (we will not use its value).

If $h$ is a random function with the Wiener measure as its distribution, then it is well known that any difference of the form $h(t_2)-h(t_1)$ (where $t_2>t_1$) has normal distributions with variance $t_2-t_1$ and moreover, a collection of these differences is independent if they correspond to non-overlapping intervals.

If $X$ is a random variable with normal distribution and variance $\sigma^2$ in an arbitrary probability space $(\Omega, \mathcal{F}, P)$ and $\Delta\in \mathbb{R}$ and $r>0$ are arbitrary values, then clearly
\[P(\lvert X+\Delta\rvert<r) \le \frac{1}{\sqrt{2\pi\sigma^2}}\cdot 2r\]
because $\frac{1}{\sqrt{2\pi\sigma^2}}$ is the maximal value of the density function of $X$.

Using this for the differences of the values of $h$ yields
\[\mu(M_{n,k,i}) \le \left(\frac{1}{\sqrt{2\pi\frac1n}}\cdot 2\cdot \frac{4k}n\right)^3 = n^{-\frac32}\cdot \frac{128\sqrt{2}}{\pi^{\frac32}}\cdot k^3,\]
\[\mu\left(\bigcup_{i=0}^{n-1} M_{n, k, i}\right) \le n^{-\frac12}\cdot \frac{128\sqrt{2}}{\pi^{\frac32}}\cdot k^3 \xrightarrow{n\to \infty} 0\]
and therefore
\[\mu(\LipR+g)\le\mu\left(\bigcup_{k\in \mathbb{N}}\bigcup_{m\in\mathbb{N}}\bigcap_{n\ge m} \bigcup_{i=0}^{n-1} M_{n, k, i}\right)=0,\]
as we claimed.
\end{proof}

Note that \cite{ZajNwD} shows that the set
\[\{f \in \mathcal{C}[0,1] : f \text{{\upshape{} has a derivative $f'(x)\in \mathbb{R}\cup\{-\infty, \infty\}$ at some point $0\le x\le 1$}}\}\]
is not Haar null, but its complement is also not Haar null.

\subsection{Compact sets are small}\label{ssec:cptissmall}

This technique is based on the idea that in non-locally-compact groups the compact sets are \enquote{small} in the sense that they have empty interior. This naturally inspires the following question:

\begin{question}\label{que:cpthaarnull}
Is it true that the compact subsets are Haar null (or Haar meager) in every non-locally-compact Polish group?
\end{question}

While this question is still open, there are several partial results which give positive answers when some additional assumptions are satisfied. It is frequently possible to use these partial results as lemmas. As both the system of Haar null sets and the system of Haar meager sets are $\sigma$-ideals, they also imply that $K_\sigma$ sets (countable unions of compact sets) are Haar null and/or Haar meager in these groups.

The following theorem is \cite[Proposition 12]{Doug}, one of the earliest results in this topic.

\begin{theorem}\label{thm:tsicpthaarnull}
Let $G$ be a non-locally-compact Polish group admitting a two-sided invariant metric. Then every compact subset of $G$ is Haar null.
\end{theorem}

\begin{proof}
We will use \autoref{thm:matresult} to prove this result; this proof is not essentially different from the proof in \cite{Doug}, but separates the ideas specific to compact sets (this proof) and the construction of a limit measure (the proof of \autoref{thm:matresult}).

Fix a two-sided invariant metric $d$ on $G$ and let $C\subseteq G$ be an arbitrary compact subset. We need to prove that for every $\delta>0$ and neighborhood $U$ of $1_G$ there exists a Borel probability measure $\mu$ on $G$ such that the support of $\mu$ is contained in $U$ and $\mu(gCh)<\delta$ for every $g, h\in G$.

Fix $\delta>0$ and a neighborhood $U$ of $1_G$. We may assume that $U$ is open. As $G$ is non-locally-compact, the open set $U$ is not totally bounded, hence there exists an $\varepsilon>0$ such that $U$ cannot be covered by finitely many open balls of radius $2\varepsilon$.

As $C$ is compact, hence totally bounded, there exists a $N\in\omega$ such that $C$ can be covered by $N$ open balls of radius $\varepsilon$. This means that if $X\subseteq C$ and every $x, x'\in X$ satisfies $x\neq x'\Rightarrow d(x,x')\ge 2\varepsilon$, then $\lvert X \rvert \le N$ (because each of the $N$ open balls of radius $\varepsilon$ covering $C$ may contain at most one element of $X$). Using the invariance of $d$ this yields that for every $g, h\in G$ if $X\subseteq gCh$ and every $x, x'\in X$ satisfies $x\neq x'\Rightarrow d(x,x')\ge 2\varepsilon$, then $\lvert X\rvert \le N$.

As $U$ cannot be covered by finitely many open balls of radius $2\varepsilon$, it is possible to choose a sequence $(u_n)_{n\in\omega}$ such that $u_n\in U$ and $u_n\notin \bigcup_{i=0}^{n-1} B(u_{i}, 2\varepsilon)$ for every $n\in\omega$. Choose an integer $M$ that is larger than $\frac{N}{\delta}$ and let $Y = \{u_n : 0\le n< M\}$. Let $\mu$ be the measure on $Y$ which assigns measure $\frac{1}{M}$ to every point in $Y$. If $g, h\in G$ are arbitrary, then $\mu(gCh)=\frac{\lvert gCh\cap Y\rvert}{M}$ and here every $y, y'\in gCh \cap Y$ satisfies $y\neq y'\Rightarrow d(y, y')\ge 2\varepsilon$, and hence $\lvert gCh\cap Y\rvert \le N$, and thus $\mu(gCh)\le \frac{N}{M}< \delta$.
\end{proof}

The following result works in all non-locally-compact Polish groups, but only proves that the compact sets are right Haar null (this is a weaker notion than Haar nullness, see \autoref{ssec:leftrighthaarnull} for the definition and properties).

\begin{proposition}\label{pro:cptsetsmallfromsteinhaus}
Let $G$ be a non-locally-compact Polish group. Then every compact subset of $G$ is right Haar null.
\end{proposition}

\begin{proof}
Suppose that $C\subseteq G$ is compact but not right Haar null. Applying \autoref{thm:SteinhausRHN} yields that there exist a $n\in\omega$ and $h_0, h_1, \ldots, h_{n-1}\in G$ such that
\[h_0 CC^{-1} h_0^{-1}\cup h_1 CC^{-1} h_1^{-1}\cup \ldots \cup h_{n-1} CC^{-1} h_{n-1}^{-1}\]
is a neighborhood of $1_G$. But $CC^{-1}$ is compact (as it is the image of $C\times C$ under the continuous map $(x, y)\mapsto xy^{-1}$), thus its conjugates are also compact, and the union of finitely many compact sets is also compact, and this is a contradiction, because a neighborhood cannot be compact in $G$.
\end{proof}

The paper \cite{DV} investigates the question in the case of Haar meager sets, we state the main results without proofs. This article introduces the \emph{finite translation property} with the following definition:

\begin{definition}
A set $A\subseteq G$ is said to have the \emph{finite translation property} if for every open set $\emptyset\neq U \subseteq G$ there exists a finite set $M\subseteq U$ such that for every $g, h\in G$ we have $gMh\nsubseteq A$.
\end{definition}

The first part of the proof is the following result which allows using this property to prove that a set is strongly Haar meager (this is a stronger notion than Haar meagerness, see \autoref{ssec:stronghm} for the definition and properties). The role of this result is roughly similar to the role of \autoref{thm:matresult} in the case of measure; its proof involves a relatively complex recursive construction. 

\begin{theorem}
If an $F_\sigma$ set $A\subseteq G$ has the finite translation property, then $A$ is strongly Haar meager.
\end{theorem}

The second part is showing that the compact sets have the finite translation property when there is a two-sided invariant metric; the proof is relatively simple and very similar to the one used in \autoref{thm:tsicpthaarnull}.

\begin{theorem}
Let $G$ be a non-locally-compact Polish group admitting a two-sided invariant metric. Then every compact subset of $G$ has the finite translation property.
\end{theorem}

These results yield the analogue of \autoref{thm:tsicpthaarnull}. Note that in the case when $G$ is abelian, it is also possible to prove this as a corollary of \autoref{thm:SteinhausJab}, using the method of the proof of \autoref{pro:cptsetsmallfromsteinhaus}.

\begin{corollary}\label{thm:tsicpthaarmeager}
Let $G$ be a non-locally-compact Polish group admitting a two-sided invariant metric. Then every compact subset of $G$ is (strongly) Haar meager.
\end{corollary}

In addition to these, \cite{DV} also shows that compact sets have the finite translation property in $S_\infty$ (the group of all permutations of a countably infinite set).

We illustrate the usage of this technique with a simple example.
\begin{example}
In the non-locally-compact Polish group $(\mathbb{Z}^\omega, +)$ there are subsets $A, B \subseteq \mathbb{Z}^\omega$ such that they are neither Haar null nor Haar meager, but for every $x\in \mathbb{Z}^\omega$ the intersection $(A+x) \cap B$ is both Haar null and Haar meager.
\end{example}

\begin{proof}
It is well-known that in $\mathbb{Z}^\omega$ a closed set $C$ is compact if and only if
\[C\subseteq \prod_{n\in\omega} \{u_n, u_n+1, u_n+2, \ldots, v_n-1, v_n\}\text{ for some }u, v\in \mathbb{Z}^\omega,\]
as sets of this kind are closed subsets in a product of compact sets and in the other direction if $C\subseteq \mathbb{Z}^\omega$ is compact, then the projections map it into compact subsets of $\mathbb{Z}$.

Let $A=\{a \in \mathbb{Z}^\omega : a_n\le 0\text{ for every }n\in\omega\}$ and $B=\{b \in \mathbb{Z}^\omega : b_n\ge 0\text{ for every }n\in\omega\}$. It is easy to check that these sets satisfy the condition of \autoref{lem:translateall} (the result used in the last technique) and this implies that $A$ and $B$ are neither Haar null nor Haar meager.

On the other hand, the set $(A+x)\cap B=\{z\in \mathbb{Z}^\omega : 0\le z_n \le x_n\}$ is compact, hence \autoref{thm:tsicpthaarnull} and \autoref{thm:tsicpthaarmeager} shows that it is Haar null and Haar meager.
\end{proof}

Note that this phenomenon is impossible in the locally compact case, where non-small sets have density points and if we translate a density point of one set onto a density point of the other, then the intersection will be non-small. On the other hand, \cite[Theorem 4]{MZel} proves that if $(G, +)$ is abelian and non-locally-compact, then there are sets $A, B\subset G$ which are not Haar null, but satisfy that $(A+x) \cap B$ is Haar null for every $x\in G$.

\subsection{Random construction}\label{ssec:random}

This is a technique that is useful when one wants to prove that a not very small set is still small enough to be Haar null. (For example this technique may work for sets that are not Haar meager, because it does not prove Haar meagerness.) The main idea of this technique is that a witness measure for a Haar null set $A\subseteq G$ is a Borel \emph{probability} measure and one can use the language of probability theory (e.g.\ random variables, conditional probabilities, stochastic processes) to construct it and prove that it is indeed a witness measure.

For example, the paper \cite{DM} applies this to prove a result in $S_\infty$, the group of all permutations of the natural numbers (endowed with the topology of pointwise convergence). We illustrate this technique by reproducing the core ideas of this proof. (The proof also contains relatively long calculations which we omit.)

\begin{theorem}[Dougherty-Mycielski]
In the Polish group $S_\infty$ let $X$ be the set of permutations that have infinitely many infinite cycles and finitely many finite cycles. Then the complement of $X$ is Haar null.
\end{theorem}

\begin{proof}[Proof (sketch).]
We will find a Borel probability measure $\mu$ on $S_\infty$ such that $\mu(gXh)=1$ for every $g,h\in S_\infty$. Using that $X$ is conjugacy invariant $\mu(gXh)=\mu(gh(h^{-1}Xh))=\mu(ghX)$ and here $\{gh : g, h\in S_\infty\}=S_\infty=\{g^{-1} : g\in S_\infty\}$, thus it is enough to show that $\mu(g^{-1}X)=1$ for every $g\in S_\infty$.

We define the probability measure $\mu$ by describing a procedure which chooses a random permutation $p$ with distribution $\mu$. (This way we can describe a relatively complicated measure in a way that keeps the calculations manageable.) 

Fix a sequence $k_0< k_1 < k_2 < \ldots$ of natural numbers which are large enough (the actual growth rate is used by the omitted parts of the proof). The procedure will choose values for $p(0)$, $p^{-1}(0)$, $p(1)$, $p^{-1}(1)$, $p(2)$, $p^{-1}(2)$, \ldots\ in this order, skipping those which are already defined (e.g.\ if we choose $p(0)=1$ in the first step, then the step for $p^{-1}(1)$ is omitted, as we already know that $p^{-1}(1)=0$). When we have to choose a value for $p(n)$, we choose randomly a natural number $n'< k_n$ which is still available as an image (that is, $n'\ge n$ and $n'$ is not among the already determined values $p(0)$, $p(1)$, \ldots, $p(n-1)$); we assign equal probabilities to each of these choices. Similarly, when we have to choose a value for $p^{-1}(n)$, we choose randomly a natural number $n'\le k_n$ which is still available as a preimage (that is, $n'> n$ and $n'$ is not among the already determined values $p^{-1}(0)$, $p^{-1}(1)$, \ldots, $p^{-1}(n-1)$); we assign equal probabilities to each of these choices again. When we are finished with these steps, the resulting object $p$ is clearly a well-defined permutation, as every $n\in \omega$ has exactly one image and exactly one preimage assigned to it. We can assume that $k_n$ is large enough to satisfy $k_n> 2n+1$ and this guarantees that we never \enquote{run out} of choices.

We say that $p_0$ is a \emph{possible partial result}, if it can arise after finitely many steps of this process. We omit the relatively long combinatorial arguments which show that the following claim is true:
\begin{claim}
Assume that $p_0$ is a possible partial result, $g\in S_\infty$ is an arbitrary element and $M$ is a natural number. Then there is a natural number $N$ such that the conditional probability with respect to $\mu$, under the condition of extending $p_0$, of the event that the permutation $p$ chosen by our process will be such that $gp$ has no finite cycles including a number greater than $N$ and no two of the numbers $N + 1, \ldots , N + M$ are in the same cycle of $gp$ is at least $\frac{1}{2}$.
\end{claim}
Using this claim, it is possible to show the following claim by induction on $i$ (we also omit this part of the proof):
\begin{claim}
Assume that (as in the previous claim) $p_0$ is a possible partial result, $g\in S_\infty$ is an arbitrary element and $M$ is a natural number. Then for every $i\in\omega$ the conditional probability (with respect to $\mu$, under the condition of extending $p_0$) of the event that the permutation $p$ chosen by our process satisfies that $gp$ has only finitely many finite cycles and at least $M$ infinite cycles is at least $1-2^{-i}$.
\end{claim}
Applying this second claim for every $i\in\omega$ in the special case when $p_0$ is the empty partial permutation yields that the (unconditional) probability (with respect to $\mu$) of the event that the permutation $p$ chosen by our process satisfies that $gp$ has only finitely many finite cycles and at least $M$ infinite cycles is 1. Since this is true for every $M\in\omega$, the permutation $gp$ has infinitely many infinite cycles with $\mu$-probability 1. This shows that $\mu(g^{-1}X)=1$ for the arbitrary permutation $g$, so we are done.
\end{proof}

This set $X$ is the union of countably many conjugacy classes of permutations, one for each finite list of sizes for the finite cycles in the permutation; the paper \cite{DM} also shows that none of these conjugacy classes are Haar null.

\subsection{Sets containing translates of all compact sets}\label{ssec:cpttranslates}

Proving that a set is not Haar null from the definitions requires showing that \emph{all} Borel probability measures fail to witness that it is Haar null, which is frequently harder than just showing \emph{one} measure witnesses that the set is Haar null. The situation is similar for Haar meager sets, where even the choice of the domain of the witness function is not straightforward, although the equivalence $(1) \Leftrightarrow (2)$ in \autoref{thm:witnessdomcantor} can be used to eliminate this extra choice.

Fortunately, in many cases the following simple sufficient condition is enough to show that a set is not Haar null (in fact, not even generalized Haar null) and not Haar meager. This lemma is stated e.g.\ as \cite[Lemma 2.1]{ST}, but it is also common to use this reasoning without stating it separately as a lemma.

\begin{lemma}\label{lem:translateall}
Suppose that a set $A\subseteq G$ satisfies that for every compact set $C\subseteq G$ there are $g, h\in G$ such that $gCh\subseteq A$. Then $A$ is neither generalized Haar null nor Haar meager.
\end{lemma}

\begin{proof}

If $A$ were generalized Haar null, then by \autoref{thm:witnessmeasequiv} there would be a universally measurable set $B\supseteq A$ and a Borel probability measure $\mu$ with compact support $C\subseteq G$ such that $\mu(g'Bh')=0$ for every $g', h'\in G$, but there are $g, h\in G$ such that $gCh\subseteq A\subseteq B$ and thus $\mu(g^{-1}Bh^{-1})\ge \mu(C)=1$, a contradiction.

Similarly, if $A$ would be Haar meager, then there would be a Borel set $B\supseteq A$, a (nonempty) compact metric space $K$ and a continuous function $f : K \to G$ such that $f^{-1}(g'Bh')$ is meager in $K$ for every $g', h'\in G$, but there are $g, h\in G$ such that $gf(K)h\subseteq A\subseteq B$ and thus $f^{-1}(g^{-1}Bh^{-1})\supseteq f^{-1}(f(K))=K$ is not meager in $K$, a contradiction.
\end{proof}

\section{A brief outlook}\label{sec:outlook}

It is clearly beyond the scope of this paper to collect the countless number of results and applications of Haar null sets in various fields of mathematics. However, we now give a highly incomplete list of works using Haar null sets, and encourage the reader to use these as starting points for further reading.

First we mention the paper \cite{OY} which also surveys applications of Haar null sets in addition to presenting their core properties. However, this paper was published in 2005 and therefore it does not contain several recent ideas.

The recent paper \cite{BGJS} introduces and systematically investigates (in the abelian case) the so called Haar-$\mathcal{I}$ sets, a common generalization of the notions of Haar null and Haar meager sets.

One of the original motivations of Christensen was to consider versions of automatic continuity as discussed e.g.\ in \autoref{cor:autocont}, in this topic see the excellent survey paper \cite{Ros}, and also \cite{RosUM} and \cite{FS}.

The Rademacher theorem states that a Lipschitz function between Euclidean spaces is differentiable almost everywhere. The second motivation of Christensen was to extend this result to Banach spaces, see the paper \cite{ChrMTZ} by Christensen, and also \cite{BMfirst,BM,Zaj,BBL,Borw}.

A closely related result is the Alexandrov theorem stating that convex functions on Euclidean spaces are twice differentiable almost everywhere. For extensions of this result to Banach spaces see e.g.\ \cite{MZaj,MatFSN}.

There are various other interesting directions of research in functional analysis involving Haar null sets, see e.g.\ \cite{EMM,Duda,Mat,MatConv}.

When Haar null sets were rediscovered (under the name \enquote{shy sets}) in \cite{HSY}, 
they were the first ones to apply Haar null sets in the theory of dynamical systems. Since then numerous other such applications have been found, see e.g.\ \cite{Avila,BZ,Boun,KH}.

Multifractal analysis is a large area within geometric measure theory. There are numerous interesting results concerning the multifractal analysis of co-Haar null functions and measures, see e.g.\ \cite{BH,BHext,ABD,Olsen}.

There are countless results in geometric measure theory calculating various fractal dimensions of images, graphs, and level sets of certain functions, e.g.\ the generic ones, co-Haar null ones, Brownian motion, etc. A recurring theme is that the generic set is as small as it can, be, while the co-Haar null one is a large as it can be. See e.g.\ \cite{BDEH,Balka,FH,FHnote,GJMNOP,McC,BFFH,BHcf,FF,Kahane}.

Let us also mention a few results that are measure theoretic duals of classical results of real analysis about generic continuous functions: \cite{HZ, Hunt, ZajNwD} concern nowhere differentiable functions, and \cite{BDEB} is the Bruckner-Garg theorem describing the topological structure of the level sets. Interestingly, in the case of \cite{ZajNwD} and \cite{BDEB} the generic behavior happens with \enquote{probability strictly between 0 and 1}, that is, the set of functions exhibiting the behaviour in question is neither Haar null nor co-Haar null.

A large part of modern descriptive set theory deals with homeomorphism groups of compact metric spaces, and automorphism groups of countable first-order structures, such as e.g.\ $\mathcal{H}[0,1]$, the group on increasing homeomorphisms of the unit interval, $S_\infty$, the permutation group of the natural numbers or $\mathrm{Aut}(\mathbb{Q},<)$, the group of increasing bijections of the rational numbers. There have been numerous papers describing the structure of the random element of these groups (that is, the properties which are true for all elements except for a Haar null set), see e.g.\ \cite{DM,CK,ST,DEKKVACS,DEKKVAQ,DEKKVARG,DEKKVH}.

There is also quite some literature dealing with the set theoretic aspects of Haar null sets. The so called cardinal invariants of (various versions of) Haar null sets are calculated in \cite{BanakhI} and \cite{EP}; other related results can be found in e.g.\ \cite{SolND} and \cite{ES}.

Finally, there are also interesting but somewhat sporadic results involving Haar null sets in the theory of differential equations: \cite{EHS,Joly}
and in the theory of functional equations: \cite{Brzd,Gajda,CF,HSY,Rebhuhn}.

\section{List of open questions}\label{sec:questions}

This section collects the open questions which are mentioned in this survey. As this survey is limited to studying fundamental properties and useful techniques, this collection is also limited to open questions from these topics, and we do not mention an enormous number of problems concerning other aspects of Haar null and Haar meager sets. Most of these questions outside the scope of our survey can be found in the references mentioned in \autoref{sec:outlook}.

\begin{restate}{que:ghngthn}[Elekes-Vidny\' anszky \nobracketcite{Question 5.4}{EVGdhull}] Is $\mathcal{GHN}(G) \supsetneqq \mathcal{HN}(G)$ in all non-locally-compact Polish groups?
\end{restate}

\begin{restate}{que:lhneqrhnamenone}[Solecki \nobracketcite{Question 5.1}{SolAmen}]
Let $G$ be an amenable at 1 Polish group. Do left Haar null subsets of $G$ coincide with right Haar null subsets of $G$?
\end{restate}

\begin{restate}{que:shmsigmaideal}[Darji \nobracketcite{Problem 3}{Da}]
Is the system of strongly Haar meager sets a $\sigma$-ideal?
\end{restate}

Even the following variant of the previous question seems to be open:

\begin{restate}{que:shmideal}
Is the system of strongly Haar meager sets an ideal?
\end{restate}

The following question is the main problem of \autoref{ssec:steinhaus}, where we describe lots of partial results (there are groups where the answer is negative, but some weaker variants can be answered positively):
\begin{restate}{que:steinhausgen}
Let $A$ be a Borel (or universally measurable) subset in the Polish group $G$ that is not Haar null (or not generalized Haar null, or not Haar meager). What can we say about the groups $G$ where $1_G \in \Int(AA^{-1})$ is neccessarily satisfied?
\end{restate}

The following question is the main problem of \autoref{ssec:decomposition}, where we describe lots of positive partial results:

\begin{restate}{que:darji}[Darji \cite{DaOP}]\label{re:que:darji}
Can every uncountable Polish group be written as the union of two sets, one meager and the other Haar null?
\end{restate}

In \autoref{thm:darjiqueeqv} we prove that the following question from \cite[Question 5.5]{EVGdhull} and \cite{BanakhII} is equivalent to the previously stated one:
\begin{restate}{que:banakh}[Elekes-Vindny\'anszky, Banakh]
Is each countable subset of an uncountable Polish group $G$ contained in a $G_\delta$ Haar null subset of $G$?
\end{restate}

The questions \cite[Question 6.6]{DEKKVH} and \cite[Question 6.7]{DEKKVH}  highlight the following special cases of \autorefre{que:darji}:
\begin{restate}{que:darjidekkvspeccases}[Darji-Elekes-Kalina-Kiss-Vidny\'anszky]
Is it possible to write the following uncountable Polish groups as the union of a meager and a Haar null set:
\begin{multistmt}
\item $\mathcal{H}(\mathbb{D}^n)$, the group of orientation-preserving self-homeomorphisms of the closed $n$-ball $\mathbb{D}^n$ (where $n\ge 3$ is an integer),
\item $\mathcal{H}(\mathbb{S}^n)$, the group of orientation-preserving self-homeomorphisms of the $n$-sphere $\mathbb{S}^n$ (where $n\ge 2$ is an integer),
\item the homeomorphism group of the Hilbert cube $[0,1]^\omega$?
\end{multistmt}
\end{restate}

As every Haar meager set is meager (but the converse is not true in general -- see \autoref{ssec:loccptcase}) the following question is a stronger variant of \autoref{que:darji}:

\begin{restate}{que:jablonska}[Jab\l o\'nska \nobracketcite{Question 4}{Jab}]
Can every uncountable abelian Polish group be written as the union of two sets, one Haar meager and the other Haar null?
\end{restate}

The paper \cite{Do} answers this question positively in some classes of groups. 

In non-locally-compact groups the compact sets are \enquote{small} in the sense that they have empty interior. This inspires the following question:

\begin{restate}{que:cpthaarnull}
Is it true that the compact subsets are Haar null (or Haar meager) in every non-locally-compact Polish group?
\end{restate}

Several papers contain partial answers to this question, these are summarized in \autoref{ssec:cptissmall}.

Let us conclude this list by asking the following vague question.

\numberwithin{question}{section}

\begin{question}
Can one prove some sort of uniqueness of the ideal of Haar null sets? Somewhat more precisely, can one prove that if $\mathcal{I}$ is a $\sigma$-ideal (say on $\mathbb{Z}^\omega$) that is translation-invariant, has a Borel basis, contains the compact sets, is orthogonal to $\mathcal{M}$, definable in some sense and possesses certain further \enquote{natural} properties then $\mathcal{I} = \mathcal{HN}$?
\end{question}

Interesting results of very similar nature can be found in \cite{FZ, Zak}, but these works build heavily on properties of Lebesgue-null sets not shared by Haar null sets in non-locally-compact groups, e.g. ccc-ness and Fubini properties.

\textsc{Acknowledgments.} We are greatly indebted to Professor J. Mycielski for numerous helpful suggestions. 

\bibliographystyle{sia}

\begin{thebibliography}{99}
\bibitem{BanakhII}T. Banakh, A Polish group containing a Haar null $F_\sigma$-subgroup that cannot be enlarged to a Haar null $G_\delta$-set. Available at \arxivlink{1604.00150v1}.
\end{thebibliography}

\end{document}